\documentclass[letterpaper, reqno]{amsart}

\usepackage[utf8]{inputenc}
\usepackage{amsmath, amssymb, amsthm, mathrsfs}
\usepackage{mathtools}
\usepackage{graphicx}
\usepackage{tikz}
\usepackage{dsfont}
\usepackage{enumerate}
\usepackage{bookmark}
\usepackage[sort]{cite}

\usepackage{cleveref}

\usepackage[final]{showlabels}

\calclayout

\usetikzlibrary{positioning}

\DeclareMathOperator{\supp}{supp}

\DeclareMathOperator{\Ai}{Ai}
\DeclareMathOperator{\sgn}{sgn}
\DeclareMathOperator{\signature}{sign}
\DeclareMathOperator{\Hess}{Hess}

\newtheorem{theorem}{Theorem}[section]

\newtheorem{lemma}[theorem]{Lemma}
\newtheorem{corollary}[theorem]{Corollary}

\newtheorem{rmk}{Remark}

\crefname{cor}{corollary}{corollaries}

\newcommand{\jBra}[1]{\langle #1 \rangle}

\newcommand{\bbZ}{\mathbb{Z}}
\newcommand{\bbR}{\mathbb{R}}

\newcommand{\bbC}{\mathbb{C}}
\newcommand{\bbOne}{\mathds{1}}
\newcommand{\cR}{\mathcal{R}}
\newcommand{\cT}{\mathcal{T}}
\newcommand{\cS}{\mathcal{S}}

\newcommand{\rmI}{\mathrm{I}}
\newcommand{\rmII}{\mathrm{II}}
\newcommand{\rmIII}{\mathrm{III}}
\newcommand{\rmIV}{\mathrm{IV}}

\newcommand{\philin}{\phi_\text{lin}}

\newcommand{\btext}[1]{\left\{\textup{#1}\right\}}
\newcommand{\bnpl}{\btext{non-pseudolocal terms}}
\newcommand{\bbetter}{\btext{better}}
\newcommand{\bsim}{\btext{similar terms}}

\newcommand{\bsime}{\btext{similar or easier terms}}

\newcommand{\nablaes}{\nabla_{\eta,\sigma}}

\crefname{equation}{}{}
\numberwithin{equation}{section}

\title{Long time decay and asymptotics for the complex mKdV equation}
\author{Gavin Stewart}
\address{Courant Institute of Mathematical Sciences, New York University, 251 Mercer St., New York, NY 10012, USA}
\address{Department of Mathematics, Rutgers University, 110 Frelinghuysen Rd., Piscataway, NJ, 08854, USA}
\email{\href{mailto:gavin.stewart@rutgers.edu}{gavin.stewart@rutgers.edu}}

\begin{document}

\begin{abstract}
    We study the asymptotics of the complex modified Korteweg-de Vries equation $\partial_t u + \partial_x^3 u = -|u|^2 \partial_x u$, which can be used to model vortex filament dynamics.  In the real-valued case, it is known that solutions with small, localized initial data exhibit modified scattering for $|x| \geq t^{1/3}$ and behave self-similarly for $|x| \leq t^{1/3}$.  We prove that the same asymptotics hold for complex mKdV.  The major difficulty in the complex case is that the nonlinearity cannot be expressed as a derivative, which makes the low-frequency dynamics harder to control.  To overcome this difficulty, we introduce the decomposition $u = S + w$, where $S$ is a self-similar solution with the same mean as $u$ and $w$ is a remainder that has better decay.  By using the explicit expression for $S$, we are able to get better low-frequency behavior for $u$ than we could from dispersive estimates alone.  An advantage of our method is its robustness: It does not depend on the precise algebraic structure of the equation, and as such can be more readily adapted to other contexts. 
\end{abstract}

\maketitle

\section{Introduction}

We study the complex modified Korteweg-de Vries (mKdV) equation
\begin{equation}\label{eqn:cmkdv}
    \partial_t u + \partial_x^3 u = -|u|^2 \partial_x u
\end{equation}
This equation can be used (after a change of variables) to describe vortex filament dynamics.  Beginning with the higher-order localized induction equation
\begin{equation}\label{eqn:LIE}
    \partial_t X = \partial_s X \times \partial_s^2 X + a \bigl(\partial_s^3 X + \frac{3}{2} \partial_s^2 X \times (\partial_s X \times \partial_s^2 X)
\end{equation}
for the vortex filament $X(s,t)$, Lamb~\cite{lambSolitonsMovingSpace1977} and Fukumoto and Miyazaki~\cite{fukumotoThreedimensionalDistortionsVortex1991} define a new complex function $\psi$ via the Hashimoto transform
\begin{equation*}
    \psi(s, t) = \kappa(s, t) \exp\biggl( i \int_0^s \tau(\sigma,t)\,d\sigma\biggr)
\end{equation*}
where $\kappa$ and $\tau$ are, respectively, the curvature and torsion of the vortex filament (c.f.~\cite{hasimotoSolitonVortexFilament1972}). By rewriting~\eqref{eqn:LIE} using the Serret-Frenet formulas, the authors find that $\psi$ satisfies the Hirota equation
\begin{equation*}
    i \partial_t \psi + \partial_s^2 \psi + \frac{1}{2} |\psi|^2 \psi - i a\bigl( \partial_s^3 \psi + \frac{3}{2} |\psi|^2 \partial_s \psi\bigr)= 0
\end{equation*}
Performing the change of variables
\begin{equation*}
    u(x, t) = \sqrt{\frac{3}{2} a} e^{-\frac{i}{3a}s - \frac{i}{27 a^3}t} \psi\bigl(-x - \frac{1}{3a^2}t, \frac{t}{a}\bigr)
\end{equation*}
we see that $u$ satisfies~\eqref{eqn:cmkdv}.  Equation~\eqref{eqn:cmkdv} also appears as a model in nonlinear optics, where it models higher-order corrections for waves traveling in a nonlinear medium~\cite{rodriguezStandardEmbeddedSolitons2003,heFewcycleOpticalRogue2014,biswasOpticalSolitonsPresence2018}.  It is completely integrable and has infinitely many conserved quantities, the first few of which are momentum, angular twist, and energy~\cite{ancoTravelingWavesConservation2012}:
\begin{equation*}
    \mathcal{P}(u) = \int |u|^2\;dx,\qquad \mathcal{W}(u) = \int |u|^2\arg(u)_x\;dx, \qquad \mathcal{E}(u) = \int \frac{1}{2}|\partial_x u|^2 - \frac{1}{4}|u|^4\;dx
\end{equation*}

\subsection{Known results}

The smoothing and maximal function estimates developed by Kenig, Ponce, and Vega in~\cite{kenigWellposednessScatteringResults1993,kenigGeneralizedKortewegdeVries1989} can be used to show that~\eqref{eqn:cmkdv} is locally wellposed in $H^s$, $s \geq \frac{1}{4}$.  For real initial data, Colliander, Keel, Staffilani, Takaoka, and Tao show in~\cite{collianderSharpGlobalWellposedness2003} that for $s > 1/4$ solutions exist globally, and global wellposedness for real initial data at the $s = 1/4$ endpoint was shown independently by Kishimoto in~\cite{kishimotoWellposednessCauchyProblem2009} and by Guo in~\cite{guoGlobalWellposednessKorteweg2009}.   If we require uniformly continuous dependence on the initial data, then the $s = 1/4$ endpoint is sharp; see~\cite{kenigIllposednessCanonicalDispersive2001} for the focusing case and~\cite{christAsymptoticsFrequencyModulation2003} for the defocusing case.  Local wellposedness has also been show in the weighted Sobolev spaces $H^{s} \cap |x|^{-m}L^2$ for $s \geq \max(2m, 1/4)$, in~\cite{katoCauchyProblemGeneralized1983,fonsecaPersistencePropertiesFractional2015}.  For the equation set on the torus, wellposedness was shown by Chapouto in a range of Fourier-Lebesgue spaces~\cite{chapoutoRemarkWellposednessModified2021,chapoutoRefinedWellPosednessResult2021}.  Relaxing the requirement of uniformly continuous dependence on the initial data, Harrop-Griffiths, Killip, and Visan used the complete integrability of the equation to prove a weaker form of wellposedness in $H^s$ for $s > - 1/2$ in ~\cite{harrop-griffithsSharpWellposednessCubic2020}.  They also show that for $s \leq -1/2$, the equation exhibits instantaneous norm inflation, so no wellposedness result is possible.  Outside the scale of spaces $H^s$, Gr"unrock proved in~\cite{grunrockImprovedLocalWellposedness2004} that the equation is locally wellposed for real-valued initial data in the spaces $\widehat{H^s_r}$ defined by the norms $\lVert u \rVert_{\widehat{H^s_r}} := \lVert \jBra{\xi} \hat{u} \rVert_{L^{r'}}$ for the parameter range $\frac{4}{3} < r \leq 2$, $s \geq \frac{1}{2} - \frac{1}{2r}$.  The parameter range was later improved by Gr\"unrock and Herr in~\cite{grunrockLocalWellposednessModified2009} to $1 < r \leq 2$, $s \geq \frac{1}{2} - \frac{1}{2r}$, which has the scaling-critical space $\widehat{H}^0_1 = \mathcal{F}L^\infty$ as the (excluded) endpoint.  Using an approximation argument, Correia, C\^ote, and Vega were able to show a form of wellposedness in a critical space contained in $\mathcal{F}L^\infty$ in~\cite{correiaSelfSimilarDynamicsModified2020}.

The long-time asymptotics of the real-valued mKdV equation have received a great deal of study.  The first complete results were given in~\cite{deiftSteepestDescentMethod1993}, where Deift and Zhou used the complete integrability of the equations to obtain asymptotic formulas using the inverse scattering transform.  The first results not depending on complete integrability were derived by Hayashi and Naumkin in~\cite{hayashiLargeTimeBehavior1999,hayashiModifiedKortewegVries2001}, where it was shown that solutions starting with small, localized data decay at the linear rate and exhibit {Painlev\'{e}} asymptotics in the self-similar region $|x| \leq t^{1/3}$.  These results were extended to proving modified scattering in the region $x \leq -t^{-1/3}$ and more rapid decay in the region $x \geq t^{-1/3}$ by Harrop-Griffiths in~\cite{harrop-griffithsLongTimeBehavior2016} using the method of testing with wave packets developed by Ifrim and Tataru in~\cite{ifrimGlobalBoundsCubic2015,ifrimTwoDimensionalWater2016}.  These results were also proved independently by Germain, Pusateri, and Rousset in~\cite{germainAsymptoticStabilitySolitons2016} using the method of space-time resonances, and it was further shown that solitons are stable under small, localized perturbations, and that for long times the perturbation has the same asymptotics as in the small data case.  More recently, Correia, C\^ote, and Vega extended these results in~\cite{correiaSelfSimilarDynamicsModified2020} by allowing the solution to have a jump discontinuity at $0$ in Fourier space, which corresponds to studying the dynamics of vortex filaments with corners, see~\cite{perelmanSelfsimilarPlanarCurves2007}.  For the complex equation with nonlinearity $\partial_x(|u|^2 u)$, asymptotic results in the region $x/t < 0$ were given in~\cite{zhangSpectralAnalysisLongtime2022} using the Deift-Zhou steepest descent method.  Results for the Sasa-Satsuma mKdV equation (nonlinearity $2|u|^2 \partial_x u + u \partial_x |u|^2)$ are given for $x/t > 0$ in~\cite{liuLongtimeAsymptoticsSasa2019} and in the self-similar region $|x| \leq t^{-1/3}$ in~\cite{huang_asymptotics_2020}.  Except for results in the region $x/t < 0$ given in~\cite{huangLongtimeAsymptoticHirota2015} using the steepest descent method, the asymptotics of the complex mKdV equation~\eqref{eqn:cmkdv} do not appear to have been studied.

Asymptotic results for the mKdV equation hinge on the decay properties of solutions to the linear Airy\footnote{Note that this should not be confused with the Airy equation $y'' - xy = 0$, which occurs (in a rescaled form) in~\Cref{sec:self-sim}.} equation
\begin{equation*}
    \left\{\begin{array}{l}
        \partial_t u_\text{lin} + \partial_x^3 u_\text{lin} = 0\\
        u_\text{lin}(t=0) = u_0
    \end{array}\right.
\end{equation*}
If $u_0$ is localized and regular, then the asymptotic behavior of $u_\text{lin}$ will be similar to that of the fundamental solution $F(x,t) = (3t)^{-1/3}\Ai((3t)^{-1/3}x)$.  In particular,
\begin{equation}\label{eqn:linear-estimates}
    |u_\text{lin}(x,t)| \lesssim t^{-1/3} \jBra{x/t^{1/3}}^{-1/4},\qquad\qquad |\partial_x u_\text{lin}(x,t)| \lesssim t^{-2/3} \jBra{x/t^{1/3}}^{1/4}
\end{equation}
so $|u_\text{lin} \partial_x u_\text{lin}| \lesssim t^{-1}$, which suggests the problem~\eqref{eqn:cmkdv} (like its real-valued counterpart), should be critical with respect to scattering.

\subsection{Main results}

We will consider the equation with data prescribed at $t=1$:
\begin{equation}\label{eqn:cmkdv-t-1}
    \left\{\begin{array}{c}
        \partial_t u + \partial_x^3 u = -|u|^2 \partial_x u\\
        u(t=1) = e^{-\partial_x^3}u_*
    \end{array}\right.
\end{equation}
We will prove the following result:
\begin{theorem}\label{thm:main-theorem}
    There exists an $\epsilon_0 > 0$ such that if $\epsilon < \epsilon_0$, and $u_* \in H^2$ satisfies
    \begin{equation}\label{eqn:main-thm-initial-bdd}
        \lVert \hat{u}_* \rVert_{L^\infty} + \lVert x u_* \rVert_{L^2} \leq \epsilon
    \end{equation}
    then the solution $u$ to~\eqref{eqn:cmkdv-t-1} exists on $[1,\infty)$ and has the following asymptotics in different regions of physical space:
    \paragraph{$x \geq t^{1/3}$} In this region, we have rapid decay of the form
    \begin{equation}\label{eqn:positive-x-asymp}
        |u(x,t)| \lesssim \epsilon t^{-1/3} (x t^{-1/3} )^{-3/4}
    \end{equation}
    \paragraph{$x \leq -t^{-1/3}$} Here, we have modified scattering
    \begin{equation}\label{eqn:negative-x-asymp}\begin{split}
        u(x,t) =& \frac{1}{\sqrt{12 t \xi_0}}\sum_{\nu \in \{1, -1\}}\exp\left(-2\nu it\xi_0^3 + \nu i\frac{\pi}{4} - i \nu \int_1^t \frac{|\hat{f}(\nu\xi_0,s)|^2}{s}\;ds\right)\hat{f}_\infty(\nu\xi_0)\\
        &\qquad + O(\epsilon t^{-1/3} (x t^{-1/3})^{-9/28})
    \end{split}\end{equation}
    where $\xi_0 = \sqrt{\frac{-x}{3t}}$ and $f_\infty$ is a bounded function of $\xi$. 
    \paragraph{$|x| \leq t^{1/3 + 4\beta}$} Here, we have self-similar behavior given by
    \begin{equation}\label{eqn:small-x-asymptotics}
        u(x,t) = S(x,t;\alpha) + O(\epsilon t^{-1/3-\beta})
    \end{equation}
    where $\alpha$ is some complex number with $|\alpha| \lesssim \epsilon$,  $\beta = \frac{1}{6} - C\epsilon^2$ for some constant $C$, and $S$ is a self-similar solution; that is, $S(x,t;\alpha) = t^{-1/3} \sigma(x/t^{1/3};\alpha)$, where $\sigma$ is a bounded solution of
    \begin{equation}\label{eqn:self-sim-def-intro}
        \left\{\begin{array}{c}
            \sigma'' - \frac{1}{3}x \sigma = -\frac{1}{3}|\sigma|^2\sigma\\
            \hat{\sigma}(0) = \alpha
        \end{array}\right.
    \end{equation}
\end{theorem}
\begin{rmk}
    Equation~\eqref{eqn:self-sim-def-intro} is nothing more than a complex, phase-rotation invariant version of the Painlev\'{e} II equation,
    \begin{equation}\label{eqn:real-painleve-II}
        \left\{\begin{array}{c}
                \tau'' - x \tau = \tau^3\\
                \hat{\tau}(0) = \alpha \in \bbR
            \end{array}\right.
    \end{equation}
    It is known (see~\cite{deiftAsymptoticsPainleveII1995,hastingsBoundaryValueProblem1980,correiaAsymptoticsFourierSpace2020}) that~\eqref{eqn:real-painleve-II} has a unique bounded solution for $|\alpha| < 1$, and this fact will be used in~\Cref{sec:self-sim} to prove that~\eqref{eqn:self-sim-def-intro} has a unique bounded solution.
\end{rmk}
\begin{rmk}\label{rmk:regularity-assumption-main-thm}
    Note that the assumption that $u_* \in H^2$ is only used to prove local wellposedness for the equation (using the $H^2 \cap x^{-1}L^2$ theory from~\cite{katoCauchyProblemGeneralized1983}).  In particular, it plays no role in the a priori estimates that give us the asymptotics, and we do not need any smallness assumption on the $H^2$ norm of $u_*$.
    
    If we could prove local wellposedness for~\eqref{eqn:cmkdv-t-1} with $u_* \in \mathcal{F}L^\infty \cap x^{-1}L^2$, then we could drop the requirement that $u_* \in H^2$ entirely, since our arguments would then imply that the local solution can be extended to a global one.  However, proving local wellposedness in this space is not straightforward: the quasilinear behavior of the problem appears to preclude the use of a fixed-point argument, and smooth functions are not dense in $\mathcal{F}L^\infty$, which makes compactness arguments more complicated.  It is possible that local existence could be proved by arguing along the lines of~\cite{correiaSelfSimilarDynamicsModified2020}; however, nontrivial modifications outside the scope of this paper would be needed to account for the different algebraic structure of the nonlinearity in the complex case.
\end{rmk}

It might seem somewhat unnatural to prescribe initial data in this form.  However, by combining~\Cref{thm:main-theorem} with the weighted local wellposedness result in~\cite{katoCauchyProblemGeneralized1983} we obtain a result with initial conditions given at $t = 0$:
\begin{corollary}\label{thm:main-thm-u-0}
    Let $u$ solve 
    \begin{equation*}
    \left\{\begin{array}{c}
        \partial_t u + \partial_x^3 u = -|u|^2 \partial_x u\\
        u(t=0) = u_0
    \end{array}\right.
\end{equation*}
    Then, there exists an $\epsilon_0 > 0$ such that for all $\epsilon < \epsilon_0$, if 
    \begin{equation*}
        \lVert x u_0 \rVert_{L^2} + \lVert u_0 \rVert_{H^2} \leq \varepsilon
    \end{equation*}
    then the solution $u$ has the same asymptotics as in~\Cref{thm:main-theorem}.
\end{corollary}
\begin{proof}
    By~\cite[Theorem 8.1]{katoCauchyProblemGeneralized1983}, for $\epsilon$ small enough there exists a local solution in $C([0,1], H^2 \cap |x|^{-1} L^2)$ satisfying $\sup_{0 \leq t \leq 1} \lVert u(t) \rVert_{H^2} + \lVert xu(t) \rVert_{L^2} \lesssim \epsilon$.  Now, let $u_* = e^{\partial_x^3} u(1)$.  Since the linear propagator is unitary on $L^2$ Sobolev spaces, $u_* \in H^2$.  Moreover, by using the identity $x e^{t\partial_x^3} = e^{t\partial_x^3} (x - 3t \partial_x^2)$, we see that
    \begin{equation*}\begin{split}
        \lVert \jBra{x} u_* \rVert_{L^2}  \leq& \lVert u_* \rVert_{L^2} + \lVert x e^{\partial_x^3} u(1) \rVert_{L^2}\\
        \lesssim& \lVert u(1) \rVert_{L^2} + \lVert (x - 3 \partial_x^2) u(1) \rVert_{L^2}
        \lesssim \epsilon
    \end{split}\end{equation*}
    By taking the Fourier transform and recalling the one dimensional Sobolev-Morrey embedding $H^1 \to C^{0,1/2}$, we see that $u_*$ satisfies a bound of the form~\eqref{eqn:main-thm-initial-bdd}, so \Cref{thm:main-theorem} gives the result for $\epsilon$ sufficiently small.
\end{proof}
\begin{rmk}
    As discussed in~\Cref{rmk:regularity-assumption-main-thm}, the role of the $H^2$ hypothesis in~\Cref{thm:main-thm-u-0} is largely to allow us to use the weighted local wellposedness theory of~\cite{katoCauchyProblemGeneralized1983}.  In this case, however, it is much less clear that we could obtain a wellposedness theorem in the scaling critical space $\mathcal{F}L^\infty \cap x^{-1}L^2$ because the dispersive decay estimates degenerate at $t = 0$.  Even in the real-valued case, very little is known about wellposedness on $[0,1]$ with initial data in $\mathcal{F}L^\infty \cap x^{-1}L^2$: see~\cite{correiaSelfSimilarDynamicsModified2020}.
\end{rmk}
The methods we develop do not depend on the precise algebraic properties of the nonlinearity or on the complete integrability of the equation.  This robustness is especially important in applications, since~\eqref{eqn:cmkdv} generally arises as a reduced order model for more complicated equations.  In addition, our method can be applied to similar problems with Painlev\'e asymptotics: we use it to get small data asymptotics for the Ablowitz-Ladik equation in~\cite{stewartAblowitzLadik}.  Also see~\Cref{sec:nonlin-modif} for a sketch of how the argument can be modified to apply to other complex mKdV-type equations, including the Sasa-Satsuma equation.

\subsection{Main difficulties}

The main difficulty for complex mKdV over real-valued mKdV is the unfavorable location of the derivative in the nonlinearity, which creates significant obstacles for the proof.

The first difficulty arises when we try to control $u$ in weighted spaces.  Our argument requires us to control $Lu$ in $L^2$, where 
\begin{equation*}
    L = e^{-t\partial_x^3}x e^{t\partial_x^3} = x - 3t\partial_x^2    
\end{equation*}
Prior works for real mKdV estimate $Lu$ by relating it to the scaling transform of $u$:
\begin{equation*}
	\Lambda u = (1 + x\partial_x + 3t\partial_t)
\end{equation*}
via the identity 
\begin{equation*}
	Lu = \partial_x^{-1} \Lambda u - 3tu^3
\end{equation*}
which holds for solutions of real mKdV.  Since the nonlinearity for real mKdV can be written as a derivative, we can integrate the equation for $\Lambda u$ and perform an energy estimate to get control of $\partial_x^{-1} \Lambda u$.  This strategy fails completely for complex mKdV: the relationship between $Lu$ and $\Lambda u$ now reads
\begin{equation*}
	Lu = \partial_x^{-1} \Lambda u - 3t \partial_x^{-1} \left(|u|^2 \partial_x u\right)
\end{equation*}
and the last term cannot be bounded in $L^2$.  Thus, we must use a different approach.

Our argument, very roughly speaking, amounts to performing an energy estimate on $Lu$ directly.  The fact that we have a derivative in the nonlinearity means that we must exploit some cancellation (via integration by parts) to avoid having to estimate terms containing a $\partial_x Lu$ factor.  This is analogous to the situation for the $H^k$ energy estimates for mKdV: see~\cite[Chapter 4]{taoNonlinearDispersiveEquations2006}.

The other (and more major) problem appears when we try to establish bounds on $\hat{u}$ at low frequencies.  Writing $f = e^{t\partial_x^3} u$ for the linear profile of $u$, we see that
\begin{equation*}
	\partial_t \hat f(\xi, t) = -\frac{i}{2\pi} \int_1^t\int e^{it\phi} \hat{f}(\eta, t) \overline{\hat{f}(-\sigma, t)} (\xi - \eta - \sigma)\hat{f}(\xi - \eta - \sigma, t)\;d\eta d\sigma
\end{equation*}
for $\phi(\xi, \eta, \sigma) = \xi^3 - (\xi-\eta-\sigma)^3 - \eta^3 - \sigma^3 = 3(\eta + \sigma)(\xi - \eta)(\xi - \sigma)$.  When $|\xi| \geq t^{-1/3}$, we can perform a stationary phase estimate on the nonlinear term to find that $\partial_t \hat{f}$ satisfies a perturbed Hamiltonian ODE, which we can integrate to get bounds for $\hat{f}$.  However, the stationary phase estimate degenerates at low frequencies, which prevents us from applying this argument for $|\xi| < t^{-1/3}$.  In the real-valued case, the failure of the stationary phase estimates is compensated by the favorable position of the derivative, allowing us to obtain global bounds using only dispersive estimates.  In the complex-valued case, the derivative structure is less favorable, and the dispersive estimates for $u$ are only strong enough to prove $|\partial_t \hat{f}| \lesssim \epsilon^3 t^{-1}$, which is insufficient to prove global bounds.

To overcome this second difficulty, we perform a modulation argument.  We write $u = S + w$, where $S$ is a self-similar solution to~\eqref{eqn:cmkdv} satisfying $\hat{S}(0) = \hat{u}(0)$ and $w$ is a remainder.  This decomposition is advantageous: the fact that $S$ is self-similar implies that $|S|^2 \partial_x S$ is a derivative, which gives additional cancelation at low frequencies.  Moreover, the Fourier space estimates of Correia, C\^ote, and Vega in~\cite{correiaAsymptoticsFourierSpace2020} can be combined with our estimates on the linear propagator to show that $S$ obeys the same decay estimates as $u$.  In particular, $S$ and $u$ have matched asymptotics for $|x| \leq t^{1/3}$ (which corresponds to $|\xi| \leq t^{-1/3}$ in frequency space), so low-frequency projections of $w$ obey stronger decay bounds than either $u$ or $S$.  By writing
\begin{equation*}\begin{split}
	\partial_t \hat{f}(0, t) =& -\frac{1}{\sqrt{2\pi}} \int |u|^2 \partial_x u \;dx\\ 
	=& -\frac{1}{\sqrt{2\pi}} \int |u|^2 \partial_x u - |S|^2 \partial_x S\;dx\\
	=&  -\frac{1}{\sqrt{2\pi}}\int |u|^2 \partial_x w + 2\Re(u \overline{w}) \partial_x S\;dx
\end{split}\end{equation*}
and using this improved decay, we can prove that $\partial_t \hat{f}(0,t)$ decays at an integrable rate, which allows us to show that $\hat{f}(\xi,t)$ is bounded for $|\xi| \lesssim t^{-1/3}$.  As a bonus, we immediately get the asymptotics $u \approx S$ for $|x| \lesssim t^{1/3}$.  

This argument has some similarities with the one used by Hayashi and Naumkin in~\cite{hayashiModifiedKortewegVries2001}: indeed, the estimates we find for $w$ are largely identical to theirs.  However, our argument differs from theirs in three key respects.  First, we cannot estimate $Lw$ using the scaling vector field $\Lambda$, so we instead must use the method of space-time resonances to perform an energy estimate on $Lw$ directly.  Second, since the mean $\hat{u}(0, t)$ varies in time, we must modulate in time rather than subtracting a fixed self-similar solution.  Finally, our argument is different in terms of how the estimates on $w$ fit into the proof.  In~\cite{hayashiModifiedKortewegVries2001}, the estimates on $w$ are performed after the solution has been shown to decay at the linear rate for all time, and are only necessary to obtain the asymptotics in the self-similar region.  In our work, on the other hand, the estimates on $w$ are necessary in order to prove that the solution $u$ decays at the linear rate globally in time.

\subsection{Plan of the proof}

\subsubsection{Overview of the space-time resonance method}
To prove~\Cref{thm:main-theorem}, we will work within the framework of the method of space-time resonances.  This method, first developed by Germain, Shatah, and Masmoudi in~\cite{germainGlobalSolutions3D2008} and independently by Gustafson, Nakanishi, and Tsai in~\cite{gustafsonScatteringTheoryGross2009}, has been used to derive improved time of existence and asymptotics for a variety of equations with dispersive character; see~\cite{germainGlobalExistenceCoupled2011, germainGlobalExistenceEulerMaxwell2014,ionescuGlobalRegularity2d2018,ionescuEinsteinKleinGordonCoupledSystem2020,cordobaGlobalSolutionsGeneralized2019,katoNewProofLongrange2011,haniScatteringZakharovSystem2013,germainGlobalSolutions3D2008,germainGlobalSolution2D2012}.  The method begins by rewriting the nonlinear equation for $u$ in terms of the profile $f(t) = e^{t\partial_x^3} u(t)$:
\begin{equation}\label{eqn:cmkdv-profile-eqn}
    \left\{\begin{array}{l}
        \partial_t f =  -e^{t\partial_x^3}\left(\left|e^{-t\partial_x^3} f \right|^2 \partial_x e^{-t\partial_x^3} f\right)\\
        f(t=1) = u_*
    \end{array}\right.
\end{equation}
This can be re-written in mild form as
\begin{equation}\label{eqn:cmkdv-profile-eqn-mild}
    \hat f(\xi,t) = \hat{u}_*(\xi) - \frac{i}{2\pi} \int_1^t\iint e^{is\phi} \hat{f}(\eta,t) \overline{\hat{f}(-\sigma,t)} (\xi - \eta - \sigma)\hat{f}(\xi - \eta - \sigma,t)\,d\eta d\sigma ds
\end{equation}
where $\phi$ is the phase associated with the four wave mixing by the cubic nonlinearity:
\begin{equation}\label{eqn:4-wave-phase}
    \phi(\xi, \eta, \sigma) = \xi^3 - (\xi-\eta-\sigma)^3 - \eta^3 - \sigma^3 = 3(\eta + \sigma)(\xi - \eta)(\xi - \sigma)
\end{equation}
Roughly speaking, we would like to show that the change in $f$ (given by the integral in~\eqref{eqn:cmkdv-profile-eqn-mild}) is small in some norm that gives us the required decay estimates for $u$.  Heuristically, if we imagine that $\hat{f}$ is a smooth bump function, then the integral term in~\eqref{eqn:cmkdv-profile-eqn-mild} will be dominated by the stationary points of the phase where $\nabla_{s, \eta,\sigma} (s \phi(\xi,\eta,\sigma)) = 0$.  The points where $\phi = 0$ corresponds to a resonance (in the classical sense of the term) in the nonlinear interaction between plane waves of frequencies $\xi-\eta-\sigma, \eta$ and $-\sigma$.  

For dispersive PDEs, it is more natural to think in terms of \emph{wave packets} instead of plane waves.  A wave packet at frequency $\xi$ is a bump function that travels at the group velocity, which for complex mKdV is $v_\xi = -3\xi^2$.  Clearly, wave packets can interact over large timescales only if they have the same group velocity, and the condition $\nablaes \phi = 0$ is precisely what is required for three wave packets at frequencies $\xi-\eta-\sigma$, $\eta$, and $-\sigma$ to have the same group velocities.  See~\cite{germainSpacetimeResonances2010} for an expository overview of the method.

In our application, we will see in~\Cref{sec:linear-ests} the decay we want in~\Cref{thm:main-theorem} follows from the estimate $\sup_{t \geq 1} \lVert f \rVert_X \lesssim \epsilon$, where the $X$ norm is defined by
\begin{equation}\label{eqn:X-def}
    \lVert f \rVert_{X} =  \lVert \hat f(t) \rVert_{L^\infty} + t^{-1/6}\lVert  xf(t) \rVert_{L^2}
\end{equation}
Note that this norm is scale invariant.  We will use a bootstrap argument to show that this norm is small for all time (see~\Cref{sec:reduction} for details).

\subsubsection{Step 1: Stationary phase estimate for high frequencies}\label{sec:intro-step1}
Let us first consider the $L^\infty$ bound for $\hat{f}(\xi,t)$.  The analysis here is similar to the one conducted in~\cite[Section 1.3]{germainAsymptoticStabilitySolitons2016}.  Since we are interested in obtaining a (uniform) pointwise bound, it is natural to consider $\xi$ fixed.  The stationary points $\nablaes \phi = 0$ are then given by
\begin{align*}
    (\eta_1,\sigma_1) =& (\xi,\xi)\\
    (\eta_2,\sigma_2) =& (\xi,-\xi)\\
    (\eta_3,\sigma_3) =& (-\xi,\xi)\\
    (\eta_4,\sigma_4) =& (\xi/3,\xi/3)
\end{align*}
A formal stationary phase calculation then shows that
\begin{equation*}
    \partial_t \hat f(\xi, t) = -\frac{i\sgn \xi}{6t} | \hat{f}(\xi, t)|^2 \hat{f}(\xi, t) + ce^{it\frac{8}{9}\xi^3}\frac{\sgn \xi}{t} | \hat{f}(\xi/3, t)|^2 \hat{f}(-\xi/3, t) + \{\text{error}\}
\end{equation*}
where $c$ is some constant whose exact value is unimportant.  Since the second term has a highly oscillatory phase, we expect that it will not be relevant on timescales $t \geq |\xi|^{-3}$.  Similarly, we expect that the error term will be higher order in $t$, and hence will not contribute significantly to the asymptotics.  After discarding the oscillatory term and the error term, we are left with a Hamiltonian ODE, which we can integrate explicitly to find that
\begin{equation*}
    \hat f(\xi, t) \approx \exp\left(-\frac{i}{6}\sgn \xi \int_1^t \frac{|\hat{f}(\xi,s)|^2}{s}\;ds\right) f_\infty(\xi)
\end{equation*}
for some bounded function $f_\infty$ (cf.~\cite{katoNewProofLongrange2011}).  We will formalize this argument for $|\xi| \geq t^{-1/3}$ in~\Cref{sec:st-phase-setup,sec:st-phase-sec}.

\subsubsection{Step 2: Modulation analysis in the self-similar region}\label{sec:intro-step2}
The above argument only applies for frequencies $|\xi| \geq t^{-1/3}$.  For smaller frequencies, there is not enough oscillation to neglect the oscillating term, and the error term in the stationary phase expansion becomes unacceptably large due to the coalescence of the stationary points $(\eta_i, \sigma_i)$ as $\xi \to 0$.  Using the embedding $\dot{H}^1 \to C^{0,1/2}$ and making the bootstrap assumption that $\lVert \partial_\xi \hat{f} \rVert_{L^2} = \lVert xf \rVert_{L^2} \lesssim \epsilon t^{1/6}$, we see that the problem of controlling low frequencies reduces to understanding the behavior of the zero Fourier mode.  In the real-valued case, $\hat{f}(0,t)$ is conserved by the flow (and hence the low frequency bounds are immediate), but in the complex-valued case,
\begin{equation*}
    \partial_t \hat{f}(0, t) = \partial_t \hat{u}(0, t) = -\frac{1}{\sqrt{2\pi}}\int |u|^2 \partial_x u\;dx
\end{equation*}
which is not zero in general.

The main difficulty for $|\xi| \leq t^{-1/3}$ is that the low-frequency component of $u$ evolves in a genuinely nonlinear manner.  By analogy with the real-valued problem, we expect $u$ to exhibit self-similar asymptotics for $|x| \leq t^{1/3}$, which corresponds to the low frequency range $|\xi| \leq t^{-1/3}$.  Thus, we will attempt to show that the behavior of $u$ at low frequencies is approximately self-similar.  If $S(x,t) = t^{-1/3} \sigma(x t^{-1/3})$ is a self-similar solution of~\eqref{eqn:cmkdv}, then $\sigma$ satisfies the third-order ODE
\begin{equation}\label{eqn:sigma-third-order-ODE}
    \partial_x^3 \sigma - \frac{1}{3}\partial_x(x\sigma) = -|\sigma|^2 \sigma_x
\end{equation}
In order for $S$ to be compatible with the asymptotics given in~\Cref{thm:main-theorem}, we need $\sigma$ to be bounded and lie in a certain weighted $L^2$ space.  In particular, this means $\sigma$ must solve the Painlev\'e-type equation by~\eqref{eqn:self-sim-def-intro}.  Since the mean of $u$ changes in time, we will also need to modulate the mean of the self-similar solution.  This leads us to impose the condition $\hat{\sigma}(0) = p$, which by~\cite{correiaAsymptoticsFourierSpace2020} is enough to determine $\sigma$.  By examining~\eqref{eqn:sigma-third-order-ODE} and recalling that $S$ is obtained by applying the self-similar scaling to $\sigma$, we see that $|S|^2\partial_x S$ is a derivative and thus has mean zero.  If we choose $S$ such that $\hat{S}(0,t) = \hat{\sigma}(0) = \hat{f}(0,t)$, then this implies that
\begin{equation*}
    \partial_t \hat{f}(0, t) = -\frac{1}{\sqrt{2\pi}} \int |u|^2 \partial_x u - |S|^2 \partial_x S\;dx =  -\frac{1}{\sqrt{2\pi}} \int |u|^2 \partial_x w + (w\overline{w} + w \overline{u}) \partial_x S\;dx
\end{equation*}
where $w = u - S$.

\subsubsection{Step 3: Weighted bounds for \texorpdfstring{$w$}{w}}\label{sec:intro-step3}
We now consider the difference $w$ in more detail.  By definition, $w$ has mean zero and so does $g = e^{t\partial_x^3} w$.  We will show that $\lVert xg \rVert_{L^2} \lesssim \epsilon t^{1/6 - \beta}$ for $\beta$ as in~\Cref{thm:main-theorem}.  (This is similar to \cite{hayashiModifiedKortewegVries2001}, where $Lw$ is shown to obey better $L^2$ bounds than $Lu$).  Using this estimate, we find that $w$ has better dispersive decay than $u$, which allows us to prove in~\Cref{sec:low-freq-bdds} that
\begin{equation*}
    |\partial_t \hat{f}(0, t)| \lesssim \epsilon^3 t^{-1-\beta}
\end{equation*}
Integrating in time gives the boundedness for low frequencies, and shows that $u(x,t) \approx S(x,t;\alpha)$ for $|x| \lesssim t^{1/3}$ and $t$ large, where $\alpha = \lim_{t\to\infty} \hat{f}(0,t)$.  Moreover, using the self-similar scaling and~\eqref{eqn:self-sim-def-intro}, it can be seen that $S$ satisfies $\lVert LS \rVert_{L^2} \sim \epsilon^3 t^{1/6}$.  Thus,
\begin{equation*}
    \lVert xf \rVert_{L^2} \leq \lVert LS \rVert_{L^2} + \lVert xg \rVert_{L^2}
\end{equation*}
so the bound $\lVert xf \rVert_{L^2} \lesssim \epsilon t^{1/6}$ also follows from the improved $L^2$ bound for $xg$. Thus, the proof is complete if we can obtain the weighted $L^2$ bound for $g$.  This is the most technical and involved part of the argument, and is tackled in~\Cref{sec:weighted-L2}.  We will sketch the argument below.

In the real-valued case, it is possible to use the scaling vector field to control $Lw$ in $L^2$, but the non-divergence form of the nonlinearity in~\eqref{eqn:cmkdv} precludes this argument.  Instead, we will use a more direct approach.  Since $\mathcal{F}(xg) = -i \partial_\xi \hat{g}$, by Plancherel's theorem it suffices to prove bounds on $\partial_\xi \hat{g}$ in $L^2$.  We find that
\begin{equation}\label{eqn:cmkdv-localization-g-eqn}\begin{split}
    \partial_t \partial_\xi \hat g =&  \frac{s}{2\pi} \int \partial_\xi \phi e^{is\phi} \hat{f}(\eta) (\xi-\eta-\sigma)\hat{g}(\xi-\eta-\sigma) \overline{\hat{f}(-\sigma)}\;d\eta d\sigma  \\&\qquad - \frac{i}{2\pi} \int e^{is\phi} \hat{f}(\eta) (\xi - \eta - \sigma) \partial_\xi \hat{g}(\xi-\eta-\sigma) \overline{\hat{f}(-\sigma)}\;d\eta d\sigma\\
    &+\{\text{easier terms}\}
\end{split}\end{equation}
The second term appears concerning, since our estimates do not allow us to control derivatives of $Lw$.  However, by writing it as $e^{-t\partial_x^3} \left(|u|^2 \partial_x Lw\right)$ in physical space and performing an energy estimate, we see that this term is actually harmless:
\begin{equation*}
    \frac{1}{2}\partial_t \lVert xg \rVert_{L^2}^2 = -\int |u|^2 \partial_x \left| Lw \right|^2\;dx + \cdots =  \int \partial_x \left| u \right|^2 \left| Lw \right|^2\;dx + \cdots
\end{equation*}
and $\int \partial_x |u|^2 |Lw|^2\;dx \lesssim \epsilon^2 t^{-1} \lVert xg \rVert_{L^2}^2$, which is consistent with the slow growth of $xg$ in $L^2$.  Thus, it only remains to control the first term in~\eqref{eqn:cmkdv-localization-g-eqn}.  We do this by considering the space-time resonance structure of $\phi$, together with cancellations coming from the $\partial_\xi \phi$ multiplier.  The derivatives of $\phi$ are
\begin{equation}\label{eqn:phi-derivatives}\begin{split}
    \partial_\xi \phi =& 3(\eta + \sigma)(2\xi - \eta - \sigma)\\
    \partial_\eta \phi =& 3(\xi - \sigma)(\xi - 2\eta - \sigma)\\
    \partial_\sigma \phi =& 3(\xi - \eta) (\xi - \eta -2\sigma)\
\end{split}\end{equation}
Based on~\cref{eqn:4-wave-phase,eqn:phi-derivatives}, we introduce the space-time resonant sets
\begin{align*}
    \mathcal{T} =& \{\xi = \eta\} \cup \{\xi = \sigma\}\\
    \mathcal{S} =& \{\eta = \sigma = \xi/3\}\\
    \mathcal{R} =& \mathcal{S} \cap \mathcal{T} \\
    =& \{(0,0,0)\}
\end{align*}
where $\mathcal{T}$ is the set of time resonances (where $\phi$ vanishes to higher order than $\partial_\xi \phi$), $\mathcal{S}$ is the set of space resonances (where $\nablaes \phi$ vanishes to higher order than $\partial_\xi \phi$), and the set of space-time resonances is given by $\mathcal{R} = \mathcal{T} \cap \mathcal{S}$.  

Away from the set $\mathcal{T}$ of time resonances, the quotient $\frac{\partial_\xi \phi}{\phi}$ is bounded, and we may integrate by parts in the time variable.  This is similar to the normal form transformation method introduced by Shatah in~\cite{shatahNormalFormsQuadratic1985}. In particular, it transforms the cubic nonlinearity into a quintic nonlinearity, which gives us more decay and leads to better bounds.  We will pursue this idea in~\Cref{sec:xf-tr}.

On the other hand, outside the set $\mathcal{S}$ of space resonances, $\frac{\partial_\xi \phi}{|\nablaes \phi|}$ is bounded, and we can integrate by parts using the relation $\frac{\nablaes\phi}{is|\nablaes\phi|^2} \cdot \nablaes e^{is\phi} = e^{is\phi}$ to gain a power of $s^{-1}$.  This is similar in spirit to the vector field method developed by Klainerman in~\cite{klainermanNullConditionGlobal1986}.  In principle, this integration could result in a loss of derivatives when the space weight and the derivative fall on the same term.  In practice, however, we only need to apply this integration by parts in a small neighborhood of $\mathcal{T}$, where it can be seen that $|\xi - \eta - \sigma| \lesssim \max (|\eta|, |\sigma|)$, where we can move the derivative from the term with an $L$ weight to an unweighted term.  The details of this argument are given in~\Cref{sec:xf-sr}

Finally, in a small (time-dependent) neighborhood of the space-time resonant set $\mathcal{R}$, we can integrate crudely using the volume bounds in Fourier space and the H\"older bound $|\hat{g}(\xi)| \lesssim \epsilon t^{1/6-\beta}|\xi|^{1/2}$ together with the $L^\infty$ bound for $\hat{f}$ to bound the contribution from $\mathcal{R}$ in~\Cref{sec:xf-str}.

\subsection{On the structure of the nonlinearity}\label{sec:nonlin-modif} Although the arguments in this paper are given for the focusing nonlinearity $-|u|^2 \partial_x u$, with slight modifications they can be applied to any nonlinearity of the form $a |u|^2 \partial_x u + b u^2 \partial_x \overline{u}$ for $a$ and $b$ real.  Taking $a = 1$ gives the defocusing equation.  For $a = 3b$, we have the Sasa-Satsuma equation, whose asymptotics were studied in~\cite{liu_long-time_2019,huang_asymptotics_2020}.  We now sketch the changes necessary to handle this more general class of nonlinearities.

With the new nonlinearity,~\eqref{eqn:cmkdv-profile-eqn-mild} must be modified:
\begin{equation}\begin{split}
    \hat f(\xi,t) &= \hat{u}_*(\xi) + a\frac{i}{2\pi} \int_1^t\iint e^{is\phi} \hat{f}(\eta,t) \overline{\hat{f}(-\sigma,t)} (\xi - \eta - \sigma)\hat{f}(\xi - \eta - \sigma,t)\,d\eta d\sigma ds\\
    &+ b\frac{i}{2\pi} \int_1^t\iint e^{is\phi} \hat{f}(\eta,t) \hat{f}(\sigma,t) (\xi - \eta - \sigma)\overline{\hat{f}(-\xi + \eta + \sigma,t)}\,d\eta d\sigma ds
\end{split}\end{equation}
Notice that the second integral still has the same phase $\phi$ given in~\eqref{eqn:4-wave-phase}, so the stationary points given in~\Cref{sec:intro-step1} remain unchanged.  The stationary phase expansion of $\partial_t \hat{f}(\xi,t)$ now contains terms involving $|\hat{f}(-\xi,t)|^2$:
\begin{equation*}\begin{split}
    \partial_t \hat f(\xi, t) =& \frac{ia\sgn \xi}{6t} | \hat{f}(\xi, t)|^2 \hat{f}(\xi, t) + cae^{it\frac{8}{9}\xi^3}\frac{\sgn \xi}{t} | \hat{f}(\xi/3, t)|^2 \hat{f}(-\xi/3, t)\\
    &- \frac{ib \sgn \xi}{6t} |\hat{f}(\xi, t)|^2\hat{f}(\xi, t) + \frac{2ib \sgn \xi}{6t} |\hat{f}(-\xi,t)|^2 \hat{f}(\xi, t) \\
    &+ cae^{it\frac{8}{9}\xi^3}\frac{\sgn \xi}{t} \hat{f}(\xi/3, t)^2 \overline{\hat{f}(-\xi/3, t)} + \{\text{error}\}
\end{split}\end{equation*}
If we drop the oscillating terms and error, we obtain a coupled ODE system for $\hat{f}(\xi)$ and $\hat{f}(-\xi)$:
\begin{equation*}\begin{split}
    \partial_t \hat{f}(\xi,t) =& \frac{i}{6t} ((a-b) |\hat{f}(\xi,t)|^2 + 2b |\hat{f}(-\xi,t)|^2) \hat{f}(\xi,t)\\
    \partial_t \hat{f}(-\xi,t) =& \frac{i}{6t} (2b |\hat{f}(\xi,t)|^2 + (a-b) |\hat{f}(-\xi,t)|^2) \hat{f}(-\xi,t)
\end{split}\end{equation*}
In the above system, $|\hat{f}(\pm\xi)|^2$ are conserved, so we can explicitly integrate to get the logarithmic phase correction (which now depends on $\hat{f(-\xi)}$ as well as $\hat{f}(\xi)$).  The self-similar profile $\sigma$ discussed in~\Cref{sec:intro-step2} now satisfies
\begin{equation*}
    \partial_x^3\sigma - \frac{1}{3}\partial_x(x \sigma) = a|\sigma|^2 \partial_x\sigma + b \sigma^2 \partial_x\overline{\sigma}
\end{equation*}
As before, we can conclude that for self-similar solution $S$, the expression  $a|S|^2 \partial_x S + b S^2 \partial_x \overline{S}$ is a total derivative, which allows us to get the improved decay for $\partial_t \hat{f}(0,t)$ using the better decay of $w$.  Since the equation satisfied by $\sigma$ is still gauge invariant, we can use the arguments in~\Cref{sec:self-sim} reduce the problem to studying the gauge-invariant Painlev\'e II equation
\begin{equation*}
    \left\{\begin{array}{c}
        \sigma''(y;p) = \frac{1}{3}y \sigma(y;p) + \frac{1}{3}(a+b) |\sigma(y;p)|^2 \sigma(y;p)\\
        \hat{\sigma}(0;p) = p\\
        \hat{\sigma}(\eta;p) \text{ is continuous at } \eta=0
    \end{array}\right.
\end{equation*}
Finally, we note that the key energy estimate from~\Cref{sec:intro-step3} still holds for the new form of the nonlinearity:
\begin{equation*}\begin{split}
    \frac{1}{2}\partial_t\lVert xg \rVert_{L^2}^2 =& a \int |u|^2 \partial_x |Lw|^2\;dx + b \Re \int u^2 \partial_x (\overline{Lw})^2\;dx + \cdots\\
    =& -a \int \partial_x|u|^2  |Lw|^2\;dx - b \Re \int \partial_x u^2  (\overline{Lw})^2\;dx + \cdots\\
\end{split}\end{equation*}
The remaining space-time resonant estimates are insensitive to the exact conjugate structure of the nonlinearity, so the rest of the argument goes through with no changes.

\subsection{Organization of the paper}  The plan of the rest of the paper is as follows: In~\Cref{sec:notation}, we present some notation that will be used throughout the paper and give some conventions and results on pseudoproduct operators.  In~\Cref{sec:linear-ests}, we give decay estimates for the linear equation, and show how these estimates allow us to control bi- and trilinear terms.  In~\Cref{sec:self-sim}, we will consider the self-similar solution $S$, and derive estimates which will be necessary for the later analysis.   In~\Cref{sec:reduction}, we will show that~\Cref{thm:main-theorem} follows if we prove that $\hat{f}$ evolves by logarithmic phase rotation, and that $\lVert xg \rVert_{L^2}$, $\lVert \hat{f} \rVert_{L^\infty}$, and $|\partial_t \hat{f}(0,t)|$ satisfy certain bootstrap estimates.  In~\Cref{sec:weighted-L2}, we prove that under the bootstrap assumptions, $\lVert xg \rVert_{L^2} \lesssim \epsilon t^{1/6-\beta}$ using space-time resonances and a Gr\"onwall argument.  We then verify that $|\partial_t \hat{u}|$ has the required decay (and hence that $\hat{f}$ is bounded for low frequencies) and show that at high frequencies $\hat{f}$ is bounded and undergoes the required logarithmic phase rotation in~\Cref{sec:L-infty-est}.

\section{Preliminaries}\label{sec:notation}
\subsection{Notation and basic inequalities}
We will make use of the Japanese bracket notation
\begin{equation*}
    \jBra{x} := \sqrt{1 + x^2}
\end{equation*}

When discussing constants, we will write $C$ to denote a (positive) absolute constant, the exact value of which can change from line to line.  We also write $c_j$ and $c_{j,k}$ to denote sequences which are $\ell^2$ summable to some arbitrary absolute constant:
\begin{equation*}\begin{split}
    \left(\sum_j c_j^2\right)^{1/2} = C\\
    \left(\sum_{j,k} c_{j,k}^2\right)^{1/2} = C
\end{split}\end{equation*}
If $X$ and $Y$ are two quantities which we wish to compare, but we want to suppress constant factors, we will write
\begin{itemize}
    \item $X \lesssim Y$ if $X \leq CY$ for some $C > 0$,
    \item $X \sim Y$ if $X \lesssim Y$ and $Y \lesssim X$,
    \item $X \ll Y$ if $X \leq c Y$, where $c$ is a small constant, the exact value of which depends on the context.
\end{itemize}
If we want to allow the implicit constant to depend on some parameters $P_1, P_2, \cdots P_n$, then we will write $X \lesssim_{P_1, P_2, \cdots P_n} Y$, $X \sim_{P_1, P_2, \cdots P_n} Y,$ or $X \ll_{P_1, P_2, \cdots P_n} Y$, respectively.

We use the Fourier transform convention
\begin{equation*}
    \mathcal{F}f(\xi) = \hat{f}(\xi) := \frac{1}{\sqrt{2\pi}}\int f(x) e^{-ix\xi}\,dx
\end{equation*}
with the inverse transformation
\begin{equation*}
    \mathcal{F}^{-1}(\xi) = \check f(x) := \frac{1}{\sqrt{2\pi}}\int f(x) e^{ix\xi}\,dx
\end{equation*}
Under this convention, multiplication and convolution are linked by
\begin{equation*}\begin{split}
    \mathcal{F}(fg)(\xi) =& \frac{1}{\sqrt{2\pi}} \hat{f} * \hat{g}(x)\\
    \mathcal{F}(f * g)(\xi) =& \sqrt{2\pi} \hat{f}(\xi) \hat{g}(\xi)
\end{split}\end{equation*}
where
\begin{equation*}
    f*g(x) = \int f(x-y) g(y)\,dy = \int f(y) g(x-y)\,dy
\end{equation*}

Using the Fourier transform, we can generalize the notion of differential operators to define Fourier multiplication operators.  A Fourier multiplication operator with symbol $m : \bbR \to \bbC$ is given by
\begin{equation*}
    m(D) f(x) := \mathcal{F}^{-1}(m(\xi) \hat{f}(\xi))(x)
\end{equation*}
The Littlewood-Paley projection operators are an especially important family of Fourier multiplication operators.  Let ${\psi \in C^\infty(\bbR)}$ be a function supported on $B_2(0)$ which is identically zero on $B_{1/2}(0)$ satisfying
\begin{equation*}
    \sum_{j \in \bbZ} \psi(2^j \xi) = 1 
\end{equation*}
for all $\xi \neq 0$.  Then, we define the Littlewood-Paley projectors as
\begin{equation*}
    P_j = \psi_j(D) = \psi\left(\frac{D}{2^j}\right)
\end{equation*}
and define $P_j^+$ and $P_j^-$ to be the projectors to positive and negative frequencies, respectively:
\begin{equation*}
    P_j^+ = \psi_j(D) \bbOne_{D > 0},\qquad\qquad P_j^- = \psi_j(D) \bbOne_{D < 0},
\end{equation*}
We write
\begin{equation*}
    P_{\leq j} = \sum_{k \leq j} P_k,\qquad P_{\geq j} = \sum_{k \geq j} P_k, \qquad P_{[j_1, j_2]} = \sum_{j_1 \leq k \leq j_2} P_k
\end{equation*}
with $P_{<j}$ and $P_{> j}$ being defined similarly.  We also define
\begin{equation*}
    P_{\lesssim j} = \sum_{k \leq j + 10} P_k\qquad P_{\ll j} = \sum_{k < j + 10} P_k \qquad P_{\sim j} = \sum_{j - 10 \leq k \leq j + 10} P_k
\end{equation*}
All the Littlewood-Paley projectors are bounded from $L^p(\bbR) \to L^p(\bbR)$, and moreover we have the Plancherel-type identity
\begin{equation*}
    \sum_{j \in \bbZ}\lVert P_j f \rVert_{L^2}^2 \sim \lVert f \rVert_{L^2}^2
\end{equation*}
Furthermore, if $f$ has mean $0$ (i.e., $\hat{f}(0) = 0$), then taking Fourier transforms and applying Hardy's inequality (see~\cite{hardyInequalities1959}), we find that
\begin{equation}\label{eqn:hardy-est}
    \lVert f \rVert_{\dot{H}^{-1}}^2 := \sum_{2^j} 2^{-2j} \lVert P_j f \rVert_{L^2}^2  \lesssim \lVert xf \rVert_{L^2}^2
\end{equation}

When discussing the complex mKdV equation, we will find it convenient to use the (time-dependent) frequency projectors given by
\begin{equation*}
    Q_j = \begin{cases}
        P_j & 2^j > t^{-1/3}\\
        P_{\leq j} & 2^{j-1} < t^{-1/3} \leq 2^{j}\\
        0 & \text{else}
    \end{cases}
\end{equation*}
Clearly, these projectors obey the same $L^p \to L^p$ bounds as the usual Littlewood-Paley projectors uniformly in time.  The decision not to distinguish between frequencies $\lesssim t^{-1/3}$ is motivated by our desire to have the frequency localization of $\hat{f}$ determine the spatial localization of $u = e^{-t\partial_x^3} f$ through the group velocity relation. For frequencies $\gtrsim t^{-1/3}$, this is possible (see~\Cref{sec:linear-ests}), but the uncertainty principle implies that this spatial localization deteriorates when we project to frequencies $\ll t^{-1/3}$.  To denote the projector to the low frequencies, we will sometimes write
\begin{equation*}
    Q_{\leq \log t^{-1/3}} = \psi_{\leq \log t^{-1/3}}(D) := P_{\leq j}
\end{equation*}
 where $j \in \bbZ$ is such that $2^{j-1} < t^{-1/3} \leq 2^{j}$, so
 \begin{equation*}
     \operatorname{Id} = Q_{\leq \log t^{-1/3}} + \sum_{2^j > t^{-1/3}} Q_j
 \end{equation*}
 As with the projectors $P_j$, we define
\begin{equation*}
    Q_{\leq j} = \sum_{k \leq j} Q_k,\qquad Q_{\geq j} = \sum_{k \geq j} Q_k, \qquad Q_{[j_1, j_2]} = \sum_{j_1 \leq k \leq j_2} Q_k
\end{equation*}
with $Q_{< j}$ and $Q_{> j}$ being defined analogously, and
 \begin{equation*}
     Q_{\lesssim j} = Q_{\leq j + 10},\qquad Q_{\sim j} =  Q_{[j - 10, j + 10]}, \qquad Q_{\ll j} = \chi_{< j - 10}
 \end{equation*}
 We will often indicate the frequency localization of a function through a subscript, so
 \begin{equation*}\begin{split}
     f_j =& Q_j f\\
     f_{[j_1,j_2]} =& Q_{[j_1,j_2]} f
 \end{split}\end{equation*}
 and similarly for $f_{< j}$, $f_{\sim j}$, etc.

To complement the $Q_j$'s, it will be useful to consider time-dependent functions $\chi_j$ with the property that if $f$ is a bump function localized in space near $0$, then $e^{-t\partial_x^3} Q_j f$ will be localized near the support of $\chi_j$ (up to more rapidly decaying tails).  To do this, we define
\begin{equation*}
    \chi_j(x;t) = \begin{cases}
        \chi(x/(t2^{2j})) & 2^j > t^{-1/3}\\
        \sum_{2^k \leq t^{-1/3}} \chi(x/(t2^{2k})) & 2^{j} \leq t^{-1/3} < 2^{j+1}\\
        0 & 2^{j+1} \leq t^{-1/3}
    \end{cases}
\end{equation*}
where $\chi$ is a non-negative bump function localized in the region $|x| \approx 1$ chosen so that $\sum_{j} \chi_j(x,t) = 1$ for all $x \neq 0$.  As with the Fourier projectors, we define
\begin{equation*}
     \chi_{\leq j} = \sum_{k \leq j} \chi_{k}, \qquad \chi_{< j} = \sum_{k < j} \chi_k, \qquad \chi_{[j_1, j_2]} = \sum_{j_1 \leq k \leq j_2} \chi_{k}
\end{equation*}
with $\chi_{> j}$ and $\chi_{\geq j}$ defined in analogously, and
\begin{equation*}
    \chi_{\lesssim j} = \chi_{\leq j + 10},\qquad \chi_{\ll j} = \chi_{< j - 10},\qquad \chi_{\sim j} = \chi_{[j - 10, j + 10]}
\end{equation*}
and similarly for $\chi_{\gtrsim j}$ and $\chi_{\gg j}$.

Note that for $f \in L^2$, each of the families $\{\chi_k f\}$, $\{Q_j f \}$ and $\{\chi_k Q_j f \}$ is almost orthogonal, which implies that
\begin{equation}\label{eqn:phy-and-fourier-almost-orthogonality}
    \lVert f \rVert_{L^2}^2 \sim \sum_{2^k \gtrsim t^{-1/3}} \lVert \chi_k f \rVert_{L^2}^2 \sim  \sum_{2^j \gtrsim t^{-1/3}} \lVert Q_j f \rVert_{L^2}^2 \sim \sum_{2^j, 2^k \gtrsim t^{-1/3}} \lVert \chi_k Q_j f \rVert_{L^2}^2
\end{equation}

We also recall the following bound (which expresses the pseudolocality of the projectors $Q_j$): For $\frac{1}{p} = \frac{1}{p_1} + \frac{1}{p_2}$,
\begin{equation}\label{eqn:P_j-loc-bound}
    \lVert\left( Q_{\leq j} f\right) g\rVert_{L^p} \lesssim_N \jBra{2^j d(\supp(f), \supp(g))}^{-N} \lVert f \rVert_{L^{p_1}}\lVert g \rVert_{L^{p_2}}
\end{equation}
which can be obtained by writing $Q_{\leq j} f = \check{Q}_{\leq j} * f$ and noting that $\check{Q}_{\leq j}$ is a rapidly decreasing function.  In particular, if the supports of $f$ and $g$ are separated by a distance much larger than $2^{-j}$, the term on the right in~\eqref{eqn:P_j-loc-bound} will be small.
\begin{rmk}
    The notion of pseudolocality holds in a much greater generality for pseudodifferential operators $a(x, hD)$, see~\cite{zworskiSemiclassicalAnalysis2012}.  In equation~\eqref{eqn:P_j-loc-bound}, $2^{-j}$ plays the role of the small parameter $h$.
\end{rmk}
\begin{rmk}
    We will often apply the estimate~\eqref{eqn:P_j-loc-bound} as follows: Taking $g = \chi_j$, we find that
    \begin{equation}\label{eqn:pseudoloc-comm-est}
        \left\lVert \chi_j Q_{\leq j} \left((1 - \chi_{\sim j}) F\right)\right\rVert_{L^p} \lesssim_N (t 2^{3j})^{-N} \lVert F \rVert_{L^p}
    \end{equation}
    so
    \begin{equation*}
        \lVert \chi_j Q_{\leq j} F \rVert_{L^p} \lesssim_N \lVert \chi_{\sim j} F \rVert_{L^p} + (t2^{3j})^{-N} \lVert F \rVert_{L^p}
    \end{equation*}
    In our applications, $\chi_j F$ and $\chi_{\sim j} F$ will generally obey the same sorts of bounds, so we can essentially commute the frequency localization of $Q_{\leq j}$ and the spatial localization of $\chi_j$ up to an error which is summable in $j$.
\end{rmk}
\begin{rmk}
    By writing $Q_j = Q_{\leq j} - Q_{\leq j-1}$, we see that the same bounds continue to hold if $Q_{\leq j}$ is replaced by $Q_{j}$
\end{rmk}

\subsection{Multilinear harmonic analysis}\label{sec:multilinear-defs}
For a symbol $m : \bbR^3 \to \bbC$, we define the trilinear pseudoproduct operator $T_m$ by
\begin{equation*}
    \mathcal{F}T_m(f,g,h)(\xi) = \frac{1}{2\pi} \int m(\xi,\eta,\sigma) \hat{f}(\eta) \hat{g}(\xi - \eta - \sigma) \hat{h}(\sigma)\,d\eta d\sigma
\end{equation*}
In particular, $T_1(f,g,h)(x) = (fgh)(x)$, so $T_m$ can be thought of as a generalized product.  If the symbol $m$ is sufficiently well-behaved, we can show that the pseudoproduct $T_m(f,g,h)$ obeys H\"older-type bounds (see also~\cite{coifmanAuDelaOperateurs1978}):
\begin{theorem}\label{thm:L1-symbol-bounds}
    Suppose $m$ is a symbol with $\check{m} \in L^1$ and $\frac{1}{p_1} + \frac{1}{p_2} + \frac{1}{p_3} = \frac{1}{p}$.  Then,
    \begin{equation}
        \lVert T_m(f,g,h) \rVert_{L^p} \lesssim \lVert \check{m} \rVert_{L^1} \lVert f \rVert_{L^{p_1}} \lVert g\rVert_{L^{p_2}} \lVert h\rVert_{L^{p_3}}
    \end{equation}
\end{theorem}
\begin{proof}
    By inverting the Fourier transform, we find
    \begin{equation*}
        T_m(f,g,h)(x) = \frac{1}{(2\pi)^{3/2}}\int \check{m}(y,z,w) f(x - y - z) g(x-y) h(x-y-w)\,dydzdw
    \end{equation*}
    which yields the result by Young's inequality.
\end{proof}
\begin{rmk}\label{rmk:freq-loc-symbol-bounds}
    In our analysis, we will often consider symbols $m$ which are supported on a region of volume $O\left(2^{3j}\right)$ and satisfy the symbol bounds
    \begin{equation*}
        |\partial_{\xi,\eta,\sigma}^\alpha m(\xi,\eta,\sigma)| \lesssim_\alpha 2^{-|\alpha|j}
    \end{equation*}
    For such symbols,
    \begin{equation*}
        \left|\check{m}(y,z,w)\right| \lesssim_N \frac{2^j}{(1 + |2^{j}y|)^N}\frac{2^{j}}{(1 + |2^{j}z|)^N}\frac{2^{j}}{(1 + |2^{j}w|)^N}
    \end{equation*}
    which shows that $m$ satisfies the hypotheses of~\Cref{thm:L1-symbol-bounds}.
\end{rmk}
\begin{rmk}\label{rmk:measure-symbols}
    Note that~\Cref{thm:L1-symbol-bounds} remains true (with essentially the same proof) if we weaken the hypothesis $\check{m} \in L^1$ to $\check{m} \in \mathcal{M}$, where $\mathcal{M}$ is the space of finite measures.  In particular, this generalization covers the motivating example $m = 1$ (which has $\check{m} = (2\pi)^{3/2} \delta \in \mathcal{M}$).  
\end{rmk}
For our application, it is useful to study frequency localized symbols $m_j$ supported on a region $|\xi| + |\eta| + |\sigma| \lesssim 2^j$ and satisfying the decay estimate
\begin{equation*}
    |\partial_{\xi,\eta,\sigma}^\alpha m_j(\xi,\eta,\sigma)| \lesssim_\alpha 2^{-j|\alpha|} 
\end{equation*}
In the work that follows, we will often allow these symbols to denote different functions from line-to-line.  These symbols satisfy the hypotheses of~\Cref{rmk:freq-loc-symbol-bounds}.  They also have the pseudolocality property given in the following lemma:
\begin{lemma}\label{thm:cm-paraprod-pseudolocality}
    Suppose that $f_1, f_2, f_3, f_4$ are functions, and suppose that $\supp f_i$ and $\supp f_k$ are separated by a distance $R$ for some $i \neq k$.  Then, for any $p_1,p_2,p_3,p_4,p \in [1,\infty]$ with $\frac{1}{p_1} + \frac{1}{p_2} + \frac{1}{p_3} + \frac{1}{p_4} = \frac{1}{p}$,
    \begin{equation}\label{eqn:pseudoloc-1}
        \lVert f_4 T_{m_j}(f_1,f_2,f_3)  \rVert_{L^p} \lesssim_N \jBra{2^j R}^{-N} \lVert f_1 \rVert_{L^{p_1}} \lVert f_2 \rVert_{L^{p_2}} \lVert f_3 \rVert_{L^{p_3}} \lVert f_4 \rVert_{L^{p_4}}
    \end{equation}
\end{lemma}
\begin{proof}
    We will assume $i=1, k = 2$, since the other cases are similar.  Arguing as in the proof of~\Cref{thm:L1-symbol-bounds} and using the hypothesis on the supports of $f_1$ and $f_2$, we have that
    \begin{equation*}
        f_4T_{m_j}(f_1,f_2,f_3) = \smashoperator[lr]{\int_{|z| \geq R}} \check{m}_j(y,z,w) f_1\left(x - y - z\right) f_2\left(x - y\right) f_3\left(x - y - w\right) f_4(x)\,dy dz dw
    \end{equation*}
    Using Minkowski's and H\"older's inequalities, we see that~\eqref{eqn:pseudoloc-1} reduces to showing that 
    \begin{equation*}
        \int_{|z| \geq R} |\check{m}_j(y,z,w)|\,dydzdw \lesssim_N \jBra{2^j R}^{-N}
    \end{equation*}
    Since~\Cref{rmk:freq-loc-symbol-bounds} gives us the bound
    \begin{equation*}
        |\check{m}_j(y,z,w)| \lesssim_N 2^{3j}\jBra{(2^j y,2^j z,2^j w)}^{-N-10}
    \end{equation*}
    the result follows.
\end{proof}

\section{Linear and multilinear estimates}\label{sec:linear-ests}

\subsection{The linear estimate}
We now turn our attention to the linear part of~\eqref{eqn:cmkdv}.  We begin by proving a linear estimate for the Airy propagator.  Define the spaces $X_j$ for $2^j \gtrsim t^{-1/3}$ by the norm
\begin{equation}\label{eqn:X-j-norm-def}
    \lVert f \rVert_{X_j} := \lVert \widehat{Q_{\sim j} f} \rVert_{L^\infty} + t^{-1/6} \lVert x Q_{\sim j} f \rVert_{L^2}
\end{equation}
and note that $\lVert f \rVert_{X_j}^2 \lesssim \lVert f \rVert_{X}^2$, where $X$ is the norm defined in~\eqref{eqn:X-def}.  
\begin{lemma}\label{thm:lin-decay-lemma}
    Let $u(x,t) = e^{-t\partial_x^3} f(x,t)$.  For $2^j > t^{-1/3}$, we have the pointwise estimate
    \begin{equation}\label{eqn:freq-loc-ptwise-decay}\begin{split}
        P^{\pm}_{j} u(x,t) =& \frac{1}{\sqrt{12 t \xi_0}} e^{\mp 2it\xi_0^3 \pm i \frac{\pi}{4}} \widehat{P^{\pm}_j f}\left(\pm \xi_0\right) \bbOne_{x < 0}\\ &\qquad + O\left(t^{-1/3} \left(2^j t^{1/3}\right)^{-9/14}\right) \bbOne_{x < 0}\chi_{\sim j}(x,t) \lVert f \rVert_{X_j}\\
        &\qquad + O\left(t^{-1/3} \left((2^j + 2^{-j/3} |\xi_0|^{4/3}) t^{1/3}\right)^{-3/2}\right)\lVert f \rVert_{X_j}
    \end{split}\end{equation}
    where $\xi_0 = \sqrt{\left|\frac{x}{3t}\right|}$.  Moreover, we have the estimate
    \begin{equation}\label{eqn:low-freq-ptwise-decay}
        |Q_{\leq \log t^{-1/3}} u(x,t)| \lesssim t^{-1/3} (1 + t^{2/3} \xi_0^2)^{-1} \lVert f \rVert_{X_{\leq \log t^{-1/3}}}
    \end{equation}
    so for $p \in [4, \infty]$,
    \begin{equation}\label{eqn:freq-loc-lp-decay}
        \lVert Q_j u \rVert_{L^p} \lesssim t^{-\frac{1}{2} + \frac{1}{p}} 2^{\left(\frac{2}{p} - \frac{1}{2}\right) j} \lVert f \rVert_{X_j}
    \end{equation}
    In particular, if $p > 4$,
    \begin{equation}\label{eqn:lp-decay}
        \lVert u \rVert_{L^p} \lesssim t^{-\frac{1}{3} + \frac{1}{3p}} \lVert f \rVert_{X}
    \end{equation}
\end{lemma}
\begin{proof}
    The estimate \Cref{eqn:freq-loc-lp-decay} follows directly from~\cref{eqn:freq-loc-ptwise-decay,eqn:low-freq-ptwise-decay}, and~\cref{eqn:lp-decay} follows from~\cref{eqn:low-freq-ptwise-decay,eqn:freq-loc-lp-decay} since
    \begin{equation*}
        \lVert u \rVert_{L^p} \lesssim \lVert Q_{\leq \log t^{-1/3}} u \rVert_{L^p} + \sum_{2^j > t^{-1/3}} \lVert Q_j u \rVert_{L^p}
    \end{equation*}
    so it only remains to prove~\eqref{eqn:freq-loc-ptwise-decay} and \eqref{eqn:low-freq-ptwise-decay}.
    
    The low frequency estimate \cref{eqn:low-freq-ptwise-decay} will follow once we prove the bound
    \begin{equation*}
        |Q_{\leq \log t^{-1/3}} u| \lesssim \min(t^{-1/3}, t^{-1} \xi_0^{-2})\lVert f \rVert_{X_{\leq \log t^{-1/3}}}
    \end{equation*}
    The bound $|Q_{\leq \log t^{-1/3}} u| \lesssim t^{-1/3} \lVert f \rVert_{X_{\leq \log t^{-1/3}}}$ follows immediately from the Hausdorff-Young inequality, so it suffices to consider the case $|\xi_0| \gg t^{-1/3}$.  In this case, writing
    \begin{equation*}
        Q_{\leq \log t^{-1/3}} u(x, t) = \frac{1}{\sqrt{2\pi}} \int_\bbR \psi_{\leq \log t^{-1/3}}(\xi) e^{it\philin(\xi)} \hat{f}(\xi)\,d\xi
    \end{equation*}
    for $\philin(\xi) = \frac{x}{t}\xi + \xi^3 = \xi^3 - 3\xi_0^2\xi$, the assumption $|\xi_0| \gg t^{-1/3}$ implies that $|\partial_\xi \philin| \sim |\xi_0|^2$ on the support of $\psi_{\leq \log t^{-1/3}}\hat{f}$, so integration by parts yields
    \begin{equation*}\begin{split}
        |Q_{\leq \log t^{-1/3}} u(x,t)| \lesssim& \frac{1}{t} \int_\bbR \left|\partial_\xi \left(\frac{1}{\partial_\xi \philin}\right)\right| |\psi_{\leq \log t^{-1/3}}\hat{f}(\xi)|\,d\xi\\
        &+ \frac{1}{t} \int_\bbR \left|\frac{1}{\partial_\xi \philin}\right| \left|\partial_\xi  \left(\psi_{\leq \log t^{-1/3}}\hat{f}(\xi)\right)\right|\,d\xi\\
        \lesssim& t^{-1} |\xi_0|^{-2} \lVert f \rVert_{X_{\leq t^{-1/3}}}
    \end{split}\end{equation*}
    as required.
    
    We now turn to the estimate~\eqref{eqn:freq-loc-ptwise-decay}.  We consider the estimate for $P^{+}_j u$: the estimate for $P^{-}_j u$ is similar.  As before, we write
    \begin{equation*}
        P^+_j u(x,t) = \frac{1}{\sqrt{2\pi}} \int_0^\infty \psi^+_j(\xi) e^{it\philin(\xi)} \hat{f}(\xi)\,d\xi
    \end{equation*}
    We distinguish between three cases depending on the relative sizes of $2^j$ and $|\xi_0|$ and the sign of $x$.
    
    \paragraph{\indent \textbf{Case} $|\xi_0| < 2^{j-10}$}
    In this case, $|\partial_\xi \philin| \sim 2^{2j}$, and integration by parts gives
    \begin{subequations}\begin{align}
        |P^+_j u(x,t)| \lesssim& \frac{1}{t} \int_0^\infty \left|\partial_\xi\left( \frac{\psi^+_j(\xi)}{\partial_\xi \philin}\right)\right| |\psi^+_{\sim j}(\xi)\hat{f}(\xi)|\,d\xi\label{eqn:lin-est-small-xi-0-a}\\
        &+ \frac{1}{t} \int_0^\infty \left|\psi^+_j(\xi) \frac{1}{\partial_\xi \philin}\right| |\psi^+_{\sim j}(\xi)\partial_\xi \hat{f}(\xi)|\,d\xi\label{eqn:lin-est-small-xi-0-b}
    \end{align}\end{subequations}
    For~\eqref{eqn:lin-est-small-xi-0-a}, we observe that $\partial_{\xi} \frac{\psi^+_j(\xi)}{\partial_\xi \philin}$ has size $O(2^{-3j})$ and is supported on a region of size $O(2^j)$, so using Hardy's inequality~\eqref{eqn:hardy-est} yields
    \begin{equation*}\begin{split}
        \eqref{eqn:lin-est-small-xi-0-a} 
        \lesssim& t^{-1/3} \left(t^{1/3} 2^j \right)^{-3/2} \lVert f \rVert_{X_j}
    \end{split}\end{equation*}
    Similarly, $\frac{\psi^+_j(\xi)}{\partial_\xi \philin}$ has size $O(2^{-2j})$ and is supported on a region of size $O(2^j)$, so
    \begin{equation*}\begin{split}
        \eqref{eqn:lin-est-small-xi-0-b} 
        \lesssim& t^{-1/3} \left(2^j t^{1/3}\right)^{-3/2} \lVert f \rVert_{X_j}
    \end{split}\end{equation*}
    
    \paragraph{\indent \textbf{Case} $|\xi_0| > 2^{j+10}$}
    In this case, $|\partial_\xi \philin| \sim \xi_0^2$, and a quick calculation gives that
    \begin{equation}\label{eqn:lin-est-large-xi-0-bounds}\begin{split}
        \left\lVert \partial_{\xi} \frac{\psi^+_j(\xi)}{\partial_\xi \philin(\xi)} \right\rVert_{L^2} \lesssim& \frac{1}{\xi_0^2 2^{j/2}}\\
        \left\lVert \frac{\psi^+_j(\xi)}{\partial_\xi \philin(\xi)} \right\rVert_{L^2} \lesssim& \frac{2^{j/2}}{\xi_0^2}\\
    \end{split}\end{equation}
    Integrating by parts, we find that
    \begin{subequations}\begin{align}
        |P_j^{\pm} u(x,t)| \lesssim& \frac{1}{t} \int_0^\infty \left|\partial_{\xi} \frac{\psi^+_j(\xi)}{\partial_\xi \philin(\xi)} \right| |\psi^+_{\sim j} \hat{f}(\xi)|\,d\xi\label{eqn:lin-est-large-xi-0-1}\\
        &+ \frac{1}{t}\int_0^\infty \left|\frac{\psi^+_j(\xi)}{\partial_\xi \philin(\xi)} \right| |\psi^+_{\sim j} \partial_\xi \hat{f}(\xi)|\,d\xi\label{eqn:lin-est-large-xi-0-2}
    \end{align}\end{subequations}
    Using the bounds~\eqref{eqn:lin-est-large-xi-0-bounds} and arguing as in the case $|\xi_0| < 2^{j-10}$, we find that
    \begin{equation*}\begin{split}
        \eqref{eqn:lin-est-large-xi-0-1} + \eqref{eqn:lin-est-large-xi-0-2} 
        \lesssim& t^{-1/3} \left(t^{1/3} 2^{-j/3}\xi_0^{-4/3}\right)^{-3/2}\lVert f \rVert_{X_{j}}
    \end{split}\end{equation*}
    
     \paragraph{\indent \textbf{Case} $x > 0$, $2^{j-10} \leq |\xi_0| \leq 2^{j+10}$}
     
     Here, $|\partial_\xi \philin| \sim \xi_0^{-2}$, so the estimate is identical to the previous case.
    
    \paragraph{\indent \textbf{Case} $x < 0$, $2^{j-10} \leq |\xi_0| \leq 2^{j+10}$}
    Since $\partial_\xi \philin$ vanishes at $\xi = \xi_0$, we employ the method of stationary phase.  Let us write 
    \begin{equation*}
        P^+_j u(x,t) = \sum_{\ell = \ell_0}^{j+10} I_{j,\ell}
    \end{equation*}
    where
    \begin{equation*}\begin{split}
        I_{j,\ell} =& \frac{1}{\sqrt{2\pi}} \int_0^\infty \psi^+_j(\xi) \psi_{\ell}(\xi - \xi_0) e^{it\philin(\xi)} \widehat{Q_{\sim j}f}(\xi)\,d\xi, \qquad \ell > \ell_0\\
        I_{j,\ell_0} =& \frac{1}{\sqrt{2\pi}} \int_0^\infty \psi^{+}_j(\xi)\psi_{\leq \ell_0}(\xi - \xi_0) e^{it\philin(\xi)} \widehat{Q_{\sim j} f}(\xi)\,d\xi
    \end{split}\end{equation*}
    and $\ell_0$ is chosen such that $2^{\ell_0} \sim t^{-1/3} (2^j t^{1/3})^{-3/7}$.
    
    For the $I_{j,\ell}$ factors with $\ell > \ell_0$, we have that $|\partial_\xi \philin| \sim 2^\ell 2^j$.  Integrating by parts, we find that
    \begin{equation*}\begin{split}
        |I_{j,\ell}| \lesssim& \frac{1}{t} \int_0^\infty \left| \partial_\xi \frac{\psi^+_j(\xi)\psi_{\ell}(\xi - \xi_0)}{\partial_\xi\philin} \right| |\psi_{\sim j}(\xi) \hat{f}(\xi)| \,d\xi\\
        &+ \frac{1}{t} \int_0^\infty \left|\frac{\psi^+_j(\xi)\psi_{\ell}(\xi - \xi_0)}{\partial_\xi\philin} \right| |\psi_{\sim j}(\xi) \partial_\xi\hat{f}(\xi)| \,d\xi\\
        \lesssim& t^{-1} \left(2^{-j}2^{-\ell} \lVert \widehat{Q_{\sim j}f} \rVert_{L^\infty} + 2^{-j} 2^{-\ell/2} \lVert \partial_\xi \widehat{Q_{\sim j}f} \rVert_{L^2}\right)
    \end{split}\end{equation*}
    summing over $\ell > \ell_0$ gives
    \begin{equation}\label{eqn:far-from-st-pt-est}\begin{split}
        \sum_{\ell > \ell_0}|I_{j,\ell}| 
        \lesssim& \left( t^{-1}  2^{-j} 2^{-\ell_0} + t^{-5/6} 2^{-j} 2^{-\ell_0/2}\right) \lVert f \rVert_{X_j}
    \end{split}\end{equation}
    For the $I_{j,\ell_0}$ term, we write
    \begin{subequations}\begin{align}
        I_{j,\ell_0} =& \frac{1}{\sqrt{2\pi}} \int_0^\infty  \psi_{\leq \ell_0}(\xi - \xi_0) e^{it\philin(\xi)} \left(\psi_j^{+}(\xi)\widehat{Q_{\sim j}f}(\xi) - \psi_j^{+}(\xi_0)\widehat{Q_{\sim j}f}(\xi_0)\right)\,d\xi\label{eqn:lin-stat-ph-1}\\
        &+ \frac{1}{\sqrt{2\pi}} \psi_j^{+}(\xi_0) \hat{f}(\xi_0) \int_0^\infty  \psi_{\leq \ell_0}(\xi - \xi_0) \left(e^{it\philin(\xi)} - e^{6it\xi_0 (\xi-\xi_0)^2 -2 it \xi_0^3}\right) \,d\xi\label{eqn:lin-stat-ph-2}\\
        &+ \frac{1}{\sqrt{2\pi}} \psi_j^{+}(\xi_0)\hat{f}(\xi_0)e^{-2it\xi_0^3} \int_0^\infty  \psi_{\leq \ell_0}(\xi - \xi_0) e^{6it \xi_0 \xi^2} \,d\xi\label{eqn:lin-stat-ph-3}
    \end{align}\end{subequations}
    For the first term, we note that
    \begin{equation*}\begin{split}
        \Big| \psi^+_j(\xi)\widehat{Q_{\sim j} f}(\xi) - \psi^+_j(\xi_0)\widehat{Q_{\sim j} f}(\xi_0)\Big| \lesssim& \big(2^{-j}|\xi - \xi_0| + t^{1/6}|\xi - \xi_0|^{1/2} \big) \lVert f \rVert_{X_j}
    \end{split}\end{equation*}
    by the Sobolev-Morrey embedding $\dot{H}^1 \to C^{0,1/2}$.  Using this bound, we find that
    \begin{equation*}\begin{split}
        |\eqref{eqn:lin-stat-ph-1}| 
        \lesssim& \left(2^{3/2\ell_0} t^{1/6} + 2^{-j} 2^{2\ell_0}\right) \lVert f \rVert_{X_j}
    \end{split}\end{equation*}
    For the second term, we observe that
    \begin{equation*}
        \philin(\xi) = -2\xi_0^3 + 6\xi_0 (\xi - \xi_0)^2 + (\xi - \xi_0)^3
    \end{equation*} 
    so 
    \begin{equation*}\begin{split}
        |\eqref{eqn:lin-stat-ph-2}| \lesssim& \lVert \psi_{j}^+(\xi) \hat{f} \rVert_{L^\infty}\int_0^\infty \psi_{\leq \ell_0} (\xi - \xi_0) \left| e^{it (\xi - \xi_0)^3} - 1 \right| \,d\xi\\ 
        \lesssim& t 2^{4\ell_0} \lVert f \rVert_{X_k}
    \end{split}\end{equation*}
    Finally, rescaling and using the classical stationary phase estimate gives
    \begin{equation*}\begin{split}
        \eqref{eqn:lin-stat-ph-3} =& \frac{1}{\sqrt{2\pi}} \psi_j^{+}(\xi_0)\hat{f}(\xi_0)e^{-2it\xi_0^3} \int_0^\infty  \psi_{\leq \ell_0}(\xi - \xi_0) e^{6it \xi_0 (\xi-\xi_0)^2} \,d\xi\\
        =& \frac{2^{\ell_0}}{\sqrt{2\pi}} \psi_j^{+}(\xi_0)\hat{f}(\xi_0)e^{-2it\xi_0^3} \int_\bbR  \psi_{\leq 0}(\xi) e^{6it \xi_0 2^{2\ell_0} \xi^2} \,d\xi\\
        =& \frac{\psi_j^{+}(\xi_0)\hat{f}(\xi_0)}{\sqrt{12 t \xi_0}}e^{-2it\xi_0^3+i\frac{\pi}{4}} + O\left(t^{-3/2}2^{-2\ell_0} 2^{-3/2j} \lVert \widehat{Q_{\sim j} f} \rVert_{L^\infty}\right)
    \end{split}\end{equation*}
    Collecting the terms~\cref{eqn:far-from-st-pt-est,eqn:lin-stat-ph-1,eqn:lin-stat-ph-2,eqn:lin-stat-ph-3} and recalling the definition of $\ell_0$, we find that
    \begin{equation*}
        P^+_j u(x, t) = \frac{\psi_j^{+}(\xi_0)\hat{f}(\xi_0)}{\sqrt{12 t \xi_0}}e^{-2it\xi_0^3 + i\frac{\pi}{4}} + O\left(t^{-1/3}(2^j t^{1/3})^{-9/14}\lVert f \rVert_{X_j}\right)
    \end{equation*}
\end{proof}

As a corollary of the above estimate, we obtain an improved bilinear decay estimate:
\begin{corollary}\label{thm:simple-bilinear-decay}
    If $f, g \in X$, then
    \begin{equation*}
        |e^{-t\partial_x^3}f e^{-t\partial_x^3} \partial_x g| \lesssim t^{-1} \lVert f \rVert_{X} \lVert g \rVert_{X}
    \end{equation*}
\end{corollary}
\begin{proof}
    It suffices to show that
    \begin{equation*}
        |\chi_k e^{-t\partial_x^3}f e^{-t\partial_x^3} \partial_x g| \lesssim t^{-1} \lVert f \rVert_{X} \lVert g \rVert_{X}
    \end{equation*}
    Using~\Cref{thm:lin-decay-lemma}, we find that
    \begin{equation*}
        |\chi_{\sim k} Q_{j} e^{-t\partial_x^3}f| \lesssim \begin{cases}
            t^{-5/6} 2^{-3/2j} \lVert f \rVert_{X_j} & k < j - 20\\
            t^{-1/2} 2^{-k/2} \lVert f \rVert_{X_j} & |j - k| \leq 20\\
            t^{-5/6} 2^{j/2 - 2k} \lVert f \rVert_{X_j} & k > j - 20
        \end{cases}
    \end{equation*}
    Thus, summing in $j$, we find that
    \begin{equation*}
        |\chi_{\sim k} e^{-t\partial_x^3}f| \lesssim t^{-1/2} 2^{-k/2} \lVert f \rVert_{X}
    \end{equation*}
    Since $\lVert \partial_x g \rVert_{X_j} \sim 2^j \lVert g \rVert_{X_j}$, a similar calculation shows that
    \begin{equation*}
        |\chi_{\sim k} e^{-t\partial_x^3} \partial_x g| \lesssim t^{-1/2} 2^{k/2} \lVert g \rVert_{X}
    \end{equation*}
    which gives the result.
\end{proof}

\begin{rmk}
    Although \Cref{thm:simple-bilinear-decay} does not apply directly to pseudoproducts, we will see in \Cref{sec:weighted-L2} that~\Cref{thm:lin-decay-lemma} can be used together with the pseudolocality of pseudoproducts given in~\Cref{thm:cm-paraprod-pseudolocality} to give bounds of the same type for pseudoproducts involving a $\partial_x u$ term.
\end{rmk}
From the proof of~\Cref{thm:simple-bilinear-decay}, we also get a decay bound on $\chi_{\geq k} u$:
\begin{corollary}\label{thm:u-space-loc-lin-decay}
    If $f \in X$ and $u = e^{-t\partial_x^3} f$,
    \begin{equation*} \lVert \chi_{\geq k} u \rVert_{L^\infty} \lesssim t^{-1/2} 2^{-k/2} \lVert f \rVert_X
    \end{equation*}
\end{corollary}

It will also be important later to have decay estimates for $e^{-t\partial_x^3} g$ when $\hat{g}(0) = 0$.  We record them here:
\begin{corollary}\label{thm:w-lin-decay}
    Suppose $\hat{g}(0) = 0$ and $xg \in L^2$.  Then, if $w = e^{-t\partial_x^3} g$, we have the bounds 
    \begin{equation}\label{eqn:w-lin-decay-est}
        |Q_j w(x,t)| \lesssim \left(t^{-1/2} \chi_{\sim j} + t^{-5/6} 2^{j/2} \left(2^j + 2^{-j/3} |\xi_0|^{4/3} \right)^{-3/2} \right) \lVert xg \rVert_{L^2}c_j
    \end{equation}
    and
    \begin{equation}\label{eqn:w-space-loc-decay-est}
        \lVert \chi_{k} w \rVert_{L^\infty} \lesssim t^{-1/2} \lVert xg \rVert_{L^2} c_k
    \end{equation}
\end{corollary}
\begin{proof}
    By the Sobolev-Morrey embedding $\dot{H}^1 \to C^{0,1/2}$, we have that for $2^j \gtrsim t^{-1/3}$,
    \begin{equation*}\begin{split}
        \lVert g \rVert_{X_j} 
        \lesssim& (1 + t^{-1/6}2^{-j/2}) \lVert \mathcal{F}\left({Q_{[j - 20, j + 20]} g}\right) \rVert_{L^\infty} + t^{-1/6}\lVert Q_{[j - 20, j + 20]} (xg) \rVert_{L^2}\\
        \lesssim& 2^{j/2} \lVert Q_{[j - 30, j + 30]} (xg) \rVert_{L^2}
    \end{split}\end{equation*}
    so applying~\Cref{thm:lin-decay-lemma} gives~\eqref{eqn:w-lin-decay-est}.  To prove the localized bound~\eqref{eqn:w-space-loc-decay-est}, we note that~\eqref{eqn:w-lin-decay-est} implies that
    \begin{equation*}\begin{split}
        \lVert \chi_{k} w \rVert_{L^\infty} \leq& \lVert \chi_{k} Q_{\sim k} w \rVert_{L^\infty} + \lVert \chi_{k} Q_{\ll k} w \rVert_{L^\infty} + \lVert \chi_{k} Q_{\gg k} w \rVert_{L^\infty}\\
        \lesssim& \sum_{|k-\ell| \leq 10} t^{-1/2} \lVert xg \rVert_{L^2}c_{\ell} + t^{-5/6} \sum_{\ell < k-10 } 2^{\ell - 2k} \lVert xg \rVert_{L^2} + \sum_{\ell > k+ 10} 2^{-\ell} \lVert xg \rVert_{L^2}\\
        \lesssim& t^{-1/2} \lVert xg \rVert_{L^2}c_k + t^{-5/6} 2^{-k} \lVert xg \rVert_{L^2}
    \end{split}\end{equation*}
    which can be seen to satisfy the required $\ell^2_k$ summability condition.
\end{proof}

\subsection{Bounds for cubic terms\label{sec:cubic-bounds}}

Since the complex mKdV equation has a cubic nonlinearity, we will naturally find ourselves dealing with frequency-localized terms of the form $Q_{j}(|u|^2 \partial_x u)$ and the like.  In this section, we collect some basic bounds for these terms for later reference.  Let $u_i = e^{-t\partial_x^3} f_i$ for $i = 1,2,3$ and $f_i \in X$.  We begin by considering $\chi_{\geq j - 10} Q_{\sim j}(u_1 u_2 u_3)$.  Using~\eqref{eqn:pseudoloc-comm-est} to control the error in commuting the physical and Fourier localization operators, we find that
\begin{equation}\label{eqn:cubic-sim-larger}\begin{split}
    \lVert \chi_{\geq j - 10} Q_{\sim j}(u_1 u_2 u_3)\rVert_{L^\infty} \lesssim& \lVert \chi_{\geq j - 10} Q_{\sim j} (\chi_{\geq j - 20} u_1 u_2 u_3)\rVert_{L^\infty} \\
    & + \lVert \chi_{\geq j - 10} Q_{\sim j} (\chi_{< j - 20} u_1 u_2 u_3)\rVert_{L^\infty}\\
    \lesssim& t^{-3/2} 2^{-3/2j} \prod_{i=1}^3 \lVert f_i \rVert_X
\end{split}\end{equation}
Arguing in the same manner, we find that
\begin{equation}\label{eqn:cubic-bounds-compendium}\begin{split} 
    \lVert \chi_{\geq j - 10} Q_{\lesssim j}(u_1 u_2 u_3) \rVert_{L^\infty} \lesssim& t^{-3/2} 2^{-3/2 j}\lVert f_1 \rVert_{X}\lVert f_2 \rVert_{X}\lVert f_3 \rVert_{X}
\end{split}\end{equation}
Moreover, by using the bilinear decay estimate given in~\Cref{thm:simple-bilinear-decay}, we can obtain analogous bounds for terms containing a derivative:
\begin{equation}\label{eqn:cubic-deriv-bounds-compendium}\begin{split}
    \lVert \chi_{\geq j - 10} Q_{\sim j}(u_1 \partial_x u_2 u_3) \rVert_{L^\infty} \lesssim& t^{-3/2} 2^{-j/2 }\prod_{i=1}^3 \lVert f_i \rVert_X\\
    \lVert \chi_{\geq j - 10} Q_{\lesssim j}(u_1 \partial_x u_2 u_3) \rVert_{L^\infty} \lesssim& t^{-3/2} 2^{-j/2 }\prod_{i=1}^3 \lVert f_i \rVert_X
\end{split}\end{equation}

In the bounds~\cref{eqn:cubic-sim-larger,eqn:cubic-bounds-compendium,eqn:cubic-deriv-bounds-compendium}, the frequency localization in some sense does not play a role, as the bounds would remain true if we removed the frequency projection operator.  The situation changes when we consider the cubic $Q_{\sim j}(u_1 u_2 u_3)$ in the region $|x| \ll t 2^{2j}$, since the frequency localization eliminates the worst contribution in this region.  A straightforward paraproduct decomposition yields
\begin{equation*}\begin{split}
    Q_{\sim j}\left(u_1 u_2 u_3\right) =& Q_{\sim j} \bigg(u_{1, < j-20}u_{2, < j-20 }u_{3, < j - 20} +  u_{1,[j-20, j+20]} u_{2, < j - 20} u_{3, < j - 20} \\
    &\qquad+ \sum_{\ell \geq j - 20} u_{1, \ell} u_{2,[\ell -20, \ell + 20]} u_{3, < \ell - 20} \\
    &\qquad+ \sum_{\ell \geq j - 20} u_{1, \ell} u_{2,[\ell -20, \ell + 20]} u_{3, [\ell -20, \ell + 20]}\bigg)\\ &+ \{\text{similar terms}\}
\end{split}\end{equation*}
Note that the term $Q_{\sim j}(u_{1, < j-20}u_{2, < j-20 }u_{3, < j - 20})$ vanishes.  If $k < j - 30$, then we can use~\eqref{eqn:pseudoloc-comm-est} to commute the physical- and frequency-space localizations and apply~\Cref{thm:lin-decay-lemma} to find that
\begin{equation*}\begin{split}
    \Big\lVert \chi_{k} Q_{\sim j} \big(u_{1, [j -20, j + 20]} u_{2, < j-20}, u_{3, < j - 20}\big) \Big\rVert_{L^\infty} \lesssim& t^{-11/6} 2^{-3/2 j-k}\prod_{i=1}^3 \lVert f_i \rVert_X\\
    \Big\lVert \chi_{k} Q_{\sim j} \big(u_{1, \ell} u_{2, [\ell -20, \ell + 20]} u_{3, < \ell - 20}\big)\Big\rVert_{L^\infty} \lesssim& t^{-\frac{13}{6}}2^{-j-3/2k-2\ell} \prod_{i=1}^3 \lVert f_i \rVert_X\\
    \Big\lVert \chi_{k} Q_{\sim j} \big(u_{1, \ell} u_{2,[\ell -20, \ell + 20]} u_{3, [\ell -20, \ell + 20]}\big)\Big\rVert_{L^\infty} \lesssim &  t^{-5/2}2^{-j-2k-3/2 \ell}\prod_{i=1}^3 \lVert f_i \rVert_X
\end{split}\end{equation*}
Summing in $\ell$, we find that
\begin{equation}\label{eqn:cubic-sim-with-k}
    \lVert \chi_{k} Q_{\sim j}\left(u_1 u_2 u_3\right) \rVert_{L^\infty} \lesssim t^{-11/6}2^{-3/2 j}2^{-k}\prod_{i=1}^3 \lVert f_i \rVert_X
\end{equation}
Summing over $k < j - 30$ and combining the result with~\eqref{eqn:cubic-sim-larger}, we obtain the bound
\begin{equation}\label{eqn:cubic-sim}
    \lVert Q_{\sim j}\left(u_1 u_2 u_3\right) \rVert_{L^\infty} \lesssim t^{-3/2} 2^{-3/2j} \prod_{i=1}^3 \lVert f_i \rVert_X
\end{equation}
By a similar argument, we find that for $k < j - 30$,
\begin{equation}\label{eqn:cubic-sim-w-deriv-small}
    \lVert \chi_k Q_{\sim j} (u_1 \partial_x u_2 u_3)\rVert_{L^\infty} \lesssim t^{-11/6} 2^{-j/2} 2^{-k}\prod_{i=1}^3 \lVert f_i \rVert_X
\end{equation}
so
\begin{equation}
    \label{eqn:cubic-sim-w-deriv}
    \lVert Q_{\sim j} (u_1 \partial_x u_2 u_3)\rVert_{L^\infty} \lesssim t^{-3/2} 2^{-j/2} \prod_{i=1}^3 \lVert f_i \rVert_X
\end{equation}

It will also be important to bound the mean of terms like $|u|^2 \partial_x u$, since these terms will arise naturally when we consider $\partial_t \hat{u}(0,t)$.
\begin{theorem}\label{thm:cubic-mean-thm}
    Suppose $u_i = e^{-t\partial_x^3} f_i$ for $i=1,2,3$.  Then,
    \begin{equation}\label{eqn:cubic-mean-bound}
        \int u_1 \overline{u_2} \partial_x u_3 \,dx \lesssim t^{-1} \prod_{i=1}^3\lVert f_i \rVert_{X} 
    \end{equation}
    If instead $\hat{f}_k(0,t) = 0$ for some $k = 1 ,2,3$, then
    \begin{equation}\label{eqn:cubic-mean-bound-mean-0}
        \int u_1 \overline{u_2} \partial_x u_3 \,dx \lesssim t^{-7/6} \lVert Lu_k \rVert_{L^2}\prod_{i\in \{1,2,3\}\setminus\{k\}}\lVert f_i \rVert_{X} 
    \end{equation}
    
\end{theorem}
\begin{proof}
    Let us write
    \begin{equation*}
        \int u_1 \overline{u_2} \partial_x u_3 \,dx = \rmI_j + \tilde{\rmI}_j
    \end{equation*}
    where 
    \begin{equation*}\begin{split}
        \rmI_j =& \int u_{1,j} \overline{u_{2,\leq j}} \partial_x u_3 \,dx\\
        \tilde{\rmI}_j =& \int u_{1,< j} \overline{u_{2,j}} \partial_x u_3 \,dx
    \end{split}
    \end{equation*}
    We will show how to bound $\rmI_j$: the bounds for $\tilde{\rmI}_j$ are analogous.  For $2^j \sim t^{-1/3}$, the frequency localization allows us to replace the $u_3$ factor in the definition with $u_{3, \lesssim j}$, so \begin{equation*}\begin{split}
        I_j \lesssim& \lVert u_{1,j} \rVert_{L^\infty} \lVert u_{2,\leq j} \rVert_{L^2} \lVert \partial_x u_{3, \lesssim j} \rVert_{L^2}\\
        \lesssim& t^{-1/3} 2^{2j} \prod_{i=1}^3 \lVert f_i \rVert_X\\
        \lesssim& t^{-1} \prod_{i=1}^3 \lVert f_i \rVert_X
    \end{split}\end{equation*}
    which is acceptable, so it only remains to consider the case $2^j \gg t^{-1/3}$.  There, we can write
    \begin{equation*}
        I_j = -\int e^{it\phi(0,\eta,\sigma)} (\eta + \sigma) \psi_j(\eta) \psi_{\leq j}(-\sigma) \hat{f_1}(\eta) \overline{\hat{f_2}(-\sigma)} \hat{f_3}(-\eta -\sigma)\,d\eta d\sigma
    \end{equation*}
    where $\phi$ is the four-wave mixing function given in~\eqref{eqn:4-wave-phase}.  Since $\nabla_{\eta,\sigma} \phi(0,\eta, \sigma)$ vanishes only at $\eta = \sigma = 0$, the phase in these integrals is nonstationary, and we can integrate by parts using the identity
    \begin{equation*}
        e^{it \phi(0,\eta,\sigma)} =  \frac{\nabla_{\eta,\sigma} \phi(0,\eta, \sigma)}{it|\nabla_{\eta,\sigma} \phi(0,\eta, \sigma)|^2} \cdot e^{it \phi(0,\eta,\sigma)}
    \end{equation*}
    to find that
    \begin{subequations}\begin{align}
        I_j =& -\frac{i}{t} \int \nablaes \cdot \mu_j \hat{f_1}(\eta) \overline{\hat{f_2}(-\sigma)}\hat{f_3}(-\eta -\sigma) \,d\eta d\sigma \label{eqn:I-j-1}\\
        &- \frac{i}{t} \int \mu_j^{\eta} \partial_\eta\hat{f_1}(\eta) \overline{\hat{f_2}(-\sigma)}\hat{f_3}(-\eta -\sigma) \,d\eta d\sigma \label{eqn:I-j-2}\\
        &- \frac{i}{t} \int \mu_j^{\sigma} \hat{f_1}(\eta) \partial_\sigma\overline{\hat{f_2}(-\sigma)}\hat{f_3}(-\eta -\sigma) \,d\eta d\sigma \label{eqn:I-j-3}\\
        &+ \frac{i}{t} \int (\mu_j^\eta + \mu_j^\sigma) \hat{f_1}(\eta) \overline{\hat{f_2}(-\sigma)}\partial_\eta\hat{f_3}(-\eta -\sigma) \,d\eta d\sigma \label{eqn:I-j-4}
    \end{align}\end{subequations}
    where $\mu_j$ is the (vector-valued) symbol
    \begin{equation*}
        \mu_j(\eta,\sigma) = \frac{(\eta + \sigma)\psi_j(\eta) \psi_{\leq j}(-\sigma) \nabla_{\eta,\sigma} \phi(0,\eta,\sigma)}{|\nabla_{\eta,\sigma} \phi(0,\eta,\sigma)|^2}
    \end{equation*}
    with components $\mu_j^\eta$ and $\mu_j^\sigma$.  Let us first consider~\eqref{eqn:I-j-1}.  We can rewrite this term as the Fourier transform of a pseudoproduct:
    \begin{equation*}
        \eqref{eqn:I-j-1} = 2\pi i t^{-1} 2^{-2j} \hat{T}_{\mu^1_j}(u_{1,\sim j}, \overline{u_{2, \lesssim j}}, u_{3, \lesssim j})(0)
    \end{equation*}
    with symbol
    \begin{equation*}
        \mu^1_j = 2^{2j} \psi_j(\xi) \nabla_{\eta,\sigma} \cdot \mu_j
    \end{equation*}
    Now, $\mu^1_j$ obeys the symbol bounds
    \begin{equation*}
        |\partial^\alpha_{\xi,\eta,\sigma} \mu_j^1| \lesssim 2^{-|\alpha|j}, \qquad\qquad |\supp \mu_j^1| \lesssim 2^{3j}
    \end{equation*}
    so by~\Cref{rmk:freq-loc-symbol-bounds} the pseudoproduct $T_{\mu_j^1}(\cdot,\cdot,\cdot)$ satisfies H\"older-type bounds uniformly in $j$.  Thus, the Hausdorff-Young inequality gives us the bound
    \begin{equation*}\begin{split}
        |\eqref{eqn:I-j-1}| 
        \lesssim& t^{-1} 2^{-2j} \lVert u_{1,\sim j} \rVert_{L^\infty} \lVert u_{2,\lesssim j} \rVert_{L^2} \lVert u_{3,\lesssim j} \rVert_{L^2}\\
        \lesssim& t^{-3/2} 2^{-3/2 j} \prod_{i=1}^3 \lVert f_i \rVert_X
    \end{split}
    \end{equation*}
    Turning to~\eqref{eqn:I-j-2}, we observe that the support condition implies that the three frequencies $\eta$, $\sigma$, and $- \eta - \sigma$ are $\lesssim 2^j$, and at least two of them must have magnitudes comparable to $2^j$.  Thus, we can write
    \begin{equation*}\begin{split}
        \eqref{eqn:I-j-2} =& 2\pi i t^{-1} 2^{-j} \hat{T}_{\mu^2_j}(Lu_1, \overline{u_{2, \sim j}}, u_{3, \lesssim j})(0) + 2\pi i t^{-1} 2^{-j} \hat{T}_{\mu^3_j}(Lu_1, \overline{u_{2, \ll j}}, u_{3, \sim j})(0)
    \end{split}\end{equation*}
    with
    \begin{equation*}\begin{split}
        \mu_j^2 =& 2^j \psi_j(\xi) m_j^{\eta} \\
        \mu_j^3 =& 2^j \psi_j(\xi) m_j^{\eta}
    \end{split}\end{equation*}
    Using the fact that $\mu_j^2$ and $\mu_j^3$ satisfy the symbol bounds from~\Cref{rmk:freq-loc-symbol-bounds} uniformly in $j$, we find that
    \begin{equation*}\begin{split}
        |\eqref{eqn:I-j-2}| 
        \lesssim& t^{-1} 2^{-j} \lVert Lu_1 \rVert_{L^2} \left(\lVert u_{2,\sim j} \rVert_{L^\infty} \lVert u_{3,\lesssim j} \rVert_{L^2} + \lVert u_{2, \ll j} \rVert_{L^2} \lVert u_{3, \sim j} \rVert_{L^\infty} \right)\\
        \lesssim& t^{-4/3} 2^{-j} \prod_{i=1}^3 \lVert f_i \rVert_X
    \end{split}\end{equation*}
    The estimates for~\eqref{eqn:I-j-3} and~\eqref{eqn:I-j-4} are analogous.  Thus, 
    \begin{equation*}
        \sum_{2^j \gg t^{-1/3}} I_j \lesssim t^{-1} \prod_{i=1}^3 \lVert f_i \rVert_X
    \end{equation*}
    which completes the proof of~\eqref{eqn:cubic-mean-bound}.  
    
    To prove~\eqref{eqn:cubic-mean-bound-mean-0}, we observe that the only place we used the assumption $f_k \in X$ was to obtain bounds of the form
    \begin{equation*}\begin{split}
        \lVert Lu_k \rVert_{L^2} \lesssim& t^{1/6} \lVert f_k \rVert_X\\
        \lVert u_{k, \lesssim j} \rVert_{L^2} \lesssim& 2^{j/2} \lVert f_k \rVert_X\\
        \lVert u_{k, \sim j} \rVert_{L^\infty} \lesssim& t^{-1/2} 2^{-j/2} \lVert f_k \rVert_X
    \end{split}\end{equation*}
    If we instead estimate these quantities in terms of $\lVert Lu_k \rVert_{L^2}$, we gain a factor of $t^{1/6}$ when estimating terms containing $Lu_k$.  For the other terms, we see from~\Cref{thm:w-lin-decay} that
    \begin{equation*}\begin{split}
        \lVert u_{k,\lesssim j} \rVert_{L^2} \lesssim& 2^j \lVert Lu_k \rVert_{L^2}\\
        \lVert u_{k,\sim j} \rVert_{L^\infty} \lesssim& t^{-1/2} \lVert Lu_k \rVert_{L^2}\\
    \end{split}\end{equation*}
    so we gain a factor of $2^{j/2}$ in this case.  Modifying the estimates in light of this, we obtain~\eqref{eqn:cubic-mean-bound-mean-0}.
\end{proof}

\section{Bounds for the self-similar term}\label{sec:self-sim}

As discussed in the introduction, we are interested in the self-similar solutions $S$ given by
\begin{equation*}
    S(x,t;p) = t^{-1/3}\sigma(t^{-1/3}x; p)
\end{equation*}
where $\sigma$ solves the third order ODE
\begin{equation}\label{eqn:cmkdv-self-sim-prof-expanded}
    \partial_y^3 \sigma(y;p) -  \frac{1}{3}y\partial_y \sigma(y;p) - \frac{1}{3}\sigma(y;p) = -|\sigma(y;p)|^2 \partial_y \sigma(y;p)
\end{equation}
subject to the (nonlocal) boundary condition $\hat{\sigma}(0;p) = p$.  Unlike in the real-valued case, it is not obvious that the nonlinearity in~\eqref{eqn:cmkdv-self-sim-prof-expanded} is a derivative, so it is not immediately clear that the solutions $\sigma$ we are interested in must solve the reduced order Painlev\'e equation~\eqref{eqn:phase-rot-Painleve-II}. Let us briefly give a heuristic justification for this reduction.  

Recall that we are interested in solutions to~\eqref{eqn:cmkdv} that decay at the same rate as a solution to the linear problem.  Since these solutions decay rapidly as $y \to \infty$, it is natural to assume that the nonlinearity can be treated perturbatively in this asymptotic regime.  Under this assumption, for large $y$ the solution to~\eqref{eqn:cmkdv-self-sim-prof-expanded} takes the form
\begin{equation*}
    \sigma(y;p) \sim c_1(p) \Ai(3^{-1/2}y) + c_2(p) \operatorname{Bi}(3^{-1/2}y) + c_3(p) \operatorname{Gi}(3^{-1/2}y)\qquad\qquad \text{as } y \to \infty
\end{equation*}
where $\operatorname{Ai}$ and $\operatorname{Bi}$ are the Airy functions of the first and second kind, $\operatorname{Gi}$ is Scorer's function~\cite{scorerNumericalEvaluationIntegrals1950}, and the $c_j(p)$'s are constants.  Since $\operatorname{Bi}(y) \to \infty$ as $y \to \infty$, we immediately conclude that $c_2(p) = 0$.  Since our functional framework requires that $Lu$ remain in $L^2$, we will also impose the condition $LS \in L^2$.  This is equivalent to the requirement that $(\partial_y^2 - \frac{1}{3}y)\sigma(y;p) \in L^2_y$, which is only satisfied if $c_3(p) = 0$.  Thus, we will focus on self-similar solutions to~\eqref{eqn:cmkdv} which are asymptotic to some multiple of $\Ai(3^{-1/2}y)$ at infinity.  This problem has received a great deal of study; see~\cite{deiftAsymptoticsPainleveII1995,hastingsBoundaryValueProblem1980,correiaAsymptoticsFourierSpace2020}.  In particular, from~\cite{correiaAsymptoticsFourierSpace2020} we know that $\sigma$ is the solution to the Painlev\'e II equation with the following boundary conditions:
\begin{equation}\label{eqn:phase-rot-Painleve-II}
    \left\{\begin{array}{c}
        \sigma''(y;p) = \frac{1}{3}y \sigma(y;p) - \frac{1}{3}|\sigma(y;p)|^2 \sigma(y;p)\\
        \hat{\sigma}(0;p) = p\\
        \hat{\sigma}(\eta;p) \text{ is continuous at } \eta=0
    \end{array}\right.
\end{equation}
In the notation of~\cite{correiaAsymptoticsFourierSpace2020}, this corresponds to taking $(c,\alpha) = (3^{-1/2}p,0)$ (to agree with~\cite[Equation (11)]{correiaAsymptoticsFourierSpace2020}, we must take $\sigma = \sqrt{3}V$, so $\hat{V}(0) = 3^{-1/2}p$ and there is no jump across $\eta = 0$).  We note that~\eqref{eqn:phase-rot-Painleve-II} is phase rotation invariant, so $\sigma(y;re^{i\theta}) = e^{i\theta} \sigma(y;r)$ where $r$ and $\theta$ are real.

\subsection{Bounds for \texorpdfstring{$S$}{S}}
Now, we will investigate what bounds on $S$ and $LS$ we can obtain when $\sigma$ satisfies~\eqref{eqn:phase-rot-Painleve-II}.  In this section, we will hold $p$ fixed and assume that it is small enough to apply the results of~\cite{correiaAsymptoticsFourierSpace2020}.  Let $\Phi = e^{-\partial_x^3} \sigma$.  Then, by~\cite[Theorem 1 and Remark 4]{correiaAsymptoticsFourierSpace2020}, the profile $\Phi$ can be written as
\begin{equation}\label{eqn:Phi-rep}
    \hat{\Phi}(\xi;p) = \chi(\xi) e^{ia\ln|\xi|} \left(A^\pm(p) + B^\pm(p) e^{2ia^\pm(p)\ln|\xi|} \frac{e^{-i \frac{8}{9}\xi^3}}{\xi^3} \right) + z(\xi; p) \qquad \text{for } \pm \xi \ge 0
\end{equation}
where $\chi$ is a cut-off function supported on $|\xi| \geq 1$,  $A^\pm, B^\pm,$ and $a^\pm$ are real-valued and have a Lipschitz dependence on $p$ (at least for $p$ sufficiently small), and $z$ is some function which has a Lipschitz dependence on $p$ with respect to the $H^1$ norm.\footnote{In fact, the Lipschitz dependence of $z$ can be proved in stronger spaces, but for us $H^1$ will be sufficient.}

\begin{rmk}
    Although we write~\eqref{eqn:Phi-rep} in the form given by~\cite[Remark 4]{correiaAsymptoticsFourierSpace2020} for simplicity, it is possible to obtain the representation~\eqref{eqn:Phi-rep} by using only Theorem 1 in~\cite{correiaAsymptoticsFourierSpace2020} by taking advantage of the phase rotation invariance of our equation.  Indeed, we see that $\sigma(y;p) = e^{i\theta} \sigma(y;|p|)$ satisfies~\eqref{eqn:phase-rot-Painleve-II} when $e^{i\theta}|p| = p$. The profile for $\sigma(y;|p|)$ is given in~\cite[Theorem 1]{correiaAsymptoticsFourierSpace2020} by
    \begin{equation*}
        \hat{\Phi}(\xi;|p|) = \chi(\xi) e^{ia\ln|\xi|} \left(A(|p|) + B(|p|) e^{2ia(|p|)\ln|\xi|} \frac{e^{-i \frac{8}{9}\xi^3}}{\xi^3} \right) + z(\xi; |p|) \qquad \text{for } \xi \ge 0
    \end{equation*}
    with a similar expression for $\xi < 0$ given by replacing $A$ and $B$ by $\overline{A}$ and $\overline{B}$. 
 Multiplying by $e^{i\theta}$, we obtain the representation~\eqref{eqn:Phi-rep} with
    \begin{align*}
        A^+(p) = e^{i\theta} A(|p|), \qquad & \qquad A^-(p) = e^{i\theta} \overline{A(|p|)}\\
        B^+(p) = e^{i\theta} B(|p|), \qquad & \qquad B^-(p) = e^{i\theta} \overline{B(|p|)}\\
        a^+(p) = a^-(p) = a(|p|), \qquad&\qquad z(\xi,p) = e^{i\theta} z(\xi,|p|)
    \end{align*}
    Furthermore, a geometric argument allows us to verify the Lipschitz dependence of the constants and function $z$ in~\eqref{eqn:Phi-rep}:  For any $p$ and $q$ are complex numbers with $p = |p|e^{i\theta}$ and $q = |q| e^{i\phi}$, we have that
    $$|A^+(p) - A^+(q)| \leq |1 - e^{i(\theta - \phi)}| \min\bigl(|A(|p|)|, |A(|q|)|\bigr) + |A(|p|) - A(|q|)|$$
    Since $A$ is Lipschitz, the second term is bounded by $\left||p| - |q|\right| \leq |p-q|$.  For the first term, we observe that
    $$|1 - e^{i(\theta - \phi)}|^2 = 2 (1-\cos(\theta - \phi))$$
    so, by the Law of Cosines,
    \begin{align*}
        |p-q|^2 =& |p|^2 + |q|^2 - 2|p||q| \cos(\theta - \phi)\\
        \geq& 2|p||q| (1 - \cos(\theta - \phi))\\
        \geq& \min(|p|,|q|)^2 |1 - e^{i(\theta -\phi)}|^2
    \end{align*}
    Since $A(0) = 0$, we see that $\min\left(|A(|p|)|, |A(|q|)|\right) \leq \min(|p|,|q|)$, so
    $$|1 - e^{i(\theta - \phi)}| \min\left(|A(|p|)|, |A(|q|)|\right) \lesssim |p-q|$$
    which completes the proof that $A^+$ is Lipschitz.  Similar arguments apply to the other constants and to the function $z$.
\end{rmk}

Using the representation~\eqref{eqn:Phi-rep} for $\Phi$, we see that 
$$h(x, t;p) := e^{t\partial_x^3} S(x,t;p) = t^{-1/3}\Phi(t^{-1/3} x;p)$$ satisfies 
$$\lVert h \rVert_{X} \lesssim |p|$$
Since $h \in X$ and $S = e^{-t\partial_x^3}h$, the pointwise behavior of $S$ is given by~\Cref{thm:lin-decay-lemma}.  For instance,~\eqref{eqn:lp-decay} gives us the bound
\begin{equation}\label{eqn:self-sim-lq-bd}
    \lVert S(x,t;p) \rVert_{L^q} \lesssim |p| t^{-\frac{1}{3} + \frac{1}{3q}}
\end{equation}
for all $q > 4$.  In~\Cref{sec:boot-consequences}, we will list refined bounds for solutions that are localized in space and/or frequency.

Turning to $LS$, we note that~\eqref{eqn:phase-rot-Painleve-II} implies the identity
\begin{equation}\label{eqn:self-similar-L-identity}
    LS = 3t|S|^2 S
\end{equation}
Using~\eqref{eqn:self-sim-lq-bd} with $q = 6$, we obtain the $L^2$ estimate 
\begin{equation*}
    \lVert LS \rVert_{L^2} \lesssim t \lVert S \rVert_{L^6}^3 \lesssim |p|^3 t^{1/6}
\end{equation*}
which improves on the trivial bound $\lVert LS \rVert_{L^2} \lesssim t^{1/6}\lVert h \rVert_X$ for $|p|$ small.  Similarly, by~\eqref{eqn:cmkdv-self-sim-prof-expanded}
\begin{equation}\label{eqn:self-similar-dx-L-identity}
    \partial_x LS = 3 |S|^2 \partial_x S
\end{equation}
Since~\eqref{eqn:self-similar-L-identity} and~\eqref{eqn:self-similar-dx-L-identity} allow us to express $LS$ and its derivative in terms of $S$, it is possible to obtain pointwise estimates for $LS$ and $\partial_x LS$.  We will defer this derivation until~\Cref{sec:boot-consequences}.

\subsection{Modulating in \texorpdfstring{$p$}{p}}
Recall that we are interested in self-similar solutions $S(x,t; p(t))$, where $p(t) =\hat{u}(0,t)$.  Since $\hat{u}(0,t)$ is not a conserved quantity of~\eqref{eqn:cmkdv}, we must understand how $S$ -- or, equivalently, $h$ -- changes with $p$.  As in the previous section, we will always assume that $p$ lies within some neighborhood $U$ of $0$ that is small enough that the results of~\cite{correiaAsymptoticsFourierSpace2020} apply.  To avoid cumbersome notation, we will write ``for almost every $p$'' or ``$L^\infty$ in $p$'' with the implicit understanding that we are restricting to $p \in U$.

The nonlinearity of~\eqref{eqn:phase-rot-Painleve-II} is not analytic, so we cannot expect $\hat{h}$ to be complex differentiable in $p$.  Moreover, the results of~\cite{correiaAsymptoticsFourierSpace2020} will only give us Lipschitz continuity in the parameter $p$.  However, writing $p = u + iv$, we can show that for $\hat{h}(\xi,t; u + iv)$ has a bounded derivative with respect to $u$ and $v$ for almost every $(u,v)$ in a neighborhood of the origin.  This amounts to identifying $\bbC$ with $\bbR^2$ and computing the derivative in the $\bbR^2 \to \bbC$ sense, which is sufficient for our purposes.  In particular, for $p: \bbR \to \bbC$ absolutely continuous, we have that
\begin{equation}\label{eqn:d-p-sigma-def}
    \partial_s \hat{h}(\xi,t;p(s)) = \partial_{\Re p} \hat{h} \Re p'(s) + i\partial_{\Im p} \hat{h} \Im p'(s) =: D_p \hat{h} p'(s)
\end{equation}
where we have abused notation slightly by interpreting $D_p \hat{h}$ as the derivative of a function from $\bbR^2$ to $\bbC$ and $p$ as a function from $\bbR$ to $\bbR^2$.

We now briefly comment on the nature of the equality in~\eqref{eqn:d-p-sigma-def}, and explain why it is sufficient for our purposes.  Since $h$ only has a Lipschitz dependence on $p$, the derivative $D_p \hat{h}$ is only defined for almost every $p$.  To handle this difficulty, we will use a chain rule for Lipschitz functions:
\begin{lemma}\label{lem:Lip-chain-rule}
    Suppose $f: \bbC \to \bbC$ is Lipschitz and $p: \Omega \to \bbC$ is absolutely continuous, where $\Omega \subset \bbR$.  Then, for almost every $s \in \Omega$, 
    \begin{equation}\label{eqn:Lip-chain-rule}\frac{d}{ds} f(p(s)) = D_p f(p(s)) p'(s)\end{equation}
    with the convention that the product is zero whenever $p'(s) = 0$.
\end{lemma}
\begin{proof}
    We will obtain this theorem as a consequence of the more general chain rule given in~\cite[Theorem 2.1]{ambrosioGeneralChainRule1990}.  Taking $u = p$, we see that $u$ is of bounded variation with no jumps, so the theorem of Ambrosio and Dal Maso implies that as measures,
    $$D(f(p(s))) = D_p \left(\left.f\right|_{T^p_s}\right)(p(s)) \frac{Dp}{|Dp|} |Dp|$$
    (where we have identified $\bbC$ with $\bbR^2$ in order to apply the theorem).  Noting that in our context $T^p_s = \bbR^2 \sim \bbC$ whenever $p$ is differentiable at $s$ and that $Dp = p'(s)\;ds$, we see that $\supp \left(D(f\circ p)\right) \subset \supp |Dp|$, and for $|Dp|$-a.e. $s$,
    $$D(f(p(s))) = D_p f(p(s)) p'(s)\;ds$$
    It follows that~\eqref{eqn:Lip-chain-rule} holds in an $L^1$ sense, and thus almost everywhere.
\end{proof}
\Cref{lem:Lip-chain-rule} tells us that if $p'$ is integrable, then $\partial_s \hat{h}(x,t;p(s))$ will also be integrable.  This allows us to apply modulation theory even though $\hat{h}$ is merely Lipschitz in $p$.

With these technical points in mind, let us compute $D_p \hat{h}(\xi,t;p)$.  To do so, we observe that $\hat{h}(\xi,t;p) = \hat{\Phi}(t^{1/3}\xi; p)$, and that $\hat{\Phi}$ is given by~\eqref{eqn:Phi-rep}.  Since $A^\pm, B^\pm, a^\pm$ and $z$ are all Lipschitz continuous in $p$, we find that
\begin{equation}\label{eqn:D-p-Phi-calc}\begin{split}
    D_p \hat{\Phi}(\xi; p) =& i D_p a^+ \chi(\xi) e^{ia\ln|\xi|
} \ln |\xi| \left(A^+ + 3 e^{2ia\ln|\xi|
}B^+ \frac{e^{-i\frac{8}{9}\xi^3}}{\xi^3}\right)\\
    &+ \chi(\xi) e^{ia\ln|\xi|
} \ln |\xi| \left(D_p A^+ + e^{2ia\ln|\xi|
} D_p B^+ \frac{e^{-i\frac{8}{9}\xi^3}}{\xi^3}\right)\\
    &+ D_p z(\xi)
\end{split}\end{equation}
for $\xi > 0$, with a similar expression involving $A^-$, $B^-$ and $a^-$ holding for $\xi < 0$.  Examining~\eqref{eqn:D-p-Phi-calc} and recalling that $D_p z \in H^1 \subset C^0$, we see that the dominant contribution comes from the first term in~\eqref{eqn:D-p-Phi-calc} when $|\xi| > 1$, giving us the pointwise bound
\begin{equation}\label{eqn:D-p-Phi}
    |D_p \hat{\Phi}(\xi;p)| \lesssim \ln(2 + |\xi|)
\end{equation}
Similarly, differentiating~\eqref{eqn:D-p-Phi-calc} in $\xi$, we see that all the terms other than $\partial_\xi D_p z$ decay like $\frac{\ln 2+|\xi|}{\jBra{\xi}}$ or better, while $\partial_\xi D_p z \in L^2$, so
\begin{equation}\label{eqn:D-p-dPhi}
    \lVert D_p \partial_\xi \hat{\Phi}(\xi;p) \rVert_{L^2} \lesssim 1
\end{equation}
Using these two estimates, we find that
\begin{equation*}\begin{split}
    \lVert \widehat{Q_j D_p h}(\xi,t;p) \rVert_{L^\infty_{\xi,p}} =& \lVert \widehat{Q_j D_p \Phi}(t^{1/3}\xi;p) \rVert_{L^\infty_{\xi,p}} \lesssim \ln(2 + t^{1/3}2^j)\\
    \lVert D_p\partial_\xi \hat{h}(\xi,t;p) \rVert_{L^\infty_{p}L^2_\xi} =& \lVert \partial_\xi D_p \hat{\Phi}(t^{1/3}\xi; p) \rVert_{L^\infty_p L^2_\xi} \lesssim t^{1/6}
\end{split}\end{equation*}
so
\begin{equation*}
    \lVert D_p h(\cdot;p) \rVert_{L^\infty_p X_j} \lesssim \ln(2 + t^{1/3} 2^j)
\end{equation*}
\begin{rmk}
    Note that by~\eqref{eqn:D-p-Phi-calc}, the estimates~\eqref{eqn:D-p-Phi} and~\eqref{eqn:D-p-dPhi} hold in $L^\infty_p$ (the derivatives of $A^\pm$, $B^\pm$, and $a^\pm$ exist for a.e. $p$ by Rademacher's theorem, while $D_p z \in H^1$ exists for a.e. $p$ because $H^1$ has the Radon-Nikodym property~\cite{benyaminiGeometricNonlinearFunctional2000}).  In particular, the above $L^\infty_p$ norms are well defined.
\end{rmk}

Since $D_pS = e^{t\partial_x} D_p h$,~\Cref{thm:lin-decay-lemma} gives us the pointwise decay estimate
\begin{equation*}\begin{split}
    |Q_{\sim j} D_p S | \lesssim& t^{-1/2} 2^{-j/2} \ln(2 + t^{-1/3} 2^j) \chi_{\sim j}\\
    &+ t^{-5/6}\ln(2 + t^{-1/3} 2^j)\left(2^j + 2^{-j/3} |\xi_0(x)|^{4/3}\right)^{-3/2}
\end{split}\end{equation*}
where $\xi_0 = \xi_0(x)$ is as in~\Cref{thm:lin-decay-lemma}.  It follows that
\begin{equation}\label{eqn:D-p-S-freq-loc-L-infty}
   \lVert Q_{\sim j} D_p S \rVert_{L^\infty} \lesssim t^{-1/2} 2^{-j/2} \ln(2 + t^{1/3} 2^j)
\end{equation}
and, for $4 < q < \infty$
\begin{equation}\label{eqn:D-p-S-L-p-freq-loc}
    \lVert Q_{\sim j}D_p S \rVert_{L^q} \lesssim t^{-\frac{1}{2} + \frac{1}{q}} 2^{\left(-\frac{1}{2} + \frac{2}{q}\right)j} \ln(2 + t^{1/3} 2^j)
\end{equation}
which can be summed to give
\begin{equation}\label{eqn:D-p-S-L-p}
    \lVert D_p S \rVert_{L^q} \lesssim t^{-\frac{1}{3} + \frac{1}{3q}}
\end{equation}
Finally, we observe that $L$ and $D_p$ commute, so using~\eqref{eqn:self-similar-L-identity}, we find that
\begin{equation}\label{eqn:param-deriv-LS}
    \lVert L D_p S \rVert_{L^2} = \lVert 3tD_p(|S|^2 S) \rVert_{L^2} \lesssim t \lVert S \rVert_{L^6}^2 \lVert D_p S \rVert_{L^6} \lesssim t^{1/6}|p|^2
\end{equation}

\section{Reduction of the main theorem}\label{sec:reduction}
\subsection{Reduction of the main theorem to profile estimates}

Let 
\begin{equation*}
    w(x,t) = u(x,t) - S(x,t; \hat{u}(0,t))
\end{equation*}
and define $g = e^{t\partial_x^3} w$.  The remainder of the paper will be devoted to proving the nonlinear bounds on $f$ and $g$ given in the following theorem:
\begin{theorem}\label{thm:nonlinear-bounds-thm}
    There exists an $\epsilon_0 > 0$ such that if $u_* \in H^2$ and $\lVert \hat{u}_* \rVert_{L^\infty} + \lVert x u_* \rVert_{L^2} \leq \epsilon \leq \epsilon_0$, then the solution $u$ to~\eqref{eqn:cmkdv-t-1} is global, and the following bounds hold for all $t \in [1,\infty)$
    \begin{equation}\label{eqn:desired-g-bound}
        \lVert xg(t) \rVert_{L^2} \lesssim \epsilon t^{1/6-\beta}
    \end{equation}
    \begin{equation}\label{eqn:desired-f-hat-bound}
        \lVert \hat{f}(t) \rVert_{L^\infty} \lesssim \epsilon
    \end{equation}
    where $\beta = \beta(\epsilon) = \frac{1}{6} - C \epsilon^2$ for some constant $C$.  Moreover,
    \begin{equation}\label{eqn:desired-zero-mode-conv}
        |\partial_t\hat{u}(0,t)| \lesssim \epsilon^3 t^{-1-\beta}
    \end{equation}
    and there exists a bounded function $f_\infty(\xi)$ such that \begin{equation}\label{eqn:phase-rot-dynamics-f-hat}\hat{f}(\xi,t) = \exp\left( -\frac{i}{6}\sgn \xi \int_1^t \frac{|\hat{f}(\xi,s)|^2}{s}\,ds\right)f_\infty(\xi) + O(\epsilon^3 (t^{-1/3} |\xi|)^{-1/14})\end{equation}
\end{theorem}

Assuming~\Cref{thm:nonlinear-bounds-thm}, we can prove~\Cref{thm:main-theorem}.
\begin{proof}[Proof of~\Cref{thm:main-theorem}]
    We first show that $f$ satisfies the hypotheses of~\Cref{thm:lin-decay-lemma}.  By~\eqref{eqn:desired-f-hat-bound}, $|\hat{u}(0,t)| = |\hat{f}(0,t)| \lesssim \epsilon$, so $\lVert LS \rVert_{L^2} \lesssim \epsilon^3t^{1/6}$ by~\eqref{eqn:self-similar-L-identity}.  Combining this with the bound for $xg$, we see that  
    \begin{equation*}
        \lVert xf \rVert_{L^2} \leq \lVert xg \rVert_{L^2} + \lVert LS \rVert_{L^2} \lesssim \epsilon t^{1/6}
    \end{equation*}
    Recalling the $\mathcal{F}L^\infty$ bound given in~\eqref{eqn:desired-f-hat-bound}, we see that $\lVert f \rVert_{X} \lesssim \epsilon$, which is enough to give the asymptotics~\eqref{eqn:positive-x-asymp} for $x > t^{1/3}$ using~\Cref{thm:lin-decay-lemma}.  By using the more precise expression for $\hat{f}$ given in~\eqref{eqn:phase-rot-dynamics-f-hat}, we see that $u$ has the modified scattering asymptotics given by~\eqref{eqn:negative-x-asymp} in the region $x < -t^{-1/3}$.
    
    It only remains to verify that the asymptotics for $|x| \lesssim t^{1/3 + 4\beta}$ are given by~\eqref{eqn:small-x-asymptotics}.  By~\Cref{thm:w-lin-decay} and the hypothesis~\eqref{eqn:desired-g-bound}, $\lVert w \rVert_{L^\infty} \lesssim \epsilon t^{-1/3 -\beta}$, so in the region $|x| \lesssim t^{1/3 +4\beta}$
    \begin{equation}\label{eqn:u-S-error}
        u(x,t) = S(x,t; \hat{u}(0,t)) + O(\epsilon t^{-1/3 -\beta})
    \end{equation}
    Moreover, since $t^{-1-\beta}$ is integrable over $[1,\infty)$,~\eqref{eqn:desired-zero-mode-conv} implies that $\alpha = \lim_{t\to\infty} \hat{u}(0,t)$ has size $O(\epsilon)$ and satisfies $|\hat{u}(0,t) - \alpha| = O(\epsilon^3 t^{-\beta})$.
    Using~\eqref{eqn:D-p-S-freq-loc-L-infty} to bound the terms $\lVert Q_j D_p S \rVert_{L^\infty}$, we find that
    \begin{equation*}\begin{split}
        \lVert S(x,t; \hat{u}(0,t)) - S(x,t;\alpha)\rVert_{L^\infty} \lesssim& \sum_{2^j \gtrsim t^{-1/3}} \lVert Q_{j} D_p S \rVert_{L^\infty} |\hat{u}(0,t) - \alpha|\\
        \lesssim& \epsilon^3 t^{-1/3-\beta}
    \end{split}\end{equation*}
    which, combined with~\eqref{eqn:u-S-error}, gives the self-similar asymptotics~\eqref{eqn:small-x-asymptotics}.
\end{proof}

\subsection{Plan of the proof of~\texorpdfstring{\Cref{thm:nonlinear-bounds-thm}}{Theorem 9}}

We will use a bootstrap argument to prove~\Cref{thm:nonlinear-bounds-thm}.  We begin with some qualitative observations about the local wellposedness of~\eqref{eqn:cmkdv-t-1}.    By~\cite{katoCauchyProblemGeneralized1983},~\eqref{eqn:cmkdv-t-1} has a local solution on $[1,1+\delta]$ for some $\delta > 0$ such that $u(t)$ is continuous in $H^2_x$ and $xu(t)$ is continuous in $L^2_x$.  It follows that $f(t)$ is continuous in $X$ over $[1, 1+\delta]$.  \Cref{thm:cubic-mean-thm} then implies that
\begin{equation*}
    \partial_t \hat{u}(0,t) = -\frac{1}{\sqrt{2\pi}} \int |u|^2 \partial_x u \,dx \in C(1, 1+\delta)
\end{equation*}
with
\begin{equation*}
    |\partial_t \hat{u}(0,0)| = \frac{1}{\sqrt{2\pi}} \left|\int|u_*|^2\partial_x u_*\,dx\right| \lesssim \epsilon^3
\end{equation*}
Thus, for $T$ sufficiently close to $1$, the following bootstrap hypotheses hold:
\begin{equation}\label{eqn:bootstrap-hypotheses}\tag{BH}
    \sup_{1 \leq t \leq T}\left(\lVert \hat f(t) \rVert_{L^\infty} + t^{-1/6}\lVert x f(t) \rVert_{L^2}\right) \leq M \epsilon, \qquad \sup_{1 \leq t \leq T}|\partial_t \hat{u}(0,t)| \leq M^3\epsilon^3 t^{-1-\beta}
\end{equation} 
where $M \gg 1$ is a large constant independent of $u_*$, the value of which we will specify later.  Fix an $\epsilon_0 \ll M^{-3/2}$ independent of $u_*$, and suppose that $\epsilon \leq \epsilon_0$.  Using~\eqref{eqn:bootstrap-hypotheses}, we prove in~\Cref{sec:weighted-L2} that $\lVert xg \rVert_{L^2} \leq C\epsilon t^{1/6-\beta}$ for some $C$ independent of $M$.  Then, by using this bound on $xg$ in addition to the bootstrap hypotheses, we verify that $| \partial_t \hat{u}(0,t)| \leq CM^2\epsilon^3 t^{-1-\beta}$ and that $\lVert \hat{f}(t) \rVert_{L^\infty} \leq C \epsilon$ in~\Cref{sec:L-infty-est}, where again the constants $C$ do not depend on $M$.  These results imply the improved bounds
\begin{equation}\label{eqn:bootstrap-hypos-improved}
    \sup_{1 \leq t \leq T}\left(\lVert \hat f(t) \rVert_{L^\infty} + t^{-1/6}\lVert x f(t) \rVert_{L^2}\right) \leq C \epsilon, \qquad \sup_{1 \leq t \leq T}|\partial_t \hat{u}(0,t)| \leq CM^2\epsilon^3 t^{-1-\beta} \tag{BH+}
\end{equation}
In particular, since we are free to choose $M$, we may choose $M > C$, so that~\eqref{eqn:bootstrap-hypos-improved} are stronger than the original bootstrap hypotheses in~\eqref{eqn:bootstrap-hypotheses}.  Moreover, a simple energy estimate shows that
\begin{equation*}\begin{split}
    \frac{d}{dt} \lVert \partial_x^2 u \rVert_{L^2}^2 \lesssim& \lVert u \partial_x u \rVert_{L^\infty}  \lVert \partial_x^2 u \rVert_{L^2}^2 \lesssim M^2\epsilon^2 t^{-1} \lVert \partial_x^2 u \rVert_{L^2}^2
\end{split}\end{equation*}
so, by Gr\"onwall's inequality, $\lVert \partial_x^2 u \rVert_{L^2}$ grows at most polynomially in time.  Using the $L^\infty$ bound on $\hat{f}$ to control the low frequencies, we see that $\lVert u \rVert_{H^2}$ does not blow up at time $T$.  Since the results of~\cite{katoCauchyProblemGeneralized1983} imply that the solution $u$ can be continued until $\lVert xu \rVert_{L^2} + \lVert u \rVert_{H^2}$ blows up, we can extend to a solution to a longer time interval $[1, T']$ such that the bootstrap bounds~\eqref{eqn:bootstrap-hypotheses} hold up to time $T'$.  By a standard continuity argument, this shows that the estimates~\eqref{eqn:bootstrap-hypos-improved} hold for all time.  Moreover, in the course of proving the $\mathcal{F}L^\infty$ bound in~\Cref{sec:L-infty-est}, we obtain~\eqref{eqn:phase-rot-dynamics-f-hat}, which proves~\Cref{thm:nonlinear-bounds-thm} (and hence~\Cref{thm:main-theorem}, as well).

\subsection{Basic consequences of the bootstrap estimates\label{sec:boot-consequences}}

We close this section by listing some basic estimates that follow from the bootstrap assumptions~\eqref{eqn:bootstrap-hypotheses} and the material in~\Cref{sec:linear-ests}.  

We begin by discussing the estimates for $u$.  By the first bootstrap assumption, for $t \in [1, T]$,
\begin{equation*}
    \lVert f(t) \rVert_X \lesssim M\epsilon
\end{equation*}
Thus,~\Cref{thm:lin-decay-lemma} gives us the following bounds:
\begin{equation}\label{eqn:u-lin-ests}\begin{split}
    \lVert u_j \rVert_{L^\infty} \lesssim& M\epsilon t^{-1/2}2^{-j/2}\\
    \lVert \chi_{\geq k} u \rVert_{L^\infty} \lesssim& M\epsilon t^{-1/2} 2^{-k/2}\\
    \lVert \chi_{\geq j} u_{\ll j} \rVert_{L^\infty} \lesssim& M\epsilon t^{-5/6} 2^{-3/2 j}\\
    \lVert (1 - \chi_{\sim j}) u_j \rVert_{L^\infty} \lesssim& M\epsilon t^{-5/6} 2^{-3/2 j}\\
    \lVert u_{j} \rVert_{L^4} + \lVert \chi_{j} u\rVert_{L^4} \lesssim& M \epsilon s^{-1/4}\\
    \lVert (1 - \chi_{\sim j}) u_j \rVert_{L^4} \lesssim& M \epsilon s^{-7/12} 2^{-j}\\
    \lVert \chi_{> j} u_{\ll j} \rVert_{L^4} \lesssim& M\epsilon s^{-7/12} 2^{-j}
\end{split}\end{equation}

We now consider the bounds for the self-similar solution $S = S(x,t;\hat{u}(0,t))$.  From the bootstrap assumptions and the fact that $\epsilon \ll M^{-3/2}$, we find that 
\begin{equation*}
    \sup_{1 \leq t \leq T} |\hat{u}(0,t)| \leq \epsilon + \int_1^\infty M^3\epsilon^3 s^{-1-\beta} \,ds    \lesssim \epsilon
\end{equation*}
It follows from the work in~\Cref{sec:self-sim} that $S$ obeys all the estimates in~\eqref{eqn:u-lin-ests} but without the bootstrap factor $M$.  In addition, by combining the identity~\eqref{eqn:self-similar-L-identity} with the cubic estimates from~\Cref{sec:cubic-bounds}, we find that
\begin{equation}\label{eqn:LS-cubic-bounds}\begin{split}
    \lVert (LS)_{\sim j} \rVert_{L^\infty} \lesssim& \epsilon^3 t^{-1/2} 2^{-j/2}\\
    \lVert \chi_{k} (LS)_{\sim j} \rVert_{L^\infty} \lesssim& \epsilon^3 t^{-5/6}2^{-3/2 j} 2^{-k} \qquad k < j - 30\\
    \lVert \chi_{k} (LS)_{\ll j} \rVert_{L^\infty} \lesssim& \epsilon^3 t^{-1/2} 2^{-3/2 k}\\
    \lVert \chi_{\geq k} \partial_x (LS)_{\ll j} \rVert_{L^\infty} \lesssim& \epsilon^3 t^{-1/2} 2^{-3/2 k}\\
    \lVert \partial_x (LS)_{\sim j} \rVert_{L^\infty} \lesssim& \epsilon^3 t^{-1/2} 2^{-j/2}\\
    \lVert \partial_x LS \rVert_{L^\infty} \lesssim& \epsilon^3 t^{-1/3}
\end{split}\end{equation}
All but the last equation are straightforward consequences of the cubic bounds \cref{eqn:cubic-sim-larger,eqn:cubic-bounds-compendium,eqn:cubic-deriv-bounds-compendium,eqn:cubic-sim-with-k,eqn:cubic-sim-w-deriv,eqn:cubic-sim-w-deriv-small}, and the last inequality follows from the second-to-last after summing in $j$.

Finally, we turn to the linear estimates for $w$.  By~\Cref{thm:w-lin-decay}, these estimates can be given in terms of $\lVert xg \rVert_{L^2}$:
\begin{equation}\label{eqn:w-lin-ests-L-inf}\begin{split}
    \lVert w_j \rVert_{L^\infty} \lesssim& t^{-1/2} \lVert xg \rVert_{L^2} c_j\\
    \lVert \chi_k w \rVert_{L^\infty} \lesssim& t^{-1/2} \lVert xg \rVert_{L^2} c_k\\
    \lVert \chi_k \partial_x w \rVert_{L^\infty} \lesssim& t^{-1/2} 2^k \lVert xg \rVert_{L^2} c_k\\
    \lVert \chi_k \partial_x w_{j} \rVert_{L^\infty} \lesssim& t^{-5/6} \lVert xg \rVert_{L^2} c_j \qquad\qquad k < j - 30
\end{split}
\end{equation}
The first two equations are simple restatements of~\Cref{thm:w-lin-decay}, while the last two follow from applying~\Cref{thm:w-lin-decay} to $\partial_x g$.  We will also often make use of $L^2$ estimates for $w$, so we record some below for future reference:
\begin{equation}\label{eqn:w-lin-ests-L-2}\begin{split}
    \lVert w_{\leq j} \rVert_{L^2} \lesssim& 2^j \lVert xg \rVert_{L^2}\\
    \lVert w_{j} \rVert_{L^2} \lesssim& 2^j \lVert xg \rVert_{L^2}c_j\\
    \lVert \partial_x w_{\leq j} \rVert_{L^2} \lesssim& 2^{2j} \lVert xg \rVert_{L^2}\\
    \lVert \partial_x w_{j} \rVert_{L^2} \lesssim& 2^{2j} \lVert xg \rVert_{L^2} c_j\\
    \lVert \chi_{k} w \rVert_{L^2} \lesssim& 2^k \lVert xg \rVert_{L^2} c_k\\
    \lVert \chi_k \partial_x w \rVert_{L^2} \lesssim& 2^{2k} \lVert xg \rVert_{L^2} c_k\\
    \lVert \chi_k \partial_x w_{j} \rVert_{L^2} \lesssim& t^{-1/3}2^{k} \lVert xg \rVert_{L^2}c_j \qquad\qquad k < j - 30
\end{split}
\end{equation}
The first four estimates follow from Hardy's inequality, and the other inequalities are direct consequences of the $L^\infty$ estimates~\eqref{eqn:w-lin-ests-L-inf} and the estimate $\lVert \chi_{\sim k} \rVert_{L^2} \lesssim t^{1/2} 2^k$.

\begin{rmk}\label{rmk:proj-bdds-rmk} Observe that all of the above bounds are based estimates on the linear propagator.  Since the linear propagator is well-behaved under the Littlewood-Paley projectors, all of the above estimates continue to hold if we replace $u$, $S$, or $w$ with their Littlewood-Paley projections. For instance, we get the bound $\lVert \chi_k u_{\lesssim j} \rVert_{L^\infty} \lesssim M\epsilon t^{-1/2} 2^{-k/2}$ from the second bound in~\eqref{eqn:u-lin-ests}.
\end{rmk}

\section{The weighted energy estimate}\label{sec:weighted-L2}

In this section, we will show that $\lVert xg \rVert_{L^2} \lesssim \epsilon t^{1/6 - \beta}$, where $\beta$ is as in~\Cref{thm:nonlinear-bounds-thm}.  To establish this bound, we show that the bootstrap hypotheses imply
\begin{equation}\label{eqn:xg-desired-bound}
    \lVert xg(t) \rVert_{L^2}^2  \lesssim \epsilon^2  + \int_1^t\left[  M^2\epsilon^2s^{-1}\lVert xg(s) \rVert_{L^2}^2 + M^2\epsilon^3 s^{-5/6-\beta} \lVert xg(s) \rVert_{L^2}\right]\,ds
\end{equation}
for all $t \leq T$.  By adding a factor of $\epsilon^2 t^{1/3 - 2\beta}$,~\eqref{eqn:xg-desired-bound} implies that
\begin{equation*}
    \lVert xg(t) \rVert_{L^2}^2 + \epsilon^2 t^{1/3 -2\beta} \lesssim \epsilon^2 + \int_1^t M^2\epsilon^2 s^{-1} \left(\lVert xg(s) \rVert_{L^2}^2 + \epsilon^2 s^{1/3 -2\beta} \right)\,ds
\end{equation*}
(recall that by the definition of $\beta$, $\frac{1}{6} - \beta = O(M^2\epsilon^2)$).  Applying Gr\"onwall's inequality, we obtain the desired bound for $\lVert xg(t) \rVert_{L^2}$.  To prove~\eqref{eqn:xg-desired-bound}, we write the inequality in differential form using the expansion
\begin{equation}\label{eqn:xg-division}
    \begin{split}
        x \partial_t g  =&  -x e^{t\partial_x^3} \left(|u|^2 \partial_x w +  (w \overline{u} + u \overline{w}) \partial_x S\right) +  xe^{t\partial_x^3}D_p S \partial_t \hat u(0,t)\\
        =& \mathcal{F}^{-1} \frac{1}{2\pi}\partial_\xi \int e^{it\phi} (\xi - \eta - \sigma) \left(\hat{f}(\eta) \overline{\hat{f}(-\sigma)} \hat{g}(\xi-\eta-\sigma)\right.\\
        &\qquad\qquad\qquad + \left.\left(\hat{f}(\eta) \overline{\hat{g}(-\sigma)}+ \hat{g}(\eta) \overline{\hat{f}(-\sigma)}\right) \hat{h}(\xi-\eta-\sigma)\right)\,d\eta d\sigma\\
        & +  e^{t\partial_x}LD_p S \partial_t \hat u(0,t)\\
        = & it\left(T_{\partial_\xi \phi e^{it\phi}} (f,\partial_x g, \overline{f}) + T_{\partial_\xi \phi e^{it\phi}} (f,\partial_x h, \overline{g}) + T_{\partial_\xi \phi e^{it\phi}} (g,\partial_x h, \overline{f})\right)\\
        & + e^{t\partial_x^3} \left(|u|^2 w + (u\overline{w} + \overline{u}w) S\right) + e^{t\partial_x^3} \left( |u|^2 \partial_x Lw + (u\overline{w} + \overline{w}u) \partial_x LS \right)\\
        & +  e^{t\partial_x}LD_p S \partial_t \hat u(0,t)
    \end{split}
\end{equation}
Thus,
\begin{subequations}
\begin{align}
    \begin{split}\frac{\partial_t \lVert xg \rVert_{L^2}^2}{2} =&  -t\Im \bigl\langle xg, T_{\partial_\xi \phi e^{it\phi}} (f,\partial_x g, \overline{f}) + T_{\partial_\xi \phi e^{it\phi}} (f,\partial_x h, \overline{g})\\
    &\qquad+ T_{\partial_\xi \phi e^{it\phi}} (g,\partial_x h, \overline{f}) \bigr\rangle\end{split}\label{eqn:xf-energy-estimate-paraprods}\\
    & +  \Re \langle Lw, |u|^2 w + (u\overline{w} + \overline{u}w) S \rangle\label{eqn:xf-energy-estimate-no-deriv}\\
    &+ \Re \langle Lw, |u|^2 \partial_x Lw \rangle\label{eqn:xf-energy-estimate-cancellation}\\
    &+ \Re \langle Lw, (u\overline{w} + \overline{w}u) \partial_x LS) \rangle \label{eqn:xf-energy-estimate-S-ident}\\
    & + \Re\langle Lw,  LD_p S \partial_t \hat u(0,t) \rangle \label{eqn:xf-energy-estimate-modulation-term}
\end{align}
\end{subequations}
Since
\begin{equation}\label{eqn:xg-t-1-bd}
    \lVert xg(x,1) \rVert_{L^2} \leq \lVert x u_* \rVert_{L^2} + \lVert LS(x,1; \hat{u}_*(0)) \rVert_{L^2} \lesssim \epsilon
\end{equation}
the desired inequality for $\lVert xg \rVert_{L^2}^2$ will follow if we can show that~\cref{eqn:xf-energy-estimate-cancellation,eqn:xf-energy-estimate-modulation-term,eqn:xf-energy-estimate-no-deriv,eqn:xf-energy-estimate-paraprods,eqn:xf-energy-estimate-S-ident} satisfy bounds compatible with~\eqref{eqn:xg-desired-bound}.  We will first show that the terms~\cref{eqn:xf-energy-estimate-no-deriv,eqn:xf-energy-estimate-cancellation,eqn:xf-energy-estimate-S-ident} decay in time like $M^2\epsilon^2 t^{-1} \lVert xg \rVert_{L^2}^2$.  For~\eqref{eqn:xf-energy-estimate-no-deriv}, we use the bounds from~\eqref{eqn:u-lin-ests} and~\eqref{eqn:w-lin-ests-L-2} together with almost orthogonality in space to find that
\begin{equation*}\begin{split}
    \left|\eqref{eqn:xf-energy-estimate-no-deriv}\right| \lesssim& \lVert xg \rVert_{L^2} \left(\sum_{2^k \gtrsim t^{-1/3}} \lVert \chi_k(|u|^2 w)\rVert_{L^2}^2\right)^{1/2} + \{\text{similar terms}\}\\
    \lesssim& M^2 \epsilon^2 t^{-1}\lVert xg \rVert_{L^2}^2 
\end{split}\end{equation*}
For~\eqref{eqn:xf-energy-estimate-cancellation},  we integrate by parts and use~\Cref{thm:simple-bilinear-decay} to find that
\begin{equation*}
    \begin{split}
        |\eqref{eqn:xf-energy-estimate-cancellation}| =& \left|\int \Re (e^{-t\partial_x^3} xg \partial_x \overline{e^{-t\partial_x^3} xg}) |u|^2\,dx\right|\\
        =& \frac{1}{2} \left|\int \partial_x(|u|^2) |e^{s\partial_x^3} xg|^2\,dx\right|\\
        \lesssim& M^2\epsilon^2 t^{-1} \lVert xg \rVert_{L^2}^2
    \end{split}
\end{equation*}  
Finally, for~\eqref{eqn:xf-energy-estimate-S-ident}, we use equation~\eqref{eqn:self-similar-dx-L-identity} to write $\partial_x LS = - 3t|S|^2\partial_x S$.  Then, since $h \in X$, we can use the bilinear bound given in~\Cref{thm:simple-bilinear-decay} and an argument similar to the one for~\eqref{eqn:xf-energy-estimate-no-deriv} to obtain the bound
\begin{equation*}\begin{split}
    \left|\eqref{eqn:xf-energy-estimate-S-ident}\right| \lesssim& t\lVert xg \rVert_{L^2} \lVert uw|S|^2 \partial_x S \rVert_{L^2}\\
    \lesssim& \epsilon^2 \lVert xg \rVert_{L^2} \left(\sum_{2^k \gtrsim t^{-1/3}} \lVert \chi_k(uwS) \rVert_{L^2}^2 \right)^{1/2}\\
    \lesssim& \epsilon^2 \lVert xg \rVert_{L^2} \left(\sum_{2^k \gtrsim t^{-1/3}} \left(\lVert \chi_{\sim k} u \rVert_{L^\infty} \lVert \chi_{\sim k} S \rVert_{L^\infty} \lVert \chi_{\sim k} w \rVert_{L^2}\right)^2 \right)^{1/2}\\
    \lesssim& M\epsilon^4 t^{-1} \lVert xg \rVert_{L^2}^2
\end{split}\end{equation*}
We now turn to the term~\eqref{eqn:xf-energy-estimate-modulation-term}.  By~\eqref{eqn:bootstrap-hypotheses} and~\eqref{eqn:param-deriv-LS}
\begin{equation}\label{eqn:L-D-p-S-hat-u-bd}
    \lVert LD_pS \hat{u}(0,t) \rVert_{L^2} \lesssim M^3\epsilon^5 t^{-5/6-\beta}
\end{equation}
from which it easily follows that
\begin{equation*}
    |\eqref{eqn:xf-energy-estimate-modulation-term}| \lesssim M^5\epsilon^5 t^{-5/6-\beta} \lVert xg \rVert_{L^2}
\end{equation*}
which is compatible with~\eqref{eqn:xg-desired-bound} since $\epsilon \ll M^{-3/2}$.

It only remains to control~\cref{eqn:xf-energy-estimate-paraprods}.  We will rewrite this term to exploit the space-time resonance structure of the phase $\phi$.  Recall that the space-time resonances are
\begin{align*}
    \cS &= \{\eta = \sigma = \xi / 3\}\\
    \cT &= \{\xi = \eta\} \cup \{\xi = \sigma\}\\
    \cR &= \{(0,0,0)\}
\end{align*}
Let $\chi^\cS, \chi^\cT, \chi^\cR$ be a smooth partition of unity such that:
\begin{enumerate}[(i)]
    \item $\chi^\cS$ and $\chi^\cT$ are supported away from the sets $\cS$ and $\cT$, respectively,
    \item $\chi^\cS$ and $\chi^\cT$ are $0$-homogeneous outside a ball of radius $2$ and vanish within a ball of radius $1$,
    \item $\chi^\cR$ is supported inside the ball of radius $2$, and 
    \item the support of $\chi^{\cS}$ is contained in the set
    \begin{equation*}
        \widetilde{\cT} = \left\{(\xi,\eta,\sigma) : \frac{\xi}{\eta} \in \left[1-c, 1+c\right] \text{ or } \frac{\xi}{\sigma} \in \left[1-c, 1+c\right] \right\}
    \end{equation*}
    where $c \ll 1$ is a small constant and $\chi^\cT$ is supported away from $\cT$.  (Note that $\widetilde{\cT} \cap \cS$ is empty for small $c$, so this is possible).
\end{enumerate}
Define $\chi^\bullet_t = \chi^\bullet(t^{1/3} \cdot)$ for $\bullet = \cS, \cT, \cR$.  If we write
\begin{equation}\label{eqn:basic-STR-division}
    T_{\partial_\xi \phi e^{is\phi}} = T_{\partial_\xi \phi e^{is\phi} \chi^\cT_s} + T_{\partial_\xi \phi e^{is\phi} \chi^\cS_s} + T_{\partial_\xi \phi e^{is\phi} \chi^\cR_s}
\end{equation}
then we can naturally write~\eqref{eqn:xf-energy-estimate-paraprods} as
\begin{subequations}\begin{align}
    \eqref{eqn:xf-energy-estimate-paraprods} =&  \cramped{-t\Im\langle xg, T_{\partial_\xi \phi e^{it\phi} \chi^\cR_t}(f, \partial_x g, \overline{f}) + T_{\partial_\xi \phi e^{it\phi} \chi^\cR_t}(f, \partial_x h, \overline{g}) + T_{\partial_\xi \phi e^{it\phi} \chi^\cR_t}(g, \partial_x h, \overline{f}) \rangle} \label{eqn:xg-str-term}\\
    &\cramped{- t\Im\langle xg, T_{\partial_\xi \phi e^{it\phi} \chi^\cS_t}(f, \partial_x g, \overline{f}) + T_{\partial_\xi \phi e^{it\phi} \chi^\cS_t}(f, \partial_x h, \overline{g}) + T_{\partial_\xi \phi e^{it\phi} \chi^\cS_t}(g, \partial_x h, \overline{f}) \rangle} \label{eqn:xg-space-non-res-term}\\
    &\cramped{- t\Im\langle xg, T_{\partial_\xi \phi e^{it\phi} \chi^\cT_t}(f, \partial_x g, \overline{f}) + T_{\partial_\xi \phi e^{it\phi} \chi^\cT_t}(f, \partial_x h, \overline{g}) + T_{\partial_\xi \phi e^{it\phi} \chi^\cT_t}(g, \partial_x h, \overline{f}) \rangle} \label{eqn:xg-time-non-res-term}
\end{align}\end{subequations}
In the following sections, we will show that the space-time resonant and space non-resonant terms satisfy the estimate 
\begin{equation*}
    |\eqref{eqn:xg-str-term}|, |\eqref{eqn:xg-space-non-res-term}| \lesssim M^2\epsilon^2 \lVert xg \rVert_{L^2}^2 t^{-1}\end{equation*}
pointwise in time, and that
\begin{equation*}
    \left|\int_1^t \eqref{eqn:xg-time-non-res-term} \,ds\right| \lesssim \epsilon^2 + M^2\epsilon^2\lVert xg \rVert_{L^2}^2 + \int_1^t\left[M^2\epsilon^2s^{-1}\lVert xg(s) \rVert_{L^2}^2 + M^2\epsilon^3 t^{-5/6-\beta} \lVert xg \rVert_{L^2}\right]ds
\end{equation*}
which is enough to prove~\eqref{eqn:xg-desired-bound}.

\subsection{The space-time resonant multiplier}\label{sec:xf-str}
Since the space-time resonant set is a single point, we can control its contribution using the $L^\infty$ bounds on $\hat{g},$ $\hat{f}$ and $\hat{S}$.  Define
\begin{equation*}
    m^{\cR}_t = i(\xi - \eta - \sigma)\chi^\cR_s \partial_\xi \phi e^{it\phi}
\end{equation*}
Then, $m^\cR_t$ is of size $O(t^{-1})$ and is supported within the region $|\xi| + |\eta| + |\sigma| \lesssim t^{-1/3}$, so
\begin{equation*}\begin{split}
    \left|\eqref{eqn:xg-str-term} \right| 
    &\leq t\lVert xg \rVert_{L^2}\left\lVert \int_{\bbR^2} m^\cR_t\hat g(\xi-\eta-\sigma)\hat f(\eta)\overline{\hat f(-\sigma)} \,d\eta d\sigma \right\rVert_{L^2_\xi} + \bsim\\
    &\lesssim M^2\epsilon^2t^{-1}\lVert xg \rVert_{L^2}^2
\end{split}\end{equation*}
as required.

\subsection{The space non-resonant multiplier}\label{sec:xf-sr}

We will handle the terms supported away from the space resonant set in frequency space using integration by parts in $\eta$ and $\sigma$.  To simplify notation, let us write
\begin{equation*}
    \mathcal{N}^\cS(f,g,h) = -t\Bigl(T_{\partial_\xi \phi e^{it\phi} \chi^\cS_t}(f, \partial_x g, \overline{f}) + T_{\partial_\xi \phi e^{it\phi} \chi^\cS_t}(f, \partial_x h, \overline{g}) + T_{\partial_\xi \phi e^{it\phi} \chi^\cS_t}(g, \partial_x h, \overline{f})\Bigr)
\end{equation*}
Then, the desired bound for~\eqref{eqn:xg-space-non-res-term} follows from showing that 
\begin{equation}\label{eqn:N-s-desired-bound}
    \lVert \mathcal{N}^\cS(f,g,h) \rVert_{L^2} \lesssim M^2\epsilon^2 t^{-1} \lVert xg \rVert_{L^2}
\end{equation}
Let us first consider the term $-tT_{\partial_\xi \phi e^{it\phi} \chi^\cS_t}(f,\partial_x g,\overline{f})$.  Integrating by parts in frequency gives
\begin{align}
    t\hat{T}_{\partial_\xi \phi e^{it\phi} \chi^\cS_t}(f,\partial_x g,\overline{f}) =& -i\int t\partial_\xi \phi \chi^\cS_t e^{it\phi} (\xi - \eta -\sigma) \hat{g}(\xi-\eta-\sigma) \hat{f}(\eta)  \overline{\hat{f}(-\sigma)} \,d\eta d\sigma\notag\\
    \begin{split}
    =& e^{-it\xi^3} \hat{T}_{\nablaes \cdot m^\cS_t(\xi - \eta - \sigma)}(u,w, \overline{u})\\
    &+ e^{-it\xi^3} \hat{T}_{m^{\cS,\sigma}_t}(u, \partial_x w, \overline{Lu})\\
    &+ e^{-it\xi^3} \hat{T}_{m^{\cS,\sigma}_t}(u,w, \overline{u})\\
    &- e^{-it\xi^3} \hat{T}_{m^{\cS,\sigma}_t}(u, \partial_x Lw, \overline{u})\\
    &+ \{\text{similar terms}\}
\end{split}\label{eqn:space-pseudoproduct-division}\end{align}
where $m^\cS_s$ is the vector-valued symbol $m^cS_s = \frac{\partial_\xi \phi}{|\nablaes\phi|^2} \nablaes \phi \chi^\cS_s$, and $m^{\cS,\eta}_s,m^{\cS,\sigma}_s$ are its components.  Because $Lu$ obeys worse estimates than $LS$ and $Lw$, we rewrite the term containing the $Lu$ factor using
\begin{equation*}
    {T}_{m^{\cS,\sigma}_t}(u, \partial_x w, \overline{Lu}) = {T}_{m^{\cS,\sigma}_t}(u, \partial_x w, \overline{LS}) +  {T}_{m^{\cS,\sigma}_t}(u, \partial_x u, \overline{Lw}) - {T}_{m^{\cS,\sigma}_t}(u, \partial_x S, \overline{Lw})
\end{equation*}
Similar expressions hold for the other terms in~$\mathcal{N}^\cS$.  Note that each symbol $m$ occurring after the last equality in~\eqref{eqn:space-pseudoproduct-division} satisfies Coifman-Meyer bounds
\begin{equation*}
    \left|(|\xi| + |\eta| + |\sigma|)^{|\alpha|}\partial_{\xi,\eta,\sigma}^\alpha m\right| \lesssim_\alpha 1
\end{equation*}
and is supported on  $\{\xi \sim \eta \gtrsim t^{-1/3}\} \cup \{\xi \sim \sigma \gtrsim t^{-1/3}\}$.  Thus, by dividing dyadically in frequency, we can write
\begin{subequations}\begin{align}
    \mathcal{N}^\cS(f,g,h) =& \sum_{2^j \gtrsim t^{-1/3}} T_{m_j^{\cS}}( u,  w, \overline{u})\label{eqn:mS-dyadic-no-deriv}\\
    &\qquad+ T_{m_j^{\cS}}( u, \partial_x Lw, \overline{u})\label{eqn:mS-dyadic-Lw-deriv}\\
    &\qquad+ T_{m_j^{\cS}}(u,\partial_x w, \overline{LS})\label{eqn:mS-dyadic-LS}\\
    &\qquad + T_{m_j^\cS}(u, S, \overline{w}) \label{eqn:mS-dyadic-SSw}\\
    &\qquad + T_{m_j^{\cS}}(u, \partial_x L S, \overline{w})\label{eqn:mS-dyadic-deriv-on-LS}\\
    &\qquad + T_{m_j^{\cS}}(u, \partial_x u, \overline{Lw})\label{eqn:mS-dyadic-deriv-on-S-and-Lw}\\
    &\qquad + T_{m_j^{\cS}}(LS, \partial_x u, \overline{w})\label{eqn:mS-dyadic-LS-and-outside-w}\\
    &+ \bsime\nonumber
\end{align}\end{subequations}
where $m_j^{\cS}$ denotes a generic symbol satisfying the same requirements as symbols $m_j$, but with the stricter support requirement 
\begin{equation*}
    \left\{|\xi| + |\eta| + |\sigma| \sim 2^j, |\xi - \eta| \ll 2^j\right\} \cup \left\{|\xi| + |\eta| + |\sigma| \sim 2^j, |\xi - \sigma| \ll 2^j\right\}
\end{equation*}
We now turn to the task of deriving estimates for~\cref{eqn:mS-dyadic-deriv-on-LS,eqn:mS-dyadic-LS,eqn:mS-dyadic-Lw-deriv,eqn:mS-dyadic-no-deriv,eqn:mS-dyadic-deriv-on-S-and-Lw,eqn:mS-dyadic-SSw,eqn:mS-dyadic-LS-and-outside-w}.  Using the support condition on $m^\cS_j$, we can decompose any pseudoproduct $T_{m_j^{\cS}}$ as
\begin{equation}\label{eqn:space-non-res-pseudoproduct-expansion}\begin{split}
    T_{m_j^{\cS}}(p,q,r) =& T_{m_j^\cS} (p_{\ll j}, q_{\lesssim j}, r_{\sim j}) + T_{m_j^\cS}(p_{\sim j}, q_{\lesssim j}, r_{\ll j}) + Q_{\sim j}T_{m_j^{\cS}}(p_{\sim j}, q_{\sim j}, r_{\sim j})
\end{split}\end{equation}
The term $T_{m_j^\cS}(p_{\sim j}, q_{\ll j}, r_{\sim j})$ does not appear in the expansion because of the support assumption on $m^\cS_j$.  In particular, we always have $|\xi - \eta - \sigma| \lesssim \max\{|\eta|, |\sigma|\}$, which helps us to control the derivative.  We now consider each term~\cref{eqn:mS-dyadic-no-deriv,eqn:mS-dyadic-Lw-deriv,eqn:mS-dyadic-LS,eqn:mS-dyadic-deriv-on-LS,eqn:mS-dyadic-deriv-on-S-and-Lw,eqn:mS-dyadic-SSw,eqn:mS-dyadic-LS-and-outside-w} in turn and use the division~\eqref{eqn:space-non-res-pseudoproduct-expansion} and the decay estimates to obtain the bound~\eqref{eqn:N-s-desired-bound}.

\subsubsection{The bound for \texorpdfstring{\eqref{eqn:mS-dyadic-no-deriv}}{(\ref{eqn:mS-dyadic-no-deriv}}} From~\eqref{eqn:space-non-res-pseudoproduct-expansion}, we can write
\begin{subequations}\begin{align}
    T_{m_j^\cS}(u, w, \overline{u}) =&  T_{m_{j}^\cS} (u_{\ll j}, w_{\lesssim j}, \overline{u}_{\sim j})\label{eqn:mS-no-deriv-low-freq}\\
    &+ Q_{\sim j} T_{m_j^\cS}(u_{\sim j}, w_{\sim j}, \overline{u}_{\sim j})\label{eqn:mS-no-deriv-equal-freq}\\
    &+\{\text{similar terms}\}\nonumber
\end{align}\end{subequations}
Recall that the symbols $m^\cS_j$ satisfy the hypotheses of~\Cref{thm:L1-symbol-bounds} uniformly in $j$.  Thus, using the $L^2$ bound for $w_{\sim j}$ from~\eqref{eqn:w-lin-ests-L-2} and the dispersive decay of $u_{\sim j}$ given in~\eqref{eqn:u-lin-ests}, we find that
\begin{equation*}\begin{split}
    \biggl\lVert \sum_{2^j \gtrsim t^{-1/3}}  Q_{\sim j} T_{m_j^\cS}(u_{\sim j}, w_{\sim j}, \overline{u}_{\sim j})\biggr\rVert_{L^2} \lesssim&  \biggl(\sum_{2^j \gtrsim t^{-1/3}} \lVert u_{\sim j} \rVert_{L^\infty}^4\lVert w_{\sim j} \rVert_{L^2}^2 \biggr)^{1/2}\\
    \lesssim& M^2\epsilon^2 t^{-1} \lVert xg \rVert_{L^2}
\end{split}\end{equation*}
as required.  For~\eqref{eqn:mS-no-deriv-low-freq}, we introduce a further dyadic decomposition in space to write
\begin{subequations}\begin{align}
    T_{m_{j}^\cS}(u_{\ll j}, w_{\lesssim j}, \overline{u}_{\sim j}) =& T_{m_{j}^\cS}(u_{\ll j}, w_{\lesssim j}, (1 - \chi_{[j - 20, j + 20]})\overline{u}_{\sim j})\label{eqn:mS-no-deriv-low-freq-1}\\
    &+ T_{m_{j}^\cS}(u_{\ll j}, w_{\lesssim j}, \chi_{[j - 20, j + 20]}\overline{u}_{\sim j})\label{eqn:mS-no-deriv-low-freq-2}
\end{align}\end{subequations}
The first term is straightforward to bound using~\eqref{eqn:u-lin-ests} and the $L^2$ estimates for $w$ from~\eqref{eqn:w-lin-ests-L-2}:
\begin{equation*}\begin{split}
    \lVert\eqref{eqn:mS-no-deriv-low-freq-1}\rVert_{L^2} \lesssim& \lVert u_{\ll j} \rVert_{L^\infty} \lVert (1 - \chi_{[j - 20, j + 20]}) u_{\sim j} \rVert_{L^\infty} \lVert w_{\lesssim j} \rVert_{L^2}\\
    \lesssim& M^2 \epsilon^2 t^{-7/6} 2^{-j/2} \lVert xg \rVert_{L^2}
\end{split}\end{equation*}
which is sufficient. Turning to the second term~\eqref{eqn:mS-no-deriv-low-freq-2}, we write
\begin{subequations}\begin{align}
    \eqref{eqn:mS-no-deriv-low-freq-2} =&  \chi_{[j - 30, j + 30]} T_{m_{j}^\cS}(\chi_{[j - 30, j + 30]} u_{\ll j}, \chi_{[j - 30 ,  j + 30]} w_{\lesssim j}, \chi_{[j - 20 ,  j + 20]}\overline{u}_{\sim j})\label{eqn:mS-no-deriv-low-freq-3-main}\\
    &+ (1-\chi_{[j - 30 ,  j + 30]}) T_{m_{j}^\cS}(u_{\ll j}, w_{\lesssim j}, \chi_{[j - 20 ,  j + 20]}\overline{u}_{\sim j})\label{eqn:mS-no-deriv-low-freq-3-sub-1}\\
    &+ \chi_{[j - 30 ,  j + 30]} T_{m_{j}^\cS}((1-\chi_{[j - 30 ,  j + 30]})u_{\ll j},w_{\lesssim j}, \chi_{[j - 20 ,  j + 20]}\overline{u}_{\sim j})\label{eqn:mS-no-deriv-low-freq-3-sub-2}\\
    &+ \chi_{[j - 30 ,  j + 30]} T_{m_{j}^\cS}(\chi_{[j - 30 ,  j + 30]}u_{\ll j}, (1-\chi_{[j - 30 ,  j + 30]})w_{\lesssim j}, \chi_{[j - 20 ,  j + 20]}\overline{u}_{\sim j})\label{eqn:mS-no-deriv-low-freq-3-sub-3}
\end{align}\end{subequations}
The subterms~\cref{eqn:mS-no-deriv-low-freq-3-sub-1,eqn:mS-no-deriv-low-freq-3-sub-2,eqn:mS-no-deriv-low-freq-3-sub-3} are non-pseudolocal in the sense of~\Cref{thm:cm-paraprod-pseudolocality}, so they satisfy the bound
\begin{equation*}\begin{split}
    \lVert \eqref{eqn:mS-no-deriv-low-freq-3-sub-1} \rVert_{L^2} + \lVert \eqref{eqn:mS-no-deriv-low-freq-3-sub-2} \rVert_{L^2} + \lVert \eqref{eqn:mS-no-deriv-low-freq-3-sub-3} \rVert_{L^2} 
    \lesssim& M^2\epsilon^2 t^{-11/6} 2^{-5/2j} \lVert xg \rVert_{L^2}
\end{split}\end{equation*}
which gives the required bound after summing in $j$.  To bound the leading order term~\eqref{eqn:mS-no-deriv-low-freq-3-main}, we use almost orthogonality and find that
\begin{equation*}\begin{split}
    \biggl\lVert \sum_{j} \eqref{eqn:mS-no-deriv-low-freq-3-main} \biggr\rVert_{L^2} \lesssim& \biggl(\sum_{2^j \gtrsim t^{-1/3}} \bigl(   \lVert \chi_{[j - 30 ,  j + 30]}u_{\ll j}\rVert_{L^\infty} \lVert \chi_{[j - 30 ,  j + 30]}w_{\lesssim j} \rVert_{L^2} \lVert {u}_{\sim j}\rVert_{L^\infty}\bigr)^2\biggr)^{1/2}\\
    \lesssim& M^2\epsilon^2 t^{-1} \lVert xg \rVert_{L^2}
\end{split}\end{equation*}


\subsubsection{The bound for~\texorpdfstring{\eqref{eqn:mS-dyadic-Lw-deriv}}{(\ref{eqn:mS-dyadic-Lw-deriv})}}The estimates for~\eqref{eqn:mS-dyadic-Lw-deriv} are analogous to those for~\eqref{eqn:mS-dyadic-no-deriv} once we use the bounds
\begin{equation*}\begin{split}
    \lVert \partial_x (Lw)_{j} \rVert_{L^2} \sim& 2^j\lVert xg \rVert_{L^2}c_j\\
    \lVert \partial_x (Lw)_{\lesssim j} \rVert_{L^2} \lesssim& 2^j \lVert xg \rVert_{L^2}
\end{split}\end{equation*}
in place of the Hardy-type bounds on $w$.  

\subsubsection{The bound for~\texorpdfstring{\eqref{eqn:mS-dyadic-LS}}{(\ref{eqn:mS-dyadic-LS}}} Applying~\eqref{eqn:space-non-res-pseudoproduct-expansion}, we find
\begin{subequations}\begin{align}
    T_{m_j^{\cS}}(u, \partial_x w, \overline{LS}) =& T_{m_j^{\cS}}(u_{\ll j}, \partial_x w_{\lesssim j}, (\overline{LS})_{\sim j}) \label{eqn:mS-dyadic-LS-ll-sim}\\
    &+ T_{m_j^{\cS}}(u_{\sim j}, \partial_x w_{\lesssim j}, (\overline{LS})_{\ll j}) \label{eqn:mS-dyadic-LS-sim-ll}\\
    &+ Q_{\sim j}T_{m_j^{\cS}}(u_{\sim j}, \partial_x w_{\sim j}, (\overline{LS})_{\sim j}) \label{eqn:mS-dyadic-LS-sim-sim}
\end{align}\end{subequations}
We first consider~\eqref{eqn:mS-dyadic-LS-sim-sim}.  Using almost orthogonality and the estimates for $LS$ from~\eqref{eqn:LS-cubic-bounds}, we find that
\begin{equation*}\begin{split}
    \biggl\lVert \sum_{j} \eqref{eqn:mS-dyadic-LS-sim-sim}  \biggr\rVert_{L^2} \lesssim& \biggl(\sum_{2^j \gtrsim t^{-1/3}} \lVert u_{\sim j} \rVert_{L^\infty}^2 \lVert \partial_x w_{\sim j} \rVert_{L^2}^2 \lVert (LS)_{\sim j} \rVert_{L^\infty}^2\biggr)^{1/2}\\
    \lesssim& M\epsilon^4 t^{-1} \lVert xg \rVert_{L^2}
\end{split}\end{equation*}
where we have used the $L^2$ bounds for $\partial_x w_j$ from~\eqref{eqn:w-lin-ests-L-2}.  For~\eqref{eqn:mS-dyadic-LS-ll-sim}, we divide dyadically in space to obtain
\begin{subequations}\begin{align}
    T_{m_j^{\cS}}(u_{\ll j}, \partial_x w_{\lesssim j}, (\overline{LS})_{\sim j}) =& T_{m_j^{\cS}}(\chi_{> j} u_{\ll j}, \partial_x w_{\lesssim j}, (\overline{LS})_{\sim j})\label{eqn:mS-dyadic-LS-ll-sim-1}\\
    &+ \sum_{k < j - 40} T_{m_j^{\cS}}(\chi_{k} u_{\ll j}, \partial_x w_{\lesssim j}, (\overline{LS})_{\sim j})\label{eqn:mS-dyadic-LS-ll-sim-3}\\
    &+ T_{m_j^{\cS}}(\chi_{[j - 40, j]} u_{\ll j}, \partial_x w_{\lesssim j}, (\overline{LS})_{\sim j})\label{eqn:mS-dyadic-LS-ll-sim-2}
\end{align}\end{subequations}
For the first term, the dispersive decay estimates give us the bound
\begin{equation*}\begin{split}
    \lVert \eqref{eqn:mS-dyadic-LS-ll-sim-1} \rVert_{L^2} 
    \lesssim& M \epsilon^4 t^{-4/3} 2^{-j} \lVert xg \rVert_{L^2}
\end{split}\end{equation*}
which we can sum in $j$ to get the required bound.  For the second term, we perform the further division
\begin{subequations}\begin{align}
    \eqref{eqn:mS-dyadic-LS-ll-sim-3} =& \sum_{k < j - 40} T_{m_j^{\cS}}(\chi_{k} u_{\ll j}, \chi_{\sim k}\partial_x w_{\lesssim j}, \chi_{\sim k} (\overline{LS})_{\sim j})\label{eqn:mS-dyadic-LS-ll-sim-3-main}\\
    &+ \sum_{k < j - 40} T_{m_j^{\cS}}(\chi_{k} u_{\ll j}, (1-\chi_{\sim k})\partial_x w_{\lesssim j}, (\overline{LS})_{\sim j})\label{eqn:mS-dyadic-LS-ll-sim-3-sub-1}\\
    &+ \sum_{k < j - 40} T_{m_j^{\cS}}(\chi_{k} u_{\ll j}, \chi_{\sim k}\partial_x w_{\lesssim j}, (1- \chi_{\sim k})(\overline{LS})_{\sim j})\label{eqn:mS-dyadic-LS-ll-sim-3-sub-2}
\end{align}\end{subequations}
\Cref{thm:cm-paraprod-pseudolocality} applies to the pseudoproducts in the terms~\cref{eqn:mS-dyadic-LS-ll-sim-3-sub-2,eqn:mS-dyadic-LS-ll-sim-3-sub-1}, yielding the bound
\begin{equation*}\begin{split}
    \lVert \eqref{eqn:mS-dyadic-LS-ll-sim-3-sub-1} \rVert_{L^2} + \lVert\eqref{eqn:mS-dyadic-LS-ll-sim-3-sub-2} \rVert_{L^2} 
    \lesssim& M \epsilon^4 t^{-7/6} 2^{-j/2} \lVert xg \rVert_{L^2}
\end{split}\end{equation*}
which is sufficient.  To bound the remaining term~\eqref{eqn:mS-dyadic-LS-ll-sim-3-main}, we use the bounds for the localized terms $\chi_{\sim k} \partial_x w_{\lesssim j}$ and $\chi_{\sim k} (LS)_{\sim j}$ from~\eqref{eqn:w-lin-ests-L-2} and~\eqref{eqn:LS-cubic-bounds} to obtain
\begin{equation*}\begin{split}
    \lVert \eqref{eqn:mS-dyadic-LS-ll-sim-3-main} \rVert_{L^2} \lesssim& \sum_{t^{-1/3} \leq 2^k \leq 2^{j-40}} \lVert \chi_k u_{\ll j} \rVert_{L^\infty} \lVert \chi_{\sim k} \partial_x w_{\lesssim j} \rVert_{L^2} \lVert \chi_{\sim k} (LS)_{\sim j} \rVert_{L^\infty}\\
    \lesssim& M\epsilon^4 t^{-4/3} 2^{-j} \lVert xg \rVert_{L^2}
\end{split}\end{equation*}
To bound~\eqref{eqn:mS-dyadic-LS-ll-sim-2}, we write
\begin{subequations}\begin{align}
    \eqref{eqn:mS-dyadic-LS-ll-sim-2} =&  \chi_{[j - 50, j+10]} T_{m_j^\cS}(\chi_{[j - 40, j]} S_{\ll j}, \chi_{[j - 50, j+10]} \partial_x w_{\lesssim j}, (\overline{LS})_{\sim j})\label{eqn:mS-dyadic-LS-ll-sim-2-main}\\
    &+ (1-\chi_{[j - 50, j+10]})T_{m_j^\cS}(\chi_{[j - 40, j]} S_{\ll j}, \chi_{[j - 50, j+10]} \partial_x w_{\lesssim j}, (\overline{LS})_{\sim j})\label{eqn:mS-dyadic-LS-ll-sim-2-sub-1}\\
    &+  T_{m_j^\cS}(\chi_{[j - 40, j]} S_{\ll j}, (1-\chi_{[j - 50, j+10]}) \partial_x w_{\lesssim j}, (\overline{LS})_{\sim j})\label{eqn:mS-dyadic-LS-ll-sim-2-sub-2}
\end{align}\end{subequations}
The terms~\cref{eqn:mS-dyadic-LS-ll-sim-2-sub-1,eqn:mS-dyadic-LS-ll-sim-2-sub-2} can be handled using~\Cref{thm:cm-paraprod-pseudolocality} in the same manner as~\cref{eqn:mS-dyadic-LS-ll-sim-3-sub-2,eqn:mS-dyadic-LS-ll-sim-3-sub-1}.  The terms in~\eqref{eqn:mS-dyadic-LS-ll-sim-2-main} are almost orthogonal in physical space, so we can write
\begin{equation*}\begin{split}
    \biggl\lVert \sum_{j} \eqref{eqn:mS-dyadic-LS-ll-sim-2-main} \biggr\rVert_{L^2} \lesssim& \biggl( \sum_{2^j \gtrsim t^{-1/3}} \lVert \chi_{[j - 40, j]}u_{\ll j} \rVert_{L^\infty}^2  \lVert \chi_{[j - 50, j+10]} \partial_x w_{\lesssim j} \rVert_{L^2}^2 \lVert ({LS})_{\sim j}\rVert_{L^\infty}^2\biggr)^{1/2}\\
    \lesssim& M\epsilon^4 t^{-1}\lVert xg \rVert_{L^2}
\end{split}\end{equation*}
Combining the estimates for~\cref{eqn:mS-dyadic-LS-ll-sim-1,eqn:mS-dyadic-LS-ll-sim-2,eqn:mS-dyadic-LS-ll-sim-3} shows that the bound for~\eqref{eqn:mS-dyadic-LS-ll-sim} holds.

We now turn to~\eqref{eqn:mS-dyadic-LS-sim-ll}.  We divide dyadically in space to write
\begin{subequations}\begin{align}
    \eqref{eqn:mS-dyadic-LS-sim-ll} =& \sum_{k < j - 30} T_{m_j^\cS}(u_{\sim j}, \partial_x w_{\lesssim j}, \chi_k (\overline{LS})_{\ll j}) \label{eqn:mS-dyadic-LS-sim-ll-1}\\
    &+ T_{m_j^\cS}(u_{\sim j}, \partial_x w_{\lesssim j}, \chi_{[j - 30 ,  j + 30]} (\overline{LS})_{\ll j}) \label{eqn:mS-dyadic-LS-sim-ll-2}\\
    &+ \sum_{k > j + 30} T_{m_j^\cS}(u_{\sim j}, \partial_x w_{\lesssim j}, \chi_k (\overline{LS})_{\ll j}) \label{eqn:mS-dyadic-LS-sim-ll-3}
\end{align}\end{subequations}
For~\eqref{eqn:mS-dyadic-LS-sim-ll-1}, we have
\begin{subequations}\begin{align}
    \eqref{eqn:mS-dyadic-LS-sim-ll-1} =& \sum_{k < j - 30} T_{m_j^\cS}(\chi_{\sim k} u_{\sim j}, \chi_{\sim k} \partial_x w_{\lesssim j}, \chi_k (\overline{LS})_{\ll j})\label{eqn:mS-dyadic-LS-sim-ll-1-main}\\
    &+ \sum_{k < j - 30} T_{m_j^\cS}((1 - \chi_{\sim k}) u_{\sim j}, \partial_x w_{\lesssim j}, \chi_k (\overline{LS})_{\ll j})\label{eqn:mS-dyadic-LS-sim-ll-1-sub-1}\\
    &+ \sum_{k < j - 30} T_{m_j^\cS}(\chi_{\sim k} u_{\sim j}, (1 - \chi_{\sim k}) \partial_x w_{\lesssim j}, \chi_k (\overline{LS})_{\ll j})\label{eqn:mS-dyadic-LS-sim-ll-1-sub-2}
\end{align}\end{subequations}
The bounds for~\cref{eqn:mS-dyadic-LS-sim-ll-1-sub-1,eqn:mS-dyadic-LS-sim-ll-1-sub-2} immediately follow from~\Cref{thm:cm-paraprod-pseudolocality}:
\begin{equation*}\begin{split}
    \lVert \eqref{eqn:mS-dyadic-LS-sim-ll-1-sub-1} \rVert_{L^2} + \lVert \eqref{eqn:mS-dyadic-LS-sim-ll-1-sub-2} \rVert_{L^2} \lesssim& \smashoperator[r]{\sum_{t^{-1/3} \lesssim 2^k < 2^{j - 30}}} (t2^{2k+j})^{-2} \lVert u_{\sim j} \rVert_{L^\infty} \lVert (LS)_{\ll j} \rVert_{L^\infty} \lVert \partial_x w_{\lesssim j} \rVert_{L^2}\\
    \lesssim& M\epsilon^2 t^{-7/6} 2^{-j/2} \lVert xg \rVert_{L^2}
\end{split}\end{equation*}
which is sufficient.  Turning to~\eqref{eqn:mS-dyadic-LS-sim-ll-1-main}, we have that
\begin{equation*}\begin{split}
    \lVert \eqref{eqn:mS-dyadic-LS-sim-ll-1-main} \rVert_{L^2} \lesssim& \sum_{k < j - 30} \lVert \chi_{\sim k} u_{\sim j} \rVert_{L^\infty} \lVert \chi_{\sim k} \partial_x w_{\lesssim j} \rVert_{L^2} \lVert \chi_k (LS)_{\ll j} \rVert_{L^\infty}\\
    \lesssim& M\epsilon^4 t^{-4/3} 2^{-j} \lVert xg \rVert_{L^2}
\end{split}\end{equation*}
which is acceptable, completing the argument for~\eqref{eqn:mS-dyadic-LS-sim-ll-1}.  Turning to~\eqref{eqn:mS-dyadic-LS-sim-ll-2}, we can write
\begin{subequations}\begin{align}
    \eqref{eqn:mS-dyadic-LS-sim-ll-2} =& \chi_{[j - 40 ,  j + 40]} T_{m_j^\cS} (u_{\sim j}, \chi_{[j - 40 ,  j + 40]} \partial_x w_{\lesssim j}, \chi_{[j - 30 ,  j + 30]} (\overline{LS}){\ll j})\label{eqn:mS-dyadic-LS-sim-ll-2-main}\\
    &+ (1-\chi_{[j - 40 ,  j + 40]}) T_{m_j^\cS} (u_{\sim j}, \chi_{[j - 40 ,  j + 40]} \partial_x w_{\lesssim j}, \chi_{[j - 30 ,  j + 30]} (\overline{LS}){\ll j})\label{eqn:mS-dyadic-LS-sim-ll-2-sub-1}\\
    &+ T_{m_j^\cS} (u_{\sim j}, (1-\chi_{[j - 40 ,  j + 40]}) \partial_x w_{\lesssim j}, \chi_{[j - 30 ,  j + 30]} (\overline{LS}){\ll j})\label{eqn:mS-dyadic-LS-sim-ll-2-sub-2}
\end{align}\end{subequations}
The terms~\cref{eqn:mS-dyadic-LS-sim-ll-2-sub-1,eqn:mS-dyadic-LS-sim-ll-2-sub-2} can be controlled using~\Cref{thm:cm-paraprod-pseudolocality}, so we focus on~\eqref{eqn:mS-dyadic-LS-sim-ll-2-main}.  There, we use almost orthogonality to obtain
\begin{equation*}\begin{split}
    \biggl\lVert  \sum_{j}\eqref{eqn:mS-dyadic-LS-sim-ll-2-main} \biggr\rVert_{L^2} \lesssim& \biggl(  \smashoperator[r]{\sum_{2^j\gtrsim t^{-1/3}}} \lVert u_{\sim j} \rVert_{L^\infty}^2 \lVert \chi_{[j - 40 ,  j + 40]} \partial_x w_{\lesssim j} \rVert_{L^2}^2 \lVert \chi_{[j - 30 ,  j + 30]} (LS)_{\ll j} \rVert_{L^\infty}^2\biggr)^{\frac{1}{2}}\\
    \lesssim& M\epsilon^4 t^{-1} \lVert xg \rVert_{L^2}
\end{split}\end{equation*}
as required.

Finally, for~\eqref{eqn:mS-dyadic-LS-sim-ll-3}, we have
\begin{subequations}\begin{align}
    \eqref{eqn:mS-dyadic-LS-sim-ll-3} =&  T_{m_j^\cS}(\chi_{> j + 20} u_{\sim j}, \partial_x w_{\lesssim j}, \chi_{> j + 30} (\overline{LS})_{\ll j}) \label{eqn:mS-dyadic-LS-sim-ll-3-main}\\
    &+ T_{m_j^\cS}(\chi_{\leq j - 20} u_{\sim j}, \partial_x w_{\lesssim j}, \chi_{> j + 30} (\overline{LS})_{\ll j}) \label{eqn:mS-dyadic-LS-sim-ll-3-sub}
\end{align}\end{subequations}
The second term is easily controlled using~\Cref{thm:cm-paraprod-pseudolocality}, and for the main term~\eqref{eqn:mS-dyadic-LS-sim-ll-3-main} we have that
\begin{equation*}\begin{split}
    \lVert \eqref{eqn:mS-dyadic-LS-sim-ll-3-main} \rVert_{L^2} \lesssim& \lVert  \chi_{> j + 20} u_{\sim j}\rVert_{L^\infty}\lVert \partial_x w_{\lesssim j}\rVert_{L^2}\lVert  \chi_{> j + 30} (\overline{LS})_{\ll j}) \rVert_{L^\infty}\\
    \lesssim& M^4\epsilon^4 t^{-4/3} 2^{-j} \lVert xg \rVert_{L^2}
\end{split}\end{equation*}
which is sufficient to bound~\eqref{eqn:mS-dyadic-LS-sim-ll-3}, completing the estimate for~\eqref{eqn:mS-dyadic-LS}.

\begin{rmk}\label{rmk:non-pseudolocal-exclusion}
    In estimating~\cref{eqn:mS-dyadic-no-deriv,eqn:mS-dyadic-Lw-deriv,eqn:mS-dyadic-LS}, we frequently performed decompositions in space along the lines of
    \begin{equation*}
        T_{m_j^\cS}(p,q,\chi_k r) = \chi_{\sim k} T_{m_j^\cS}(\chi_{\sim k} p, \chi_{\sim k} q,\chi_k r) + \bnpl
    \end{equation*}
    where $\bnpl$ denotes terms which can be estimated using~\Cref{thm:cm-paraprod-pseudolocality}.  The estimates for the nonpseudolocal remainder terms are routine: they do not require any refined linear or cubic estimates and are insensitive to the precise frequency localization of $p$, $q$, and $r$.  Thus, in the interest of the exposition, we will not estimate these non-pseudolocal remainders or even write them explicitly in the following sections.
\end{rmk}

\subsubsection{The bound for~\texorpdfstring{\eqref{eqn:mS-dyadic-SSw}}{(\ref{eqn:mS-dyadic-SSw})}} Here, we have
\begin{subequations}\begin{align}
    T_{m_j^\cS}(u, S, \overline{w}) =& Q_{\sim j} T_{m_j^\cS}(u_{\ll j}, S_{\ll j}, \overline{w}_{\sim j}) \label{eqn:mS-SSw-i}\\
    &+ T_{m_j^\cS}(u_{\sim j}, S_{\sim j}, \overline{w}_{\ll j})\label{eqn:mS-SSw-iv}\\
    &+ T_{m_j^\cS}(u_{\ll j}, S_{\sim j}, \overline{w}_{\sim j}) \label{eqn:mS-SSw-ii}\\
    &+ T_{m_j^\cS}(u_{\sim j}, S_{\ll j}, \overline{w}_{\ll j})\label{eqn:mS-SSw-iii}\\
    &+ Q_{\sim j} T_{m_j^\cS}(u_{\sim j}, S_{\sim j}, \overline{w}_{\sim j})\label{eqn:mS-SSw-v}
\end{align}\end{subequations}
Notice that the terms~\eqref{eqn:mS-SSw-ii} and~\eqref{eqn:mS-SSw-iii} can be controlled in the same way as~\eqref{eqn:mS-no-deriv-low-freq}, and that the term~\eqref{eqn:mS-SSw-v} can be controlled in the same way as~\eqref{eqn:mS-no-deriv-equal-freq}.   Thus, it only remains to consider the contribution from the first two terms.  For~\eqref{eqn:mS-SSw-i}, we write
\begin{subequations}\begin{align}
    \eqref{eqn:mS-SSw-i} =& Q_{\sim j} T_{m_j^\cS}(\chi_{< j - 30} u_{\ll j}, \chi_{< j - 20} S_{\ll j}, \chi_{< j - 20} \overline{w}_{\sim j}) \label{eqn:mS-SSw-i-1}\\
    &+ Q_{\sim j} T_{m_j^\cS}(\chi_{\geq j - 30} u_{\ll j}, \chi_{\geq j - 40} S_{\ll j}, \overline{w}_{\sim j}) \label{eqn:mS-SSw-i-2}\\
    &+ \bnpl\notag
\end{align}\end{subequations}
The non-pseudolocal terms are easily handled (see~\Cref{rmk:non-pseudolocal-exclusion}).  For~\eqref{eqn:mS-SSw-i-1}, we use the bound for $\chi_{\leq j - 30} w_{\sim j}$ from~\eqref{eqn:w-lin-ests-L-2} together with the fact that the terms in~\eqref{eqn:mS-SSw-i-1} are almost orthogonal to find that
\begin{equation*}\begin{split}
    \biggl\lVert \sum_{j} \eqref{eqn:mS-SSw-i-1} \biggr\rVert_{L^2} \lesssim&  \biggl( \sum_{2^j \gtrsim t^{-1/3}} \lVert u_{\ll j}\rVert_{L^\infty}^2 \lVert S_{\ll j} \rVert_{L^\infty}^2 \lVert \chi_{\leq j - 30} w_{\sim j} \rVert_{L^2}^2\biggr)^{1/2}\\
    \lesssim& M\epsilon^2 t^{-1} \lVert xg \rVert_{L^2}
\end{split}\end{equation*}
For~\eqref{eqn:mS-SSw-i-2}, the improved decay of $\chi_{\geq j - 30} u_{\ll j}$ and $\chi_{\geq j - 40} S_{\ll j}$ given by~\eqref{eqn:u-lin-ests} gives us the bound
\begin{equation*}\begin{split}
    \lVert \eqref{eqn:mS-SSw-i-2} \rVert_{L^2} \lesssim&\lVert \chi_{\geq j - 30} u_{\ll j}\rVert_{L^\infty} \lVert \chi_{\geq j - 40} S_{\ll j} \rVert_{L^\infty}\lVert w_{\sim j}\rVert_{L^2}\\
    \lesssim& M \epsilon^2 t^{-5/3}2^{-2j} \lVert xg \rVert_{L^2}
\end{split}\end{equation*}
which is sufficient after summing in $j$.  Turning to~\eqref{eqn:mS-SSw-iv}, we write
\begin{subequations}\begin{align}
    \eqref{eqn:mS-SSw-iv} =& T_{m_j^\cS}((1 - \chi_{[j-20, j+20]})u_{\sim j}, S_{\sim j}, \overline{w}_{\ll j})\label{eqn:mS-SSw-iv-1}\\
    &+ \chi_{[j-30, j+30]} T_{m_j^\cS}(\chi_{[j-20 ,  j+20]} u_{\sim j}, S_{\sim j}, \chi_{[j-30 ,  j+30]} \overline{w}_{\ll j}) \label{eqn:mS-SSw-iv-2}\\
    &+ \bnpl\notag
\end{align}\end{subequations}
The first term is easily bounded using the linear decay estimates.  
For~\eqref{eqn:mS-SSw-iv-2}, we use almost orthogonality to write
\begin{equation*}\begin{split}
    \biggl\lVert \sum_{j} \eqref{eqn:mS-SSw-iv-2} \biggr\rVert_{L^2} \lesssim&\biggl( \sum_{2^j \gtrsim t^{-1/3}} \lVert u_{\sim j}\rVert_{L^\infty}^2\lVert S_{\sim j}\rVert_{L^\infty}^2 \lVert \chi_{[j-30 ,  j+30]} \overline{w}_{\ll j})\rVert_{L^2}^2\biggr)^{1/2}\\
    \lesssim& M^2\epsilon^2 t^{-1} \lVert xg \rVert_{L^2}
\end{split}\end{equation*}
The bound for~\eqref{eqn:mS-dyadic-SSw} now follows.

\subsubsection{The bound for~\texorpdfstring{\eqref{eqn:mS-dyadic-deriv-on-LS}}{(\ref{eqn:mS-dyadic-deriv-on-LS})}} Here, we decompose the pseudoproduct as
\begin{subequations}\begin{align}
    T_{m_j^{\cS}}(u, \partial_x LS, \overline{w}) =& T_{m_j^{\cS}}(u_{\ll j}, \partial_x (LS)_{\sim j}, \overline{w}_{\sim j})\label{eqn:mS-dyadic-deriv-on-LS-ll-sim-sim}\\
    &+ Q_{\sim j} T_{m_j^{\cS}}(u_{\ll j}, \partial_x (LS)_{\ll j}, \overline{w}_{\sim j})\label{eqn:mS-dyadic-deriv-on-LS-ll-ll-sim}\\
    &+ T_{m_j^{\cS}}(u_{\sim j}, \partial_x (LS)_{\lesssim j}, \overline{w}_{\ll j})\label{eqn:mS-dyadic-deriv-on-LS-sim-ll}\\
    &+ Q_{\sim j} T_{m_j^{\cS}}(u_{\sim j}, \partial_x (LS)_{\sim j}, \overline{w}_{\sim j})\label{eqn:mS-dyadic-deriv-on-LS-all-sim}
\end{align}\end{subequations}
For~\eqref{eqn:mS-dyadic-deriv-on-LS-ll-sim-sim}, we have
\begin{subequations}\begin{align}
    \eqref{eqn:mS-dyadic-deriv-on-LS-ll-sim-sim}
    =& T_{m_j^{\cS}}( \chi_{> j}u_{\ll j}, \partial_x (LS)_{\sim j}, \overline{w}_{\sim j})\label{eqn:mS-dyadic-deriv-on-LS-ll-sim-sim-1}\\
    &+ \chi_{[j - 50 ,  j+10]}T_{m_j^{\cS}}( \chi_{[j - 40 ,  j]}u_{\ll j}, \partial_x (LS)_{\sim j}, \overline{w}_{\sim j})\label{eqn:mS-dyadic-deriv-on-LS-ll-sim-sim-2}\\
    &+ T_{m_j^{\cS}}( \chi_{< j - 40}u_{\ll j}, \partial_x (LS)_{\sim j}, \chi_{< j - 30} \overline{w}_{\sim j})\label{eqn:mS-dyadic-deriv-on-LS-ll-sim-sim-3}\\
    &+ \bnpl\notag
\end{align}\end{subequations}
The first and third terms are straightforward to bound:
\begin{equation*}\begin{split}
    \lVert \eqref{eqn:mS-dyadic-deriv-on-LS-ll-sim-sim-1} \rVert_{L^2} \lesssim& \lVert \chi_{> j} u_{\ll j} \rVert_{L^\infty} \lVert \partial_x (LS)_{\sim j} \rVert_{L^\infty} \lVert w_{\sim j} \rVert_{L^2}\\
    \lesssim& M\epsilon^4 t^{-4/3} 2^{-j} \lVert xg \rVert_{L^2}\\
    \lVert \eqref{eqn:mS-dyadic-deriv-on-LS-ll-sim-sim-3} \rVert_{L^2} \lesssim& \lVert u_{\ll j} \rVert_{L^\infty} \lVert \partial_x (LS)_{\sim j} \rVert_{L^\infty} \lVert \chi_{< j - 30} w_{\sim j} \rVert_{L^2}\\
    \lesssim& M\epsilon^4 t^{-7/6} 2^{-j/2} \lVert xg \rVert_{L^2}
\end{split}\end{equation*}
For the second term, almost orthogonality implies that
\begin{equation*}\begin{split}
    \biggl\lVert \sum_{j} \eqref{eqn:mS-dyadic-deriv-on-LS-ll-sim-sim-2} \biggr\rVert_{L^2} \lesssim& \biggl( \sum_{2^j \gtrsim t^{-1/3}} \lVert \chi_{[j - 40 ,  j]}u_{\ll j} \rVert_{L^\infty}^2 \lVert \partial_x (LS)_{\sim j}\rVert_{L^\infty}^2 \lVert {w}_{\sim j} \rVert_{L^2}^2\biggr)^{1/2}\\
    \lesssim& M\epsilon^4 t^{-1} \lVert xg \rVert_{L^2}
\end{split}\end{equation*}
which gives the required bound for~\eqref{eqn:mS-dyadic-deriv-on-LS-ll-sim-sim}.  Turning to~\eqref{eqn:mS-dyadic-deriv-on-LS-ll-ll-sim}, we have
\begin{subequations}\begin{align}
    \eqref{eqn:mS-dyadic-deriv-on-LS-ll-ll-sim} =& Q_{\sim j} T_{m_j^{\cS}}(u_{\ll j}, \partial_x (LS)_{\ll j}, \chi_{< j - 30} \overline{w}_{\sim j})\label{eqn:mS-dyadic-deriv-on-LS-ll-ll-sim-1}\\
    &+ Q_{\sim j} T_{m_j^{\cS}}(\chi_{\geq j - 40}u_{\ll j}, \chi_{\geq j - 40}\partial_x (LS)_{\ll j}, \chi_{\geq j - 30} \overline{w}_{\sim j})\label{eqn:mS-dyadic-deriv-on-LS-ll-ll-sim-2}\\
    &+ \bnpl\notag
\end{align}\end{subequations}
For~\eqref{eqn:mS-dyadic-deriv-on-LS-ll-ll-sim-1}, we have that
\begin{equation*}\begin{split}
    \biggl\lVert \sum_{j} \eqref{eqn:mS-dyadic-deriv-on-LS-ll-ll-sim-1} \biggr\rVert_{L^2}\lesssim& \biggl(\sum_{2^j \gtrsim t^{-1/3}}\lVert u_{\ll j} \rVert_{L^\infty} \lVert \partial_x (LS)_{\ll j} \rVert_{L^\infty} \lVert  \chi_{< j - 30} {w}_{\sim j}\rVert_{L^2}^2\biggr)^{1/2}\\
    \lesssim& M\epsilon^4 t^{-1} \lVert xg \rVert_{L^2}
\end{split}\end{equation*}
while for~\eqref{eqn:mS-dyadic-deriv-on-LS-ll-ll-sim-2}, we have the bound
\begin{equation*}\begin{split}
    \left\lVert \eqref{eqn:mS-dyadic-deriv-on-LS-ll-ll-sim-2} \right\rVert_{L^2} \lesssim&  \lVert \chi_{\geq j - 40} u_{\ll j} \rVert_{L^\infty} \lVert \chi_{\geq j - 40} \partial_x(LS)_{\ll j} \rVert_{L^\infty} \lVert w_{\sim j} \rVert_{L^2}\\
    \lesssim& M\epsilon^4 t^{-4/3} 2^{-j}\lVert xg \rVert_{L^2}
\end{split}\end{equation*}
For~\eqref{eqn:mS-dyadic-deriv-on-LS-sim-ll}, we write
\begin{subequations}\begin{align}
    \eqref{eqn:mS-dyadic-deriv-on-LS-sim-ll} =& T_{m_j^{\cS}}((1 - \chi_{[j - 20 ,  j + 20]})u_{\sim j}, \partial_x (LS)_{\lesssim j}, \overline{w}_{\ll j}) \label{eqn:mS-dyadic-deriv-on-LS-sim-ll-1}\\
    &+ \chi_{[j - 30 ,  j + 30]} T_{m_j^{\cS}}(\chi_{[j - 20 ,  j + 20]}u_{\sim j}, \chi_{[j - 30 ,  j + 30]}\partial_x (LS)_{\lesssim j}, \chi_{[j - 30 ,  j + 30]} \overline{w}_{\ll j}) \label{eqn:mS-dyadic-deriv-on-LS-sim-ll-2}\\
    &+ \bnpl\notag
\end{align}\end{subequations}
For the first term,
\begin{equation*}\begin{split}
    \lVert \eqref{eqn:mS-dyadic-deriv-on-LS-sim-ll-1} \rVert_{L^2} \lesssim& \lVert \chi_{[j - 20, j + 20]} u_{\sim j} \rVert_{L^\infty} \lVert \partial_x(LS)_{\lesssim j} \rVert_{L^\infty} \lVert w_{\ll j} \rVert_{L^2}\\
    \lesssim& M\epsilon^4 t^{-7/6} 2^{-j/2} \lVert xg \rVert_{L^2}
\end{split}\end{equation*}
which is sufficient.  Turning to~\eqref{eqn:mS-dyadic-deriv-on-LS-sim-ll-2}, we find that
\begin{equation*}\begin{split}
    \biggl\lVert \sum_{j} \eqref{eqn:mS-dyadic-deriv-on-LS-sim-ll-2} \biggr\rVert_{L^2} \lesssim& \biggl( \smashoperator[r]{\sum_{2^j \gtrsim t^{-1/3}}} \lVert u_{\sim j} \rVert_{L^\infty}^2 \lVert \chi_{[j - 30 ,  j + 30]}\partial_x (LS)_{\lesssim j} \rVert_{L^\infty}^2 \lVert \chi_{[j - 30 ,  j + 30]} w_{\ll j} \rVert_{L^2}^2\biggr)^{\frac{1}{2}}\\
    \lesssim& M\epsilon^4 t^{-1} \lVert xg \rVert_{L^2}
\end{split}\end{equation*}
which completes the bound for~\eqref{eqn:mS-dyadic-deriv-on-LS-sim-ll}.  Finally, for~\eqref{eqn:mS-dyadic-deriv-on-LS-all-sim}, we use almost orthogonality and~\eqref{eqn:cubic-deriv-bounds-compendium} to conclude that
\begin{equation*}\begin{split}
    \biggl\lVert \sum_{j} \eqref{eqn:mS-dyadic-deriv-on-LS-all-sim} \biggr\rVert_{L^2} \lesssim& \biggl(\sum_{2^j \gtrsim t^{-1/3}} \lVert u_{\sim j}\rVert_{L^\infty} \lVert \partial_x(LS)_{\sim j}\rVert_{L^\infty} \lVert w_{\sim j}\rVert_{L^2}^2\biggr)^{1/2}\\
    \lesssim& M\epsilon^4 t^{-1} \lVert xg \rVert_{L^2}
\end{split}\end{equation*}

\subsubsection{The bound for~\texorpdfstring{\eqref{eqn:mS-dyadic-deriv-on-S-and-Lw}}{(\ref{eqn:mS-dyadic-deriv-on-S-and-Lw}}}  Using the support condition for $m_j^{\cS}$, we can decompose each summand as
\begin{subequations}\begin{align}
    T_{m_j^\cS}(u, \partial_x u, \overline{Lw}) =&
    T_{m_j^\cS}(u_{\sim j}, \partial_x u_{\lesssim j}, \overline{Lw})\label{eqn:mS-dyadic-deriv-and-Lw-i}\\
    &+T_{m_j^\cS}(u_{\ll j}, \partial_x u_{\sim j}, (\overline{Lw})_{\sim j})\label{eqn:mS-dyadic-deriv-and-Lw-ii}\\
    &+ Q_{\sim j} T_{m_j^\cS}(u_{\ll j}, \partial_x u_{\ll j}, (\overline{Lw})_{\sim j})\label{eqn:mS-dyadic-deriv-and-Lw-iii}
\end{align}\end{subequations}
For~\eqref{eqn:mS-dyadic-deriv-and-Lw-i}, we introduce the further decomposition
\begin{subequations}\begin{align} 
    \eqref{eqn:mS-dyadic-deriv-and-Lw-i} =& \chi_{[j - 40 ,  j + 40]}T_{m_j^\cS}( \chi_{[j - 30 ,  j + 30]} u_{\sim j}, \partial_x u_{\ll j}, \chi_{[j - 40 ,  j + 40]} \overline{Lw})\label{eqn:mS-dyadic-deriv-and-Lw-i-1}\\
    &+ T_{m_j^\cS}((1 - \chi_{[j - 30 ,  j + 30]})u_{\sim j}, \partial_x u_{\ll j}, \overline{Lw})\label{eqn:mS-dyadic-deriv-and-Lw-i-2}\\
    &+ \bnpl\notag
\end{align}\end{subequations}
The second term can be controlled by a directly using the decay estimates.  For the first term, we use almost orthogonality to obtain the bound
\begin{equation*}\begin{split}
    \biggl\lVert \sum_{j} \eqref{eqn:mS-dyadic-deriv-and-Lw-i-1} \biggr\rVert_{L^2} \lesssim& \biggl( \sum_{2^j \gtrsim t^{-1/3}} \lVert u_{\sim j}\rVert_{L^\infty}^2 \lVert \partial_x u_{\lesssim j}\rVert_{L^\infty}^2  \lVert \chi_{[j - 40 ,  j + 40]} Lw \rVert_{L^2}^2\biggr)^{1/2}\\
    \lesssim& M^2\epsilon^2 t^{-1} \lVert xg \rVert_{L^2}
\end{split}\end{equation*}
The bounds for~\eqref{eqn:mS-dyadic-deriv-and-Lw-ii} and ~\eqref{eqn:mS-dyadic-deriv-and-Lw-iii} follow from similar reasoning.

\subsubsection{The bound for~\texorpdfstring{\eqref{eqn:mS-dyadic-LS-and-outside-w}}{(\ref{eqn:mS-dyadic-LS-and-outside-w})}} It only remains to consider~\eqref{eqn:mS-dyadic-LS-and-outside-w}.  We can use the support restrictions on the $m_j^\cS$ to write
\begin{subequations}\begin{align}
    T_{m_j^\cS}(LS, \partial_xS, \overline{w}) =& Q_{\sim j} T_{m_j^\cS}( \partial_x(LS)_{\sim j}, S_{\sim j}, \overline{w}_{\sim j}) \label{eqn:mS-dyadic-LS-and-outside-w-i}\\
    & + T_{m_j^\cS}( \partial_x(LS)_{\sim j}, \partial_x S_{\sim j}, \overline{w}_{\ll j}) \label{eqn:mS-dyadic-LS-and-outside-w-ii}\\
    & + T_{m_j^\cS}( (LS)_{\ll j}, \partial_x S_{\sim j}, \partial_x\overline{w}_{\lesssim j}) \label{eqn:mS-dyadic-LS-and-outside-w-iii}\\
    & + T_{m_j^\cS}(\partial_x (LS)_{\sim j}, S_{\ll j}, \overline{w}_{\ll j})\label{eqn:mS-dyadic-LS-and-outside-w-iv}\\
    &+ Q_{\sim j}T_{m_j^\cS}( (LS)_{\ll j}, S_{\ll j}, \partial_x \overline{w}_{\sim j})\label{eqn:mS-dyadic-LS-and-outside-w-v}
\end{align}
\end{subequations}
Ignoring complex conjugation and permuting the arguments of the pseudoproducts, we see that the terms~\eqref{eqn:mS-dyadic-LS-and-outside-w-i},~\eqref{eqn:mS-dyadic-LS-and-outside-w-ii}, and~\eqref{eqn:mS-dyadic-LS-and-outside-w-iii} are essentially identical to~\eqref{eqn:mS-dyadic-LS-sim-sim},~\eqref{eqn:mS-dyadic-deriv-on-LS-sim-ll}, and~\eqref{eqn:mS-dyadic-LS-sim-ll}, respectively.  Therefore, we will focus on the last two terms.  For~\eqref{eqn:mS-dyadic-LS-and-outside-w-iv}, we write
\begin{subequations}\begin{align}
    \eqref{eqn:mS-dyadic-LS-and-outside-w-iv} =& T_{m_j^{\cS}}( \chi_{> j + 20} \partial_x(LS)_{\sim j}, \chi_{> j + 10} S_{\ll j}, \overline{w}_{\ll j})\label{eqn:mS-dyadic-LS-and-outside-w-iv-1}\\
    &+ \chi_{[j - 30, j+30]}T_{m_j^{\cS}}( \chi_{[j+20, j-20]} \partial_x(LS)_{\sim j}, \chi_{[j - 30, j+30]} S_{\ll j}, \chi_{[j - 30, j+30]}\overline{w}_{\ll j})\label{eqn:mS-dyadic-LS-and-outside-w-iv-2}\\
    &+ \sum_{k < j - 20} T_{m_j^{\cS}}( \chi_{k} \partial_x(LS)_{\sim j}, \chi_{\sim k} S_{\ll j}, \chi_{\sim k} \overline{w}_{\ll j})\label{eqn:mS-dyadic-LS-and-outside-w-iv-3}\\
    &+ \bnpl\notag
\end{align}
\end{subequations}
The first and last terms can be handled using decay estimates:
\begin{equation*}\begin{split}
    \lVert \eqref{eqn:mS-dyadic-LS-and-outside-w-iv-1} \rVert_{L^2} \lesssim& \lVert \chi_{> j + 20} \partial_x(LS)_{\sim j} \rVert_{L^\infty} \lVert\chi_{> j + 10} S_{\ll j} \rVert_{L^\infty} \lVert  \overline{w}_{\ll j} \rVert_{L^2}\\
    \lesssim& \epsilon^4 t^{-4/3} 2^{-j} \lVert xg \rVert_{L^2}\\
    \lVert \eqref{eqn:mS-dyadic-LS-and-outside-w-iv-3} \rVert_{L^2} \lesssim& \sum_{k < j - 30} \lVert \partial_x(LS)_{\sim j}\rVert_{L^\infty} \lVert \chi_{\sim k} S_{\ll j}\rVert_{L^\infty} \lVert \chi_{\sim k} \overline{w}_{\ll j} \rVert_{L^2}\\
    \lesssim& \epsilon^2 t^{-7/6} 2^{-j/2} \lVert xg\rVert_{L^2} 
\end{split}
\end{equation*}
and for the second term, we use almost orthogonality to get the bound
\begin{equation*}\begin{split}
    \biggl\lVert \sum_j \eqref{eqn:mS-dyadic-LS-and-outside-w-iv-2} \biggr\rVert_{L^2} \lesssim& \biggl( \smashoperator[r]{\sum_{2^j \gtrsim t^{-1/3}}} \lVert \partial_x (LS)_{\sim j} \rVert_{L^\infty}^2 \lVert \chi_{[j-30, j+30]} S_{\ll j} \rVert_{L^\infty}^2 \lVert \chi_{[j - 30, j+30]} w_{\ll j} \rVert_{L^2}^2\biggr)^{\frac{1}{2}}\\
    \lesssim& \epsilon^4 t^{-1} \lVert xg \rVert_{L^2}
\end{split}
\end{equation*}
Turning to~\eqref{eqn:mS-dyadic-LS-and-outside-w-v}, we introduce the decomposition
\begin{subequations}\begin{align}
    \eqref{eqn:mS-dyadic-LS-and-outside-w-v} =& Q_{\sim j} T_{m_j^\cS}(\chi_{> j + 30} (LS)_{\ll j}, \chi_{> j + 20} S_{\ll j}, \partial_x w_{\sim j}) \label{eqn:mS-dyadic-LS-and-outside-w-v-1}\\
    &+ Q_{\sim j} T_{m_j^\cS}(\chi_{[j - 30, j+ 30]} (LS)_{\ll j}, \chi_{[j - 40, j+ 40]} S_{\ll j}, \partial_x w_{\sim j}) \label{eqn:mS-dyadic-LS-and-outside-w-v-2}\\
    &+ \sum_{k < j - 30} Q_{\sim j} T_{m_j^\cS}(\chi_{k} (LS)_{\ll j}, \chi_{\sim k} S_{\ll j}, \chi_{\sim k} \partial_x w_{\sim j}) \label{eqn:mS-dyadic-LS-and-outside-w-v-3}
\end{align}
\end{subequations}
The first term is easily handled:
\begin{equation*}\begin{split}
    \lVert \eqref{eqn:mS-dyadic-LS-and-outside-w-v-1} \rVert_{L^2} \lesssim& \lVert \chi_{> j + 30} (LS)_{\ll j} \rVert_{L^\infty}\lVert \chi_{> j + 20} S_{\ll j}\rVert_{L^\infty}\lVert \partial_x w_{\sim j} \rVert_{L^2}\\
    \lesssim& \epsilon^4 t^{-4/3} 2^{-j} \lVert xg \rVert_{L^2}
\end{split}
\end{equation*}
For the remaining terms, we take advantage of the almost orthogonality coming from the frequency projection to obtain the bounds
\begin{equation*}\begin{split}
    \biggl\lVert \sum_j \eqref{eqn:mS-dyadic-LS-and-outside-w-v-2} \biggr\rVert_{L^2} \lesssim& \biggl(\smashoperator[r]{\sum_{2^j \gtrsim t^{-1/3}}} \lVert \chi_{[j - 30, j+ 30]} (LS)_{\ll j}\rVert_{L^\infty}^2 \lVert\chi_{[j - 40, j+ 40]} S_{\ll j}\rVert_{L^\infty}^2\lVert \partial_x w_{\sim j} \rVert_{L^2}^2\biggr)^{\frac{1}{2}}\\
    \lesssim& \epsilon^4 t^{-1} \lVert xg \rVert_{L^2}\\
    \biggl\lVert \sum_j \eqref{eqn:mS-dyadic-LS-and-outside-w-v-2} \biggr\rVert_{L^2} \lesssim& \Biggl(\sum_{2^j \gtrsim t^{-1/3}} \biggl(\smashoperator[r]{\sum_{k < j - 30}} \lVert \chi_k (LS)_{\ll j}\rVert_{L^\infty} \lVert \chi_{\sim k} S_{\ll j} \rVert_{L^\infty} \lVert \chi_{\sim k} \partial_x w_{\sim j} \rVert_{L^2}\biggr)^2\Biggr)^{\frac{1}{2}}\\
    \lesssim& \epsilon^4 t^{-1} \lVert xg \rVert_{L^2}
\end{split}
\end{equation*}

\subsection{The time non-resonant multiplier}\label{sec:xf-tr}

We control the term~\eqref{eqn:xg-time-non-res-term} by integrating by parts in $s$.  Defining $m_s = -\frac{i(\xi-\eta-\sigma) \partial_\xi \phi \chi_s^\cT}{\phi}$, we can integrate by parts to obtain
\begin{subequations}\begin{align}
    \int_1^t\eqref{eqn:xg-time-non-res-term}\,ds =& \mathord{-} s \Re\langle e^{-s\partial_x^3} xg, T(s)\rangle\Bigr|_{s=1}^{s=t} \label{eqn:time-non-res-bdy}\\
    &+ \Re\int_1^t \langle e^{-s\partial_x^3} xg, T(s) \rangle \,ds\label{eqn:time-non-res-no-s}\\
    &+ \Re\int_1^t s\langle e^{-s\partial_x^3} xg, \tilde{T}(s) \rangle \,ds\label{eqn:time-non-res-symbol-deriv}\\
    &+ \Re\int_1^t s\langle e^{-s\partial_x^3} xg, \mathring{T}(s)\rangle \,ds \label{eqn:time-non-res-arg-deriv}\\
    &+ \Re\int_1^t s\langle e^{-s\partial_x^3} x\partial_s g, T(s) \rangle \,ds \label{eqn:time-non-inner-product-deriv}
    \end{align}\end{subequations}
where, to ease notation, we have written
\begin{equation*}\begin{split}
    T(s) =& T_{m_s}(u, w, \overline{u}) + T_{m_s}(u, S, \overline{w}) + T_{m_s}(w, S, \overline{u})\\
    \tilde{T}(s) =& T_{\partial_s m_s}(u, w, \overline{u}) + T_{\partial_s m_s}(u, S, \overline{w}) + T_{\partial_s m_s}(w, S, \overline{u})\\
    \mathring{T}(s) =& T_{m_s}(e^{-s\partial_x^3} \partial_s f, w, \overline{u}) + T_{m_s}(u, e^{-s\partial_x^3} \partial_s g, \overline{u}) + T_{m_s}(u, w, \overline{e^{-s\partial_x^3} \partial_s f}) \\
    &+ T_{m_s}(e^{-s\partial_x^3} \partial_s f, S, \overline{w})+ T_{m_s}(u, e^{-s\partial_x^3} \partial_s h, \overline{w})+ T_{m_s}(u, S, \overline{e^{-s\partial_x^3} \partial_s g})\\
    &+ T_{m_s}(e^{-s\partial_x^3} \partial_s g, S, \overline{u}) + T_{m_s}(w, e^{-s\partial_x^3} \partial_s h, \overline{u})  + T_{m_s}(w, S, \overline{e^{-s\partial_x^3} \partial_s f}) 
\end{split}\end{equation*}
Note that $m_s$ satisfies Coifman-Meyer type bounds uniformly in time:
\begin{equation}\label{eqn:m-t-symbol-bounds}
    \left|(\xi^2 + \eta^2 + \sigma^2)|^{|\alpha|/2} \partial_{\xi,\eta,\sigma}^\alpha m_s\right| \lesssim_{\alpha} 1
\end{equation}
Thus, we can decompose pseudoproducts involving $m_s$ as
\begin{equation} \label{eqn:m-t-pseudoprod-expansion}
    T_{m_s}(p,q,r) = \sum_{2^j \gtrsim s^{-1/3}} T_{m_j}(p_{\sim j},q,r) + T_{m_j}(p,q_{\sim j},r) + T_{m_j}(p,q,r_{\sim j})
\end{equation}

\subsubsection{The bound for~\texorpdfstring{\eqref{eqn:time-non-res-bdy}}{(\ref{eqn:time-non-res-bdy})}}\label{sec:T-m-s-bounds}  We will first give the argument for the boundary $s = t$.  Since none of the frequencies $\eta$, $\sigma$, and $\xi - \eta - \sigma$ play a distinguished role, and since $S$ and $u$ obey the same decay estimates, it suffices to prove bounds for $T_{m_s}(u,w,\overline{u})$.  Using~\eqref{eqn:m-t-pseudoprod-expansion}, we find that
\begin{subequations}\begin{align}
        T_{m_t}(u,w,\overline{u}) =& \sum_{2^j \gtrsim t^{-1/3}} T_{m_j}(u_{\sim j}, w_{\lesssim j}, \overline{u}_{\sim j})\label{eqn:m-t-dyadic-sim-sim}\\
        & \qquad+ T_{m_j}(u_{\sim j}, w_{\lesssim j}, \overline{u}_{\ll j})\label{eqn:m-t-dyadic-sim-ll}\\
        &\qquad+ T_{m_j}(u_{\ll j}, w_{\lesssim j}, \overline{u}_{\sim j})\label{eqn:m-t-dyadic-ll-sim}\\
        &\qquad+ Q_{\sim j} T_{m_j}(u_{\ll j}, w_{\sim j}, \overline{u}_{\ll j})\label{eqn:m-t-dyadic-ll-ll}
\end{align}
\end{subequations}
The terms~\cref{eqn:m-t-dyadic-ll-sim,eqn:m-t-dyadic-sim-ll} are controlled using the same arguments as for~\eqref{eqn:mS-no-deriv-low-freq}:
\begin{equation*}
    \lVert \eqref{eqn:m-t-dyadic-ll-sim} \rVert_{L^2} + \lVert \eqref{eqn:m-t-dyadic-sim-ll} \rVert_{L^2} \lesssim M^2\epsilon^2 t^{-1} \lVert xg \rVert_{L^2}
\end{equation*}
Similarly, using the same argument as for~\eqref{eqn:mS-SSw-i}, we find that
\begin{equation*}
    \lVert \eqref{eqn:m-t-dyadic-ll-ll} \rVert_{L^2} \lesssim M^2\epsilon^2 t^{-1} \lVert xg \rVert_{L^2}
\end{equation*}
It only remains to bound~\eqref{eqn:m-t-dyadic-sim-sim}.  Here, we have that
\begin{subequations}\begin{align}
    \eqref{eqn:m-t-dyadic-sim-sim} =& \sum_{2^j \gtrsim t^{-1/3}}\chi_{[j - 30 ,  j+ 30]} T_{m_j}(\chi_{[j - 20 ,  j + 20]} u_{\sim j}, \chi_{[j - 30 ,  j+30]} w_{\lesssim j}, \overline{u}_{\sim j})\label{eqn:m-t-dyadic-sim-sim-1}\\
    &\qquad+ 
    T_{m_j}((1- \chi_{[j - 20 ,  j + 20]}) u_{\sim j},w_{\lesssim j}, \overline{u}_{\sim j})\label{eqn:m-t-dyadic-sim-sim-2}\\
    &\qquad + \bnpl\notag
\end{align}\end{subequations}
Both terms can be estimated using the same types of arguments we have employed previously, yielding the bounds
\begin{equation*}\begin{split}
    \lVert \eqref{eqn:m-t-dyadic-sim-sim-1} \rVert_{L^2} \lesssim& \biggl(\sum_{2^j \gtrsim t^{-1/3}} \lVert u_{\sim j} \rVert_{L^\infty}^4 \lVert  \chi_{[j - 30 ,  j+30]} w_{\lesssim j} \rVert_{L^2}^2 \biggr)^{1/2}\\
    \lesssim& M^2\epsilon^2 t^{-1} \lVert xg \rVert_{L^2}\\
    \lVert \eqref{eqn:m-t-dyadic-sim-sim-2} \rVert_{L^2} \lesssim& \sum_{2^j \gtrsim t^{-1/3}} \lVert (1 - \chi_{[j - 20 ,  j + 20]}) u_{\sim j} \rVert_{L^\infty} \lVert w_{\lesssim j} \rVert_{L^2} \lVert u_{\sim j} \rVert_{L^\infty}\\
    \lesssim& M^2\epsilon^2 t^{-1} \lVert xg \rVert_{L^2}
\end{split}\end{equation*}
Combining these estimates, we see that
\begin{equation}\label{eqn:T-t-bound}
    \lVert T(t) \rVert_{L^2} \lesssim M^2 \epsilon^2 t^{-1} \lVert xg \rVert_{L^2}
\end{equation}
so, by Cauchy-Schwarz,
\begin{equation*}
    | t \langle e^{-t\partial_x^3} xg, T(t)\rangle | \lesssim M^2\epsilon^2 \lVert xg \rVert_{L^2}^2
\end{equation*}
Repeating the above arguments at $s = 1$ and recalling~\eqref{eqn:xg-t-1-bd}, we see that
\begin{equation*}
    |\langle xg(x, 1), T(1)\rangle| \lesssim \epsilon^4
\end{equation*}
so we conclude that
\begin{equation*}
    |\eqref{eqn:time-non-res-bdy}| \lesssim M^2\epsilon^2 \lVert xg \rVert_{L^2}^2 + \epsilon^4
\end{equation*}
Since $M^2\epsilon^2 \ll 1$, the first term can be absorbed into the left-hand side of~\eqref{eqn:xg-desired-bound}, and the second term is better than required.

\subsubsection{The bound for~\texorpdfstring{\eqref{eqn:time-non-res-no-s}}{(\ref{eqn:time-non-res-no-s})}}
By using the bound for $\lVert T(s) \rVert_{L^2}$ derived above, we have at once that
\begin{equation*}
    |\eqref{eqn:time-non-res-no-s}| \lesssim \int_1^t M^2 \epsilon^2 s^{-1} \lVert xg(s) \rVert_{L^2}^2 \,ds
\end{equation*}
in agreement with~\eqref{eqn:xg-desired-bound}.

\subsubsection{The bound for~\texorpdfstring{\eqref{eqn:time-non-res-symbol-deriv}}{(\ref{eqn:time-non-res-symbol-deriv})}} A simple computation shows that $s\partial_s m_s$ also obeys symbol bounds of the form~\eqref{eqn:m-t-symbol-bounds}, so $\tilde{T}(s)$ obeys the same estimates as $T(s)$, giving us the bound for~\eqref{eqn:time-non-res-symbol-deriv}.

\subsubsection{The bound for~\texorpdfstring{\eqref{eqn:time-non-res-arg-deriv}}{(\ref{eqn:time-non-res-arg-deriv})}}  By differentiating $f$, $g$, and $h$ in time, we find that
\begin{equation*}\begin{split}
    e^{-s\partial_x^3}\partial_s f =& |u|^2 \partial_x u\\
    e^{-s\partial_x^3}\partial_s h =& |S|^2 \partial_x S + D_pS \partial_t \hat{u}(0,t)\\
    e^{-s\partial_x^3} \partial_s g =& |u|^2 \partial_x w + (w\overline{u} + u\overline{w}) \partial_x S - D_p S \partial_t \hat{u}(0,t)
\end{split}\end{equation*}
Thus, we can write
\begin{subequations}\begin{align}
    \mathring{T}(s) =& \sum_{j} T_{m_j}((|u|^2 \partial_x u), w, \overline{u})\label{eqn:G-i}\\
    &\qquad+ T_{m_j}(u, D_pS \hat{u}(0,t), \overline{u})\label{eqn:G-ii}\\
    &\qquad + T_{m_j}(u,|u|^2 \partial_x w, \overline{u}) \label{eqn:G-iii}\\
    &\qquad + T_{m_j}(u,u \partial_x S \overline{w}, \overline{u}) \label{eqn:G-iv}\\
    &+ \{\text{similar terms}\} \nonumber
\end{align}\end{subequations}
The desired bound on~\eqref{eqn:time-non-res-arg-deriv} will follow once we show that
\begin{equation*}
    \lVert \mathring{T}(s) \rVert_{L^2} \lesssim M^2\epsilon^2 s^{-2} \lVert xg \rVert_{L^2} + M^2\epsilon^3s^{-5/6 - \beta}
\end{equation*}
so it suffices to bound the quantities~\cref{eqn:G-i,eqn:G-ii,eqn:G-iii,eqn:G-iv} in $L^2$.  For~\eqref{eqn:G-i}, we can write
\begin{subequations}\begin{align}
    T_{m_j}\left((|u|^2\partial_x u)_{\lesssim j}, w_{\lesssim j}, \overline{u}_{\lesssim j}\right) =&  Q_{\sim j}T_{m_j}\left((|u|^2\partial_x u)_{\ll j}, w_{\sim j}, \overline{u}_{\ll j}\right)\label{eqn:time-non-res-arg-deriv-cubic-ll-sim-ll} \\
    &+  T_{m_j}\left((|u|^2\partial_x u)_{\lesssim j}, w_{\sim j}, \overline{u}_{\sim j}\right) \label{eqn:time-non-res-arg-deriv-cubic-lesssim-sim-sim} \\
    &+ T_{m_j}\left((|u|^2\partial_x u)_{\lesssim j}, w_{\ll j}, \overline{u}_{\sim j}\right)\label{eqn:time-non-res-arg-deriv-cubic-lesssim-ll-sim}\\
    &+ T_{m_j}\left((|u|^2\partial_x u)_{\sim j}, w_{\lesssim j}, \overline{u}_{\ll j}\right)\label{eqn:time-non-res-arg-deriv-cubic-sim-lesssim-ll}
\end{align}\end{subequations}
The term~\eqref{eqn:time-non-res-arg-deriv-cubic-ll-sim-ll} is similar to~\eqref{eqn:mS-SSw-i}:
we write
\begin{subequations}\begin{align}
    \eqref{eqn:time-non-res-arg-deriv-cubic-ll-sim-ll} =& Q_{\sim j}T_{m_j}(\chi_{< j - 30} (|u|^2\partial_x u)_{\ll j}, \chi_{< j - 20} w_{\sim j}, \overline{u}_{\ll j})\label{eqn:time-non-res-arg-deriv-cubic-ll-sim-ll-1}\\
    &+ Q_{\sim j}T_{m_j}(\chi_{\geq j - 30} (|u|^2\partial_x u)_{\ll j},  w_{\sim j}, \chi_{\geq j - 40} \overline{u}_{\ll j})\label{eqn:time-non-res-arg-deriv-cubic-ll-sim-ll-2}\\
    &+ \bnpl\notag
\end{align}\end{subequations}
Using~\eqref{eqn:cubic-deriv-bounds-compendium} to control the cubic $(|u|^2 \partial_x u)_{\ll j}$, we find that
\begin{equation*}\begin{split}
    \biggl\lVert \sum_{j}\eqref{eqn:time-non-res-arg-deriv-cubic-ll-sim-ll-1} \biggr\rVert_{L^2} \lesssim& \biggl( \sum_{2^j\gtrsim s^{-1/3}} \lVert (|u|^2\partial_x u)_{\ll j}\rVert_{L^\infty}^2 \lVert \chi_{< j - 20} w_{\sim j}\rVert_{L^2}^2 \lVert \overline{u}_{\ll j} \rVert_{L^\infty}^2\biggr)^{1/2}\\
    \lesssim& M^4\epsilon^4 s^{-2} \lVert xg \rVert_{L^2}
\end{split}\end{equation*}
as required.  Similar reasoning also gives us the bound
\begin{equation*}\begin{split}
    \biggl\lVert \sum_{j}\eqref{eqn:time-non-res-arg-deriv-cubic-ll-sim-ll-2} \biggr\rVert_{L^2} \lesssim& \biggl( \sum_{2^j\gtrsim s^{-1/3}} \lVert\chi_{\geq j - 30} (|u|^2\partial_x u)_{\ll j}\rVert_{L^\infty}^2 \lVert  w_{\sim j}\rVert_{L^2}^2 \lVert \chi_{\geq j - 40} \overline{u}_{\ll j} \rVert_{L^\infty}^2\biggr)^{1/2}\\
    \lesssim& M^4\epsilon^4 s^{-2} \lVert xg \rVert_{L^2}
\end{split}\end{equation*}
which gives us the required bound for~\eqref{eqn:time-non-res-arg-deriv-cubic-ll-sim-ll}.  Turning to~\eqref{eqn:time-non-res-arg-deriv-cubic-lesssim-sim-sim}, we find that
\begin{subequations}\begin{align}
    \eqref{eqn:time-non-res-arg-deriv-cubic-lesssim-sim-sim} =& T_{m_j}((|u|^2 \partial_x u)_{\lesssim j}, w_{\sim j},  (1 - \chi_{[j - 30 ,  j + 30]})\overline{u}_{\sim j})\label{eqn:time-non-res-arg-deriv-cubic-lesssim-sim-sim-1}\\
    &+ \chi_{[j - 40 ,  j + 40]}T_{m_j}(\chi_{[j - 40 ,  j + 40]}(|u|^2 \partial_x u)_{\lesssim j}, w_{\sim j},  \chi_{[j - 30 ,  j + 30]}\overline{u}_{\sim j})\label{eqn:time-non-res-arg-deriv-cubic-lesssim-sim-sim-2}\\
    &+ \bnpl\notag
\end{align}\end{subequations}
The term~\cref{eqn:time-non-res-arg-deriv-cubic-lesssim-sim-sim-1} is easily handled using linear and cubic dispersive estimates.
For the second term, the almost orthogonality of the terms gives
\begin{equation*}\begin{split}
    \biggl\lVert \sum_{j} \eqref{eqn:time-non-res-arg-deriv-cubic-lesssim-sim-sim-1} \biggr\rVert_{L^2}  \lesssim& \biggl( \sum_{2^j \gtrsim s^{-1/3}} \lVert \chi_{[j - 40 ,  j + 40]} (|u|^2 \partial_x u)_{\lesssim j}\rVert_{L^\infty}^2 \lVert w_{\sim j} \rVert_{L^2}^2\lVert {u}_{\sim j} \rVert_{L^\infty}^2 \biggr)^{1/2}\\
    \lesssim& M^4\epsilon^4 s^{-2} \lVert xg \rVert_{L^2}
\end{split}
\end{equation*}
which is sufficient, completing the bound for~\eqref{eqn:time-non-res-arg-deriv-cubic-lesssim-sim-sim}.  Similarly, for~\eqref{eqn:time-non-res-arg-deriv-cubic-lesssim-ll-sim}, we have that
\begin{subequations}\begin{align}
    \eqref{eqn:time-non-res-arg-deriv-cubic-lesssim-ll-sim} =& T_{m_j}\bigl( (|u|^2\partial_x u)_{\lesssim j}, w_{\ll j},  (1 - \chi_{[j - 30 ,  j + 30]})\overline{u}_{\sim j}\bigr)\label{eqn:time-non-res-arg-deriv-cubic-lesssim-ll-sim-1}\\
    &+\!\! \chi_{[j - 30 ,  j + 30]}T_{m_j}\bigl( \chi_{[j - 40 ,  j + 40]}(|u|^2\partial_x u)_{\lesssim j}, \chi_{[j - 40 ,  j + 40]}w_{\ll j},  \chi_{[j - 30 ,  j + 30]}\overline{u}_{\sim j}\bigr)\label{eqn:time-non-res-arg-deriv-cubic-lesssim-ll-sim-2}\\
    &{+} \bnpl\notag
\end{align}\end{subequations}
The estimate for~\eqref{eqn:time-non-res-arg-deriv-cubic-lesssim-ll-sim-1} is analogous to the one for~\eqref{eqn:time-non-res-arg-deriv-cubic-lesssim-sim-sim-1}, while for~\eqref{eqn:time-non-res-arg-deriv-cubic-lesssim-ll-sim-2}, we have
\begin{equation*}\begin{split}
    \biggl\lVert \sum_j \eqref{eqn:time-non-res-arg-deriv-cubic-lesssim-ll-sim-2} \biggr\rVert_{L^2} \lesssim& \biggl(\sum_{j} \lVert \chi_{[j - 40 ,  j + 40]}(|u|^2\partial_x u)_{\lesssim j}\rVert_{L^\infty}^2 \lVert \chi_{[j - 40 ,  j + 40]}w_{\ll j} \rVert_{L^2}^2 \lVert {u}_{\sim j}\rVert_{L^\infty}^2 \biggr)^{\frac{1}{2}}\\
    \lesssim& M^4\epsilon^4 s^{-2} \lVert xg \rVert_{L^2}^2
\end{split}\end{equation*}
Finally, for~\eqref{eqn:time-non-res-arg-deriv-cubic-sim-lesssim-ll} we find that
\begin{subequations}\begin{align}
    \eqref{eqn:time-non-res-arg-deriv-cubic-sim-lesssim-ll} =& \sum_{k < j - 30}\chi_{\sim k}T_{m_j}(\chi_k(|u|^2\partial_x u)_{\sim j}, \chi_{\sim k} w_{\sim j}, \chi_{\sim k} \overline{u}_{\ll j})\label{eqn:time-non-res-arg-deriv-cubic-sim-lesssim-ll-1}\\
    &+\!\! \chi_{[j - 40,  j + 40]}T_{m_j}(\chi_{[j - 30,  j + 30]}(|u|^2\partial_x u)_{\sim j}, \chi_{[j - 40,  j + 40]}w_{\lesssim j}, \chi_{[j - 40,  j + 40]}\overline{u}_{\ll j})\label{eqn:time-non-res-arg-deriv-cubic-sim-lesssim-ll-2}\\
    &+ T_{m_j}(\chi_{> j + 30}(|u|^2\partial_x u)_{\sim j}, w_{\lesssim j}, \chi_{> j + 20}\overline{u}_{\ll j})\label{eqn:time-non-res-arg-deriv-cubic-sim-lesssim-ll-3}\\
    &+ \bnpl\notag
\end{align}\end{subequations}
The estimate for~\eqref{eqn:time-non-res-arg-deriv-cubic-sim-lesssim-ll-2} is essentially the same as the one for~\eqref{eqn:time-non-res-arg-deriv-cubic-lesssim-ll-sim-2}, and the term~\eqref{eqn:time-non-res-arg-deriv-cubic-sim-lesssim-ll-3} can be bounded in essentially the same manner as~\eqref{eqn:time-non-res-arg-deriv-cubic-lesssim-ll-sim-1} and~\eqref{eqn:time-non-res-arg-deriv-cubic-lesssim-sim-sim-1}, so it only remains to bound~\eqref{eqn:time-non-res-arg-deriv-cubic-sim-lesssim-ll-1}.  For this term, the refined cubic estimate~\eqref{eqn:cubic-sim-w-deriv-small} gives us the bound
\begin{equation*}
    \lVert \chi_k (|u|^2\partial_x u)_{\sim j} \rVert_{L^\infty} \lesssim M^3\epsilon^3 s^{-11/6} 2^{-j/2} 2^{-k}
\end{equation*}
from which we deduce that
\begin{equation*}\begin{split}
    \lVert \eqref{eqn:time-non-res-arg-deriv-cubic-sim-lesssim-ll-1} \rVert_{L^2} 
    \lesssim& M^4\epsilon^4 s^{-13/6} 2^{-j/2} \lVert xg \rVert_{L^2}
\end{split}\end{equation*}
which completes the estimate for~\eqref{eqn:G-i}.

The term~\eqref{eqn:G-ii} can be controlled using the $L^p$ estimates for $D_p S$.  Let us write
\begin{subequations}\begin{align}
    T_{m_j}(u_{\lesssim j},  Q_{\lesssim j} D_p S \partial_t \hat{u}(0,t), \overline{u}_{\lesssim j}) =& T_{m_j}(u,  Q_{\sim j} D_p S \partial_t \hat{u}(0,t), \overline{u})\label{eqn:G-ii-1}\\
    &+ T_{m_j}(u_{\sim j},  Q_{\ll j} D_p S \partial_t \hat{u}(0,t), \overline{u})\label{eqn:G-ii-2}\\
    &+ T_{m_j}(u_{\ll j},  Q_{\ll j} D_p S \partial_t \hat{u}(0,t), \overline{u}_{\sim j})\label{eqn:G-ii-3}
\end{align}\end{subequations}
For~\eqref{eqn:G-ii-1}, we use the $L^6$ estimates for $u$ and $Q_{\sim j} (D_p S)$ (equations~\eqref{eqn:lp-decay} and~\eqref{eqn:D-p-S-L-p-freq-loc}, respectively) to obtain the bound
\begin{equation*}\begin{split}
    \biggl\lVert \sum_j\eqref{eqn:G-ii-1} \biggr\rVert_{L^2} \lesssim &  \sum_{2^j \gtrsim s^{-1/3}}\lVert u \rVert_{L^6}^2 \lVert Q_{\sim j} D_p S \rVert_{L^6} |\partial_s \hat{u}(0,s)|\\
    \lesssim& M^5 \epsilon^5 s^{-11/6 - \beta}
\end{split}\end{equation*}
Similar reasoning using the $L^6$ bounds for $u_{\sim j}$ and $Q_{\ll j} D_p S$ (\eqref{eqn:freq-loc-lp-decay} and~\eqref{eqn:D-p-S-L-p}) yields
\begin{equation*}\begin{split}
    \biggl\lVert \sum_{j} \eqref{eqn:G-ii-2} + \eqref{eqn:G-ii-3} \biggr\rVert_{L^2}
    \lesssim& M^5 \epsilon^5 s^{-11/6 - \beta}
\end{split}\end{equation*}

Let us now consider the term~\eqref{eqn:G-iii}.  The estimates here require us to obtain bounds for cubic expressions of the form $|u|^2 \partial_x w$.  Since $g \not\in X$, these bounds do not immediately follow from the work in~\Cref{sec:cubic-bounds}, so we instead argue directly. By combining estimates from~\eqref{eqn:u-lin-ests} and~\eqref{eqn:w-lin-ests-L-2}, we see that
\begin{equation}\label{eqn:cubic-u2-w-space-loc-bound}\begin{split}
	\lVert \chi_k |u|^2 \partial_x w \rVert_{L^\infty} \lesssim M^2 \epsilon^2 s^{-3/2} \lVert xg \rVert_{L^2}c_k
\end{split}\end{equation}
which immediately implies that
\begin{equation}\label{eqn:cubic-u2-w-bound}\begin{split}
	\lVert |u|^2 \partial_x w \rVert_{L^\infty}
	\lesssim& M^2 \epsilon^2 t^{-3/2} \lVert xg \rVert_{L^2}
\end{split}\end{equation}
We will also need refined estimates in the spirit of~\eqref{eqn:cubic-sim-with-k} for $\chi_k (|u|^2\partial_x w)_{\sim j}$ when $k < j - 30$.  To obtain these bounds, we perform a Littlewood-Paley decomposition to write
\begin{equation*}\begin{split}
    \chi_k (|u|^2\partial_x w)_{\sim j} =& \chi_k Q_{\sim j} \chi_{\sim k} \bigl( |u_{< j - 20}|^2 \partial_x w_{[j - 20 ,  j + 20]}\bigr)\\
    &+ 2\chi_k Q_{\sim j} \chi_{\sim k} \bigl( \Re(u_{< j - 20} \overline{u}_{[j - 20 ,  j + 20]}) \partial_x w_{< j - 20} \bigr)  + \bbetter
\end{split}\end{equation*}
Here, $\bbetter$ denotes more rapidly decaying terms which can be neglected.  The leading order terms we estimate as
\begin{equation*}\begin{split}
    \lVert \chi_{\sim k} |u_{< j - 20}|^2 \partial_x w_{[j - 20,  j + 20]} \rVert_{L^\infty} \lesssim& M^2 \epsilon^2 s^{-11/6} 2^{-k} \lVert xg \rVert_{L^2} c_j\\
    \lVert \chi_{\sim k} u_{< j - 20} \overline{u}_{[j - 20,  j + 20]} \partial_x w_{< j - 20} \rVert_{L^\infty} \lesssim& M^2\epsilon^2 s^{-11/6} 2^{k/2 - 3/2 j} \lVert xg \rVert_{L^2}
\end{split}\end{equation*}
so
\begin{equation}\label{eqn:cubic-u2-w-refined-bound}
    \lVert \chi_k (|u|^2\partial_x w)_{\sim j} \rVert_{L^\infty} \lesssim M^2 \epsilon^2 s^{-11/6} 2^{-k} \lVert xg \rVert_{L^2}c_j
\end{equation}
where we have used the fact that $2^{3/2(k - j)}$ is $\ell^2$ in $j$ for $k < j - 30$.  Now, let us write
\begin{subequations}\begin{align}
    T_{m_j}(u, |u|^2\partial_x w, \overline{u}) =& T_{m_j}(u_{\sim j}, |u|^2\partial_x w, \overline{u}_{\lesssim j})\label{eqn:G-iii-1}\\
    & + T_{m_j}(u_{\ll j}, |u|^2\partial_x w, \overline{u}_{\sim j})\label{eqn:G-iii-2}\\
    & + Q_{\sim j}T_{m_j}(u_{\ll j}, (|u|^2\partial_x w)_{\sim j}, \overline{u}_{\ll j})\label{eqn:G-iii-3}
\end{align}\end{subequations}
For the term~\eqref{eqn:G-iii-1}, we perform the further decomposition
\begin{subequations}\begin{align}
    \eqref{eqn:G-iii-1} =& \chi_{[j - 40,  j + 40]} T_{m_j}(u_{\sim j}, \chi_{[j - 40,  j + 40]} |u|^2\partial_x w,  \chi_{[j - 30,  j + 30]}\overline{u}_{\lesssim j})\label{eqn:G-iii-1-1}\\
    &+ T_{m_j}((1- \chi_{[j - 20,  j + 20]})u_{\sim j}, |u|^2\partial_x w, (1- \chi_{[j - 30,  j + 30]})\overline{u}_{\lesssim j})\label{eqn:G-iii-1-2}\\
    &+ \bnpl\notag
\end{align}\end{subequations}
For the subterm~\eqref{eqn:G-iii-1-2}, the $L^4$ estimates from~\eqref{eqn:u-lin-ests} together with~\eqref{eqn:cubic-u2-w-bound} imply
\begin{equation*}\begin{split}
    \lVert \eqref{eqn:G-iii-1-2} \rVert_{L^2} 
    \lesssim& M^4\epsilon^4 s^{-13/6} 2^{-j/2} \lVert xg \rVert_{L^2}
\end{split}\end{equation*}
which is acceptable.  Turning to the subterm~\eqref{eqn:G-iii-1-1}, we use~\eqref{eqn:cubic-u2-w-space-loc-bound} to find that
\begin{equation*}\begin{split}
    \biggl\lVert \sum_j \eqref{eqn:G-iii-1-1} \biggr\rVert_{L^2} \lesssim& \biggl( \sum_{2^j \gtrsim s^{-1/3}} \lVert u_{\sim j} \rVert_{L^4}^2 \lVert \chi_{[j - 40 ,  j + 40]}|u|^2 \partial_x w \rVert_{L^\infty}^2 \lVert \chi_{[j - 30,  j + 30]} u_{\lesssim j} \rVert_{L^4}^2 \biggr)^{1/2}\\
    \lesssim& M^4\epsilon^4 s^{-2} \lVert xg \rVert_{L^2}
\end{split}\end{equation*}
which completes the bound for~\eqref{eqn:G-iii-1}.  The bound for~\eqref{eqn:G-iii-2} is identical.  
For~\eqref{eqn:G-iii-3}, we write
\begin{subequations}\begin{align}
    \eqref{eqn:G-iii-3} =& Q_{\sim j}\sum_{k < j - 30}T_{m_j}(\chi_{\sim k}u_{\ll j}, \chi_k(|u|^2\partial_x w)_{\sim j}, \chi_{\sim k} \overline{u}_{\ll j})\label{eqn:G-iii-3-1}\\
    &+ Q_{\sim j}T_{m_j}(\chi_{[j - 40 ,  j + 40]}u_{\ll j}, \chi_{[j - 30 ,  j + 30]}(|u|^2\partial_x w)_{\sim j}, \chi_{[j - 40 ,  j + 40]} \overline{u}_{\ll j})\label{eqn:G-iii-3-2}\\
    &+ Q_{\sim j}T_{m_j}(\chi_{> j + 20}u_{\ll j}, \chi_{> j + 30}(|u|^2\partial_x w)_{\sim j}, \chi_{> j + 20}\overline{u}_{\ll j})\label{eqn:G-iii-3-3}\\
    &+ \bnpl\notag
\end{align}\end{subequations}
For the first subterm, we use the refined cubic bound~\eqref{eqn:cubic-u2-w-refined-bound} to conclude that
\begin{equation*}\begin{split}
    \left\lVert \sum_j \eqref{eqn:G-iii-3-1} \right\rVert_{L^2} \lesssim& \left(\sum_{j} \left(\sum_{k < j - 30} \lVert \chi_{\sim k}u_{\ll j}\rVert_{L^4}^2 \lVert \chi_k(|u|^2\partial_x w)_{\sim j}\rVert_{L^\infty}\right)^2 \right)^{1/2}\\
    \lesssim& M^4\epsilon^4 s^{-2} \lVert xg \rVert_{L^2}
\end{split}\end{equation*}
For the second term, we use the bound~\eqref{eqn:cubic-u2-w-space-loc-bound} to control $\chi_{[j - 30 ,  j + 30]} |u|^2 \partial_x w$, yielding
\begin{equation*}\begin{split}
    \left\lVert \sum_{j} \eqref{eqn:G-iii-3-2} \right\rVert_{L^2} \lesssim& \left(\sum_{2^j \gtrsim s^{-1/3}} \lVert \chi_{[j - 40,  j + 40]} u_{\ll j} \rVert_{L^4}^4 \lVert \chi_{[j - 30,  j + 30]}(|u|^2 \partial_x w)_{\sim j} \rVert_{L^\infty}^2 \right)^{1/2}\\
    \lesssim& M^4\epsilon^4 s^{-2} \lVert xg \rVert_{L^2}
\end{split}\end{equation*}
Finally, for the last term, we use the refined $L^4$ bounds for $\chi_{> j + 20} u_{\ll j}$ given in~\eqref{eqn:u-lin-ests} to conclude that
\begin{equation*}\begin{split}
    \lVert \eqref{eqn:G-iii-3-1} 
    \lesssim& M^4\epsilon^4 s^{-8/3} 2^{-2j} \lVert xg \rVert_{L^2}
\end{split}\end{equation*}
which completes the bound for~\eqref{eqn:G-iii}.  The argument for~\eqref{eqn:G-iv} is similar once we establish bounds for $u \partial_x S \overline{w}$ similar to the above bounds for $|u|^2 \partial_x w$.  The decay estimates given in~\Cref{thm:lin-decay-lemma} and \Cref{thm:w-lin-decay} immediately imply that
\begin{equation*}\begin{split}
    \lVert \chi_k u \partial_x S \overline{w} \rVert_{L^\infty} \lesssim& M\epsilon^2 s^{-3/2} \lVert xg \rVert_{L^2} c_k\\
    \lVert u \partial_x S \overline{w} \rVert_{L^\infty} \lesssim& M\epsilon^2 s^{-3/2} \lVert xg \rVert_{L^2}
\end{split}\end{equation*}
which are analogous to the estimates~\eqref{eqn:cubic-u2-w-space-loc-bound} and~\eqref{eqn:cubic-u2-w-bound}.  Moreover, we see that for $k < j - 30$, we can write
\begin{equation*}\begin{split}
    \chi_k (u \partial_x S \overline{w})_{\sim j} =& \chi_k Q_{\sim j} \chi_{\sim k} \bigl(u_{< j - 20} \partial_x S_{< j - 20} \overline{w}_{[j - 20 ,  j + 20]}\}\bigr) + \\
    &+\!\chi_k Q_{\sim j} \chi_{\sim k}\bigl( (u_{< j - 20} \partial_x S_{[j - 20,  j + 20]} + u_{[j - 20 ,  j + 20]} \partial_x S_{< j - 20}) \overline{w}_{< j - 20}\bigr)\\
    &+ \bbetter
\end{split}\end{equation*}
Applying the decay estimates from~\eqref{eqn:u-lin-ests} and~\eqref{eqn:w-lin-ests-L-inf}, we obtain the bounds
\begin{equation*}\begin{split}
    \lVert \chi_k u_{< j - 20} \partial_x S_{< j - 20} w_{[j - 20,  j + 20]} \rVert_{L^\infty} \lesssim& M\epsilon^2 s^{-11/6} 2^{-j} \lVert xg \rVert_{L^2} c_j\\
    \lVert \chi_k u_{< j - 20} \partial_x S_{[j - 20,  j + 20]} w_{< j - 20} \rVert_{L^\infty} \lesssim& M\epsilon^2 s^{-\frac{11}{6}} 2^{-\frac{j+k}{2}} \lVert xg \rVert_{L^2} c_k
\end{split}\end{equation*}
with an identical estimate holding for $u_{[j - 20 ,  j + 20]} \partial_x S_{< j - 20} w_{< j - 20}$.  Thus, after noting that the restriction $k < j - 30$ allows us to write ${2^{-\frac{j+k}{2}} = 2^{-j} c_j}$ uniformly in $k$, we find that
\begin{equation*}\begin{split}
    \lVert \chi_k (u \partial_x S \overline{w})_{\sim j} \rVert_{L^\infty} \lesssim& M\epsilon^2 s^{-\frac{11}{6}} 2^{-j}  \lVert xg \rVert_{L^2} c_j
\end{split}\end{equation*}
which is better than~\eqref{eqn:cubic-u2-w-refined-bound}, since $k < j - 30$.  Using these bounds and arguing as in~\eqref{eqn:G-iii}, we see that~\eqref{eqn:G-iv} satisfies the necessary estimates.  Combining all these bounds, we have shown that
\begin{equation*}
	\lVert \mathring{T}(s) \rVert_{L^2} \lesssim M^4\epsilon^4 s^{-2} \lVert xg \rVert_{L^2} + M^5\epsilon^5 t^{-11/6 - \beta}
\end{equation*}

\subsubsection{The bound for~\texorpdfstring{\eqref{eqn:time-non-inner-product-deriv}}{(\ref{eqn:time-non-inner-product-deriv})}} The only remaining term is~\eqref{eqn:time-non-inner-product-deriv}.  Expanding $x \partial_s g$ using~\eqref{eqn:xg-division} gives
\begin{equation*}\begin{split}
    \eqref{eqn:time-non-inner-product-deriv} =& -\Im \int_1^t s^2 \Big\langle T_{\partial_\xi \phi}(u, \partial_x w, \overline{u}) + T_{\partial_\xi \phi}(u, \partial_x S, \overline{w}) + T_{\partial_\xi \phi}(w, \partial_x S, \overline{u}),T(s) \Big\rangle\,ds\\
    &+\Re \int_1^t s \Big\langle |u|^2 \partial_x Lw, T(s) \Big\rangle\,ds\\
    &+\Re \int_1^t s \Big\langle |u|^2 w + 2\Re(u\overline{w})S + \Re(u\overline{w}) \partial_x LS, T(s) \Big\rangle\,ds\\
    &+ \Re \int_1^t s \Big\langle LD_p S \partial_s \hat{u}(0,s), T(s) \Big\rangle\,ds
\end{split}\end{equation*}
Expanding the pseudoproducts via~\eqref{eqn:basic-STR-division} and noticing that $\cramped{-i(\xi - \eta - \sigma) \partial_\xi \phi \chi^\cT_s = \phi m_s}$, we write
\begin{subequations}\begin{align}
    \begin{split}\eqref{eqn:time-non-inner-product-deriv} =& \Im \int_1^t  s^2  \Big\langle T_{\phi m_s}(u,w, \overline{u}) + T_{\phi m_s}(u, S, \overline{w}) +  T_{\phi m_s}(w, S, \overline{u}), T(s) \Big\rangle\,ds\end{split}\label{eqn:TNR-IPD-time-pseudoprod-cancel}\\%
    &-\Re \int_1^t s \Big\langle |u|^2 \partial_x Lw, T(s) \Big\rangle\,ds\label{eqn:TNR-IPD-bad-cubic}\\
    &-\Re \int_1^t s \Big\langle H(s), T(s) \Big\rangle\,ds\label{eqn:TNR-IPD-H}\\
    &- \Re \int_1^t s\Big\langle L D_p S \partial_s \hat{u}(0,s), T(s) \Big\rangle\,ds\label{eqn:TNR-IPD-LD-p-S}
\end{align}\end{subequations}
where
\begin{equation*}\begin{split}
    H(s) =& e^{-s\partial_x^3} \left(T_{\partial_\xi \phi e^{is\phi}\chi^\cR_s}(f, \partial_x g, \overline{f}) + T_{\partial_\xi \phi e^{is\phi}\chi^\cR_s}(f, \partial_x h, \overline{g}) + T_{\partial_\xi \phi e^{is\phi} \chi^\cR_s}(g, \partial_x h, \overline{f}) \right)\\
    &+ e^{-s\partial_x^3} \left(T_{\partial_\xi \phi e^{is\phi}\chi^\cS_s}(f, \partial_x g, \overline{f}) + T_{\partial_\xi \phi e^{is\phi}\chi^\cS_s}(f, \partial_x h, \overline{g}) + T_{\partial_\xi \phi e^{is\phi} \chi^\cS_s}(g, \partial_x h, \overline{f}) \right)\\
    &+ |u|^2 w + 2\Re(u\overline{w})S + 2\Re(u\overline{w})LS
\end{split}\end{equation*}
Using the arguments from the previous subsections, we conclude that
\begin{equation*}
    \lVert H(s) \rVert_{L^2} \lesssim M^2\epsilon^2 s^{-1} \lVert xg(s) \rVert_{L^2}
\end{equation*}
so~\eqref{eqn:T-t-bound} and the Cauchy-Schwarz inequality give the bound
\begin{equation*}
    \left| \eqref{eqn:TNR-IPD-H} \right| \lesssim \int_1^t M^4\epsilon^4s^{-1} \lVert xg(s) \rVert_{L^2}^2\,ds
\end{equation*}
Similar reasoning using~\eqref{eqn:L-D-p-S-hat-u-bd} shows that
\begin{equation*}
    \left| \eqref{eqn:TNR-IPD-LD-p-S} \right| \lesssim \int_1^t  M^7\epsilon^7s^{-5/6-\beta} \lVert xg(s) \rVert_{L^2}\,ds
\end{equation*}
Turning to the term~\eqref{eqn:TNR-IPD-bad-cubic}, we note that
\begin{equation*}
    \left\langle |u|^2 \partial_x Lw, T(s)\right\rangle = - \left\langle Lw, \partial_x \left(|u|^2 T(s)\right)\right\rangle
\end{equation*}
so the desired bound will follow immediately if we can show that
\begin{equation*}
    \lVert \partial_x \bigl(|u|^2 T(s)\bigr) \rVert_{L^2} \lesssim M^4\epsilon^4 s^{-2} \lVert xg \rVert_{L^2}
\end{equation*}
Using~\Cref{thm:simple-bilinear-decay} and the estimates in~\Cref{sec:T-m-s-bounds}, we see that
\begin{equation*}
    \lVert \left(\partial_x |u|^2\right) T(s) \rVert_{L^2} \lesssim M^4\epsilon^4 s^{-2} \lVert xg \rVert_{L^2}
\end{equation*}
so it suffices to prove that
\begin{equation*}
    \left\lVert |u|^2 \partial_x T(s) \right\rVert_{L^2} \lesssim M^4\epsilon^4 s^{-2} \lVert xg \rVert_{L^2}
\end{equation*}
We will show how to obtain the bound for $|u|^2 \partial_x T_{m_s}(u, w,\overline{u})$: the bounds for the other terms in $T(s)$ are similar.  Recalling that $m_s = \frac{i(\xi-\eta-\sigma) \partial_\xi \phi \chi_s^\cT}{\phi}$, we can write
\begin{subequations}\begin{align}
    |u|^2\partial_x T_{m_s}(u,w,\overline{u}) =& \sum_{2^j \gtrsim s^{-1/3}} |u|^2 T_{m_j}(u, \partial_x w, \overline{u})\notag\\
    =&\sum_{2^j \gtrsim s^{-1/3}} |u|^2 T_{m_j}(u_{\sim j}, \partial_x w_{\lesssim j}, \overline{u})\label{eqn:TNR-IPD-BC-1}\\
    &\qquad+ |u|^2 T_{m_j}(u_{\ll j}, \partial_x w_{\lesssim j}, \overline{u}_{\sim j})\label{eqn:TNR-IPD-BC-2}\\
    &\qquad+ |u|^2 Q_{\sim j} T_{m_j}(u_{\ll j}, \partial_x w_{\sim j}, \overline{u}_{\ll j})\label{eqn:TNR-IPD-BC-3}
\end{align}\end{subequations}
For~\eqref{eqn:TNR-IPD-BC-1}, we perform a further division in physical space
\begin{subequations}\begin{align}
    \eqref{eqn:TNR-IPD-BC-1} =& \sum_{j}\sum_{k < j - 30} \chi_k |u|^2 T_{m_j}(\chi_{\sim k} u_{\sim j}, \chi_{\sim k}\partial_x w_{\lesssim j}, \chi_{\sim k}\overline{u})\label{eqn:TNR-IPD-BC-1-1}\\
    &+\sum_{j}\chi_{[j - 30 ,  j + 30]} |u|^2 T_{m_j}( u_{\sim j}, \chi_{[j - 40 ,  j + 40]}\partial_x w_{\lesssim j}, \chi_{[j - 40 ,  j + 40]}\overline{u})\label{eqn:TNR-IPD-BC-1-2}\\
    &+ \sum_{j} \chi_{> j + 30} |u|^2 T_{m_j}(\chi_{> j + 20} u_{\sim j}, \chi_{> j + 20} \partial_x w_{\lesssim j}, \chi_{> j + 20} \overline{u})\label{eqn:TNR-IPD-BC-1-3}\\
    &+ \bnpl\notag
\end{align}\end{subequations}
We will focus on\eqref{eqn:TNR-IPD-BC-1-2}, since the bounds for \eqref{eqn:TNR-IPD-BC-1-1} and \eqref{eqn:TNR-IPD-BC-1-3} are easier.  Here, we can use almost orthogonality to obtain
\begin{equation*}\begin{split}
    \biggl\lVert \sum_j \eqref{eqn:TNR-IPD-BC-1-2} \biggr\rVert_{L^2} \lesssim& \biggl(\sum_{j} \lVert \chi_{[j - 40 ,  j + 40]} u \rVert_{L^\infty}^6\lVert u_{\sim j}\rVert_{L^\infty}^2 \lVert \chi_{[j - 40 ,  j + 40]}\partial_x w_{\lesssim j} \rVert_{L^2}^2 \biggr)^{1/2}\\
    \lesssim& M^4\epsilon^4 s^{-2} \lVert xg \rVert_{L^2}
\end{split}\end{equation*}
Since $u_{\ll j}$ satisfies better decay estimates that $u$, we can bound~\eqref{eqn:TNR-IPD-BC-2} in the same way. For~\eqref{eqn:TNR-IPD-BC-3}, we write
\begin{subequations}\begin{align}
    \eqref{eqn:TNR-IPD-BC-3} =& \sum_{j}\sum_{k < j - 30} \chi_k |u|^2 Q_{\sim j} T_{m_j}(\chi_{\sim k} u_{\ll j}, \chi_{\sim k} \partial_x w_{\sim j}, \chi_{\sim k} \overline{u}_{\ll j})\label{eqn:TNR-IPD-BC-3-1}\\
    &+ \sum_{j} \chi_{[j - 30 ,  j + 30]} |u|^2  T_{m_j}(\chi_{[j - 40 ,  j + 40]} u_{\ll j}, \partial_x w_{\sim j}, \chi_{[j - 40 ,  j + 40]} \overline{u}_{\ll j})\label{eqn:TNR-IPD-BC-3-2}\\
    &+ \sum_{j}\chi_{> j + 30} |u|^2 T_{m_j}(\chi_{> j + 20} u_{\ll j}, \chi_{> j - 20}\partial_x w_{\sim j}, \chi_{> j + 20} \overline{u}_{\ll j})\label{eqn:TNR-IPD-BC-3-3}\\
    &+ \bnpl\notag
\end{align}\end{subequations}
We will show how to handle the first term; the remaining terms are easier.  Note that we can interchange the order of summation to obtain
\begin{equation*}
    \biggl\lVert \eqref{eqn:TNR-IPD-BC-3-1} \biggr\rVert_{L^2} = \biggl\lVert \sum_k \chi_k |u|^2 \sum_{j > k + 30} Q_{\sim j} T_{m_j}(\chi_{\sim k} u_{\ll j}, \chi_{\sim k} \partial_x w_{\sim j}, \chi_{\sim k} \overline{u}_{\ll j}) \biggr\rVert_{L^2}
\end{equation*}
Thus, using the almost orthogonality in $j$ coming from the frequency projection, we find that
\begin{equation*}\begin{split}
    \biggl\lVert \eqref{eqn:TNR-IPD-BC-3-1} \biggr\rVert_{L^2} \lesssim& \sum_k \lVert \chi_{\sim k} u \rVert_{L^\infty}^2  \biggl(\sum_{j > k + 30} \lVert \chi_{\sim k} u_{\ll j} \rVert_{L^4}^4  \lVert \chi_{\sim k} \partial_x w_{\sim j} \rVert_{L^\infty}^2 \biggr)^{1/2}\\
    \lesssim& M^4\epsilon^4 s^{-2} \lVert xg \rVert_{L^2}
\end{split}\end{equation*}

It only remains to bound the contribution from~\eqref{eqn:TNR-IPD-time-pseudoprod-cancel}. Observe that
\begin{equation*}\begin{split}
    \frac{1}{2}\partial_s \lVert T(s) \rVert_{L^2}^2 =& \Re \left\langle \partial_s T(s), T(s)\right\rangle\\
    =& \Re \left\langle e^{-s\partial_x^3} \partial_s e^{s\partial_x^3} T(s), T(s)\right\rangle - \Re \left\langle e^{-s\partial_x^3} \partial_x^3 e^{s\partial_x^3} T(s), T(s)\right\rangle
\end{split}\end{equation*}
The second term on the last line vanishes because $\partial_x^3$ is a skew-adjoint operator.  Noting that
\begin{equation*}
    e^{s\partial_x^3}T(s) = T_{m_se^{is\phi}}(f,g,\overline{f}) + T_{m_se^{is\phi}}(f,h,\overline{g}) + T_{m_se^{is\phi}}(g,h,\overline{f})
\end{equation*}
we see that
\begin{equation}\label{eqn:time-double-pseudoprod-cancellation}\begin{split}
    \frac{1}{2}\partial_s \lVert T(s) \rVert_{L^2}^2 =& \Re\bigl\langle T_{m_s\partial_s e^{is\phi}}(f,g,\overline{f}) + T_{m_s \partial_s e^{is\phi}}(f,h,\overline{g}) \\
    &\qquad\quad+ T_{m_s \partial_s e^{is\phi}}(g,h,\overline{f}), e^{s\partial_x^3}T(s) \bigr\rangle + \Re\bigl\langle \tilde{T}(s) + \mathring{T}(s), T(s) \bigr\rangle\\
\end{split}\end{equation}
Now, 
\begin{equation*}
    -i\partial_\xi \phi (\xi - \eta - \sigma) e^{is\phi} \chi^\cT_s = \phi m_s e^{is\phi} = -im_s \partial_s e^{is\phi}
\end{equation*}
so we can use~\eqref{eqn:time-double-pseudoprod-cancellation} to write
\begin{subequations}\begin{align}
    \eqref{eqn:TNR-IPD-time-pseudoprod-cancel} =& - \frac{1}{2} \int_1^t s^2 \partial_s\lVert T(s) \rVert_{L^2}^2 \,ds\label{eqn:TNR-IPD-C-1}\\
    &+ \Re \int_1^t s^2 \langle \tilde{T}(s) + \mathring{T}(s), T(s) \rangle\,ds\label{eqn:TNR-IPD-C-2}
\end{align}\end{subequations}
The bounds for $T(s)$, $\tilde{T}(s)$, and $\mathring{T}(s)$ imply that
\begin{equation*}\begin{split}
    |\eqref{eqn:TNR-IPD-C-2}| 
    \lesssim& \int_1^t M^4\epsilon^4 s^{-1}\lVert xg \rVert_{L^2}^2 + M^4\epsilon^5 s^{-5/6-\beta}\lVert xg \rVert_{L^2}\,ds
\end{split}\end{equation*}
Turning to~\eqref{eqn:TNR-IPD-C-1}, integration by parts shows that
\begin{equation*}\begin{split}
    |\eqref{eqn:TNR-IPD-C-1}| =& \left|- \left.\frac{1}{2}s^2 \lVert T(s) \rVert_{L^2}^2 \right|_{s=1}^{s=t} + \int_1^t s \lVert T(s)\rVert_{L^2}^2 \,ds\right|\\
    \lesssim& M^4\epsilon^4 \lVert xg \rVert_{L^2}^2 + \epsilon^6 + \int_1^t M^4\epsilon^4 s^{-1} \lVert xg (s) \rVert_{L^2}^2\,ds
\end{split}\end{equation*}
completing the bound for~\eqref{eqn:time-non-inner-product-deriv}.

\section{The pointwise bounds in Fourier space}\label{sec:L-infty-est}

In this section, we will show how to control $\hat{f}(\xi, t)$ and $\partial_t \hat{u}(0,t)$.  For $|\xi| < t^{-1/3}$, the H\"older continuity of $\hat{f}$ in $\xi$ reduces to showing that $\hat{f}(0,t) = O(\epsilon)$, which will follow once we show that~\eqref{eqn:desired-zero-mode-conv} holds for $\partial_t \hat{u}(0,t)$.  We show this improved decay by taking advantage of the self-similar structure at low frequencies.  On the other hand, when $|\xi| \geq t^{-1/3}$, we show that $\hat{f}(\xi,t)$ essentially has ODE dynamics, which produce the logarithmic phase correction given by~\eqref{eqn:phase-rot-dynamics-f-hat}. 

\subsection{The low-frequency bounds}\label{sec:low-freq-bdds}

We first prove the bound~\eqref{eqn:desired-zero-mode-conv} on $|\partial_t\hat{u}(0,t)|$.  Recall that
\begin{equation*}
    \partial_t \hat{u}(0,t) = -\frac{1}{\sqrt{2\pi}}\int |u|^2 \partial_x u \,dx
\end{equation*}
Using the self-similar equation~\eqref{eqn:cmkdv-self-sim-prof-expanded}, we see that $|S|^2 \partial_x S$ is the derivative of an $L^2$ function.  Thus, $\int |S|^2 \partial_x S\,dx = 0$, and we can write
\begin{equation*}\begin{split}
    \partial_t \hat{u}(0,t) =& -\frac{1}{\sqrt{2\pi}}\int |u|^2 \partial_x u - |S|^2 \partial_x S\,dx\\
    =& -\frac{1}{\sqrt{2\pi}} \int |u|^2 \partial_x w + (u \overline{w} + w\overline{u}) \partial_x S\,dx
\end{split}\end{equation*}
But, all of these terms can be controlled using~\Cref{thm:cubic-mean-thm}, giving us the bound
\begin{equation*}\begin{split}
    |\partial_t \hat{u}(0,t)| \lesssim& M^2 \epsilon^2 t^{-7/6} \lVert xg \rVert_{L^2}\\
    \lesssim& M^2 \epsilon^3 t^{-1-\beta}
\end{split}\end{equation*}
which is exactly what is required for~\eqref{eqn:desired-zero-mode-conv}.

We now turn to the task of bounding $\hat{f}(\xi,t)$ in the low frequency region $|\xi| < t^{-1/3}$.  Note that $\partial_t \hat{u}(0,t) = \partial_t \hat{f}(0,t)$, so integrating the above bound and recalling that $M^2\epsilon^2 \ll 1$ gives us the bound
\begin{equation*}
    |\hat{f}(0,t)| \lesssim |\hat{f}(0,1)| + \int_1^t |\partial_t \hat{u}(0,t)|\,dt
    \lesssim \epsilon
\end{equation*}
as required.  For the other frequencies, we note that by the Sobolev-Morrey embedding, for $|\xi| < t^{-1/3}$,
\begin{equation*}\begin{split}
    |\hat{f}(\xi,t) - \hat{f}(0,t)| \lesssim& |\xi|^{1/2} \lVert xf \rVert_{L^2}\\
    \lesssim& t^{-1/6} \left(\lVert xg \rVert_{L^2} + \lVert LS \rVert_{L^2}\right)\\
    \lesssim& \epsilon
\end{split}\end{equation*}
allowing us to conclude that $\hat{f}$ is bounded in the region $|\xi| < t^{-1/3}$.

\subsection{The perturbed Hamiltonian dynamics}\label{sec:st-phase-setup}
In this section, we show that $\hat f(\xi,t)$ satisfies a perturbed Hamiltonian ODE for each $|\xi| \geq t^{-1/3}$, and as a consequence $\lVert \hat f \rVert_{L^\infty}$ is uniformly bounded in time.  In particular, we will show that for $|\xi| \geq t^{-1/3}$, $\hat{f}(\xi,t)$ satisfies
\begin{equation}\label{eqn:hat-f-hamiltonian-dynamics}
    \partial_t \hat{f}(\xi, t) = -\frac{i\sgn \xi}{6t} | \hat{f}(\xi, t)|^2 \hat{f}(\xi, t) + ce^{it\frac{8}{9}\xi^3}\frac{\sgn \xi}{t} | \hat{f}(\xi/3, t)|^2 \hat{f}(\xi/3, t) + R(\xi,t)
\end{equation}
for some constant $c$, where $|R(\xi, t)| \lesssim M^3\epsilon^3 t^{-1} (|\xi| t^{1/3})^{-1/14}$.  Assuming for the moment that~\eqref{eqn:hat-f-hamiltonian-dynamics} holds, we define
\begin{equation*}
    B(t,\xi) := -\frac{\sgn(\xi)}{6} \int_{1}^t \frac{|\hat f(\xi,s)|^2}{s}\,ds
\end{equation*}
so $v(t,\xi) = e^{iB(t,\xi)} \hat{f}(t,\xi)$ satisfies
\begin{equation*}
    \partial_t v = ce^{it\frac{8}{9}\xi^3}e^{iB(t,\xi)}\frac{\sgn \xi}{t} | \hat{f}(\xi/3, t)|^2 \hat{f}(\xi/3, t) + e^{iB(t,\xi)}R(\xi,t)
\end{equation*}
Let us consider $|v(t_1) - v(t_2)|$ for $\max(1, |\xi|^{-3}) \leq t_1 < t_2 \leq T$, where $T$ is the time given in the bootstrap argument.  Integrating by parts using the identity $e^{it8/9|\xi|^3} = \frac{9}{8i|\xi|^3}\partial_te^{it8/9|\xi|^3}$, we find that:
\begin{equation*}\begin{split}
    \left|\int_{t_1}^{t_2} \partial_t v(s) ds\right| \lesssim& \left.\frac{1}{s|\xi|^3} | \hat{f}(s)|^3 \right|_{s=t_1}^{s=t_2} + \int_{t_1}^{t_2}  |\hat{f}(s)|^3\,\frac{ds}{s^2|\xi|^3}+ \int_{t_1}^{t_2} |\partial_s B(s)| |\hat{f}(s)|^3\,\frac{ds}{s|\xi|^3}\\
    &+ \int_{t_1}^{t_2}  |\partial_s \hat{f}(s)||\hat{f}(s)|^2\,\frac{ds}{s|\xi|^3} + \int_{t_1}^{t_2} |R(\xi,s)|\;ds\\
    =:& \rmI + \rmII + \rmIII + \rmIV  + O\left(M^3\epsilon^3 (|\xi|t_1^{1/3})^{-\frac{1}{14}}\right)
\end{split}\end{equation*}
Using the definition of $B$ and the bound on $\hat{f}$, we see that
\begin{equation}\label{eqn:v-bound-1}
\rmI + \rmII + \rmIII \lesssim (M^3\epsilon^3 + M^5\epsilon^5) t_1^{-1} |\xi|^{-3}
\end{equation}
Moreover, substituting the expression given in~\eqref{eqn:hat-f-hamiltonian-dynamics} for $\partial_s f(s)$, we find that
\begin{equation}\label{eqn:v-bound-2}
    |\rmIV| \lesssim \int_{t_1}^{t_2} \frac{M^2\epsilon^2}{s|\xi|^3}\left(\frac{M^3\epsilon^3}{s} + R(\xi,s)\right) \,ds \lesssim M^5\epsilon^5 t_1^{-1} |\xi|^{-3}
\end{equation}
Taking $t_1 = \max(1, |\xi|^{-3})$ and using the fact that $|v(\xi, t_1)| = |\hat{f}(\xi, t_1)| \lesssim \epsilon$, we see that for $t \in (t_1, T)$,
\begin{equation*}
    |\hat{f}(\xi, t)| = |v(\xi, t)| \lesssim \epsilon
\end{equation*}
which closes the bootstrap for the $\mathcal{F}L^\infty$ component of the $X$ norm.  Moreover, we have shown that $v(\xi,t)$ is Cauchy as $t\to \infty$ for $\xi \neq 0$, so $v(\xi,t)$ converges as $t \to \infty$ for each fixed nonzero $\xi$.   If we write $f_\infty(\xi) = \lim_{t\to\infty} v(\xi,t)$, then writing $\hat{f}$ in terms of $v$ shows that
\begin{equation*}
    \hat{f}(\xi,t) = \exp\left(-\frac{i}{6}\sgn \xi \int_1^t \frac{|\hat{f}(\xi,s)|^2}{s}\,ds\right) f_\infty(\xi) + O(M^3\epsilon^3(t^{-1/3}|\xi|)^{-1/14})
\end{equation*}
so~\eqref{eqn:phase-rot-dynamics-f-hat} holds.  Thus, all that remains to prove~\Cref{thm:main-theorem} is to verify~\eqref{eqn:hat-f-hamiltonian-dynamics}.

\subsection{The stationary phase estimate}\label{sec:st-phase-sec}

We will prove~\eqref{eqn:hat-f-hamiltonian-dynamics} using the method of stationary phase.  Note that we can write $\partial_t \hat{f}(\xi,t)$ as
\begin{equation}
    \partial_t \hat{f}(\xi,t) =  -\frac{i}{2\pi} \int e^{-it\phi} (\xi -\eta - \sigma) \hat{f}(\eta) \overline{\hat{f}(-\sigma)} \hat{f}(\xi - \eta - \sigma)\,d\eta d\sigma
\end{equation}
The stationary points for the phase $\phi$ are given by \begin{equation}\label{eqn:stationary-pts}
    \begin{aligned}[t]
        (\eta_1, \sigma_1) =& (\xi, \xi),\\
        (\eta_3, \sigma_3) =& (\xi, -\xi),\\
    \end{aligned}\qquad\begin{aligned}[t]
        (\eta_2, \sigma_2) =& (-\xi, \xi),\\
        (\eta_4, \sigma_4) =& (\xi/3, \xi/3),
    \end{aligned}
\end{equation}
We will now divide the integral dyadically in $\eta$ and $\sigma$, and use stationary phase to estimate each piece.  Defining $j$ to be the integer with $2^{j-1} < |\xi| \leq 2^j$, let us write
\begin{equation*}
    \partial_t \hat{f}(\xi,t) = \rmI_\text{lo} + \rmI_{\text{stat}} + \sum_{\ell > j + 10} \left(\rmI_{\ell} + \tilde{\rmI}_{\ell}\right)
\end{equation*}
where 
\begin{equation*}\begin{split}
    \rmI_{\text{lo}} =&  -\frac{i}{2\pi} \int e^{-it\phi} \Psi_{\ll j,\ll j}(\eta,\sigma) (\xi -\eta - \sigma) \hat{f}(\eta) \overline{\hat{f}(-\sigma)} \hat{f}(\xi - \eta - \sigma)\,d\eta d\sigma\\
    \rmI_{\text{stat}} =&  -\frac{i}{2\pi} \int e^{-it\phi} \Psi_{\text{med}}(\eta,\sigma)  (\xi -\eta - \sigma) \hat{f}(\eta) \overline{\hat{f}(-\sigma)} \hat{f}(\xi - \eta - \sigma)\,d\eta d\sigma\\
    \rmI_{\ell} =&  -\frac{i}{2\pi} \int e^{-it\phi} (\xi -\eta - \sigma) \Psi_{\ell, \leq \ell}(\eta,\sigma) \hat{f}(\eta) \overline{\hat{f}(-\sigma)} \hat{f}(\xi - \eta - \sigma)\,d\eta d\sigma\\
    \tilde{\rmI}_{\ell} =&  -\frac{i}{2\pi} \int e^{-it\phi} (\xi -\eta - \sigma) \Psi_{< \ell, \ell}(\eta,\sigma) \hat{f}(\eta) \overline{\hat{f}(-\sigma)} \hat{f}(\xi - \eta - \sigma)\,d\eta d\sigma
\end{split}\end{equation*}
with
\begin{equation*}\begin{split}
    \Psi_{\bullet, \star}(\eta, \sigma) =& \psi_{\bullet}(\eta) \psi_{\star}(\sigma)\\
    \Psi_{\text{med}}(\eta,\sigma) =& \Psi_{\lesssim j,\lesssim j}(\eta,\sigma) - \Psi_{\ll j,\ll j}(\eta,\sigma)
\end{split}\end{equation*}
We will show how to estimate the terms $\rmI_{\text{lo}}$, $\rmI_{\text{stat}}$ and $\rmI_{\ell}$: the estimate for the $\tilde{\rmI}_{\ell}$ terms follows from similar reasoning.

\subsubsection{The estimate for \texorpdfstring{$\rmI_{\textup{lo}}$}{I\textunderscore lo}} Over the support of the integrand, $|\partial_\eta \phi| \sim 2^{2j},$ so we can integrate by parts with respect to $\eta$ to obtain
\begin{subequations}\begin{align}
    \rmI_{\text{lo}} =& \frac{i}{2\pi t} \int e^{it\phi} \frac{\xi -\eta - \sigma}{\partial_\eta \phi} \Psi_{\ll j,\ll j}(\eta,\sigma) \partial_\eta \hat{f}(\eta) \hat{f}(\xi - \eta - \sigma) \overline{\hat{f}(-\sigma)}\,d\eta d\sigma\label{eqn:I-lo-1}\\
    &+\frac{1}{4\pi t} \int e^{it\phi} \partial_\eta \left(\frac{(\xi - \eta - \sigma)}{\partial_\eta \phi} \Psi_{\ll j,\ll j}(\eta,\sigma)\right) \hat{f}(\eta) \hat{f}(\xi - \eta - \sigma) \overline{\hat{f}(-\sigma)}\,d\eta d\sigma\label{eqn:I-lo-2}\\
    &+ \{\text{similar terms}\}\notag
\end{align}\end{subequations}
For the first term, we note that $\mu^1_j = 2^{j}\frac{\xi - \eta - \sigma}{\partial_\eta \phi} \Psi_{\ll j,\ll j}(\eta,\sigma)\psi_{\sim j}(\xi)$ is a smooth symbol supported on $|\xi|, |\eta|, |\sigma| \lesssim 2^j$ which satisfies Coifman-Meyer type symbol bounds.  Thus, by~\Cref{rmk:freq-loc-symbol-bounds} and the Hausdorff-Young inequality
\begin{equation*}\begin{split}
    |\eqref{eqn:I-lo-1}| 
    \lesssim& t^{-1}2^{-j} \left\lVert T_{\mu^1_j}(Lu_{\lesssim j}, u_{\lesssim j}, \overline{u}_{\lesssim j}) \right\rVert_{L^1}\\
    \lesssim& M^3\epsilon^3 t^{-7/6} 2^{-j/2}
\end{split}\end{equation*}
Similarly, defining the symbol $\mu^2_j = 2^{2j}\partial_\eta\left(\frac{\xi - \eta - \sigma}{\partial_\eta \phi} \Psi_{\ll j,\ll j}(\eta,\sigma)\right) \psi_{\sim j}(\xi)$, we find that
\begin{equation*}\begin{split}
    |\eqref{eqn:I-lo-2}| 
    \lesssim& t^{-1}2^{-2j} \left\lVert T_{\mu^2_j}(u_{\lesssim j}, u_{\lesssim j}, \overline{u}_{\lesssim j}) \right\rVert_{L^1}\\
    \lesssim& M^3\epsilon^3 t^{-4/3} 2^{-j}
\end{split}\end{equation*}
It follows that
\begin{equation*}
    |\rmI_\text{lo}| \lesssim M^3\epsilon^3 t^{-1} \left(t^{1/3} 2^j\right)^{-1/2}
\end{equation*}
which is consistent with the estimate for the remainder term in~\eqref{eqn:hat-f-hamiltonian-dynamics}.

\subsubsection{The estimate for \texorpdfstring{$\rmI_{\ell}$}{I\_l}}

For these terms, $|\nablaes \phi| \sim 2^{2\ell}$.  Integrating by parts using the identity $\frac{1}{it|\nablaes \phi|^2} \nablaes\phi \cdot \nablaes e^{it\phi} = e^{it\phi}$, we find that
\begin{subequations}\begin{align}
    \rmI_{\ell} =& +\frac{1}{2\pi t} \int e^{it\phi} \frac{(\xi - \eta - \sigma)\partial_\eta \phi}{|\nablaes \phi|^2} \Psi_{\ell, \leq \ell}(\eta,\sigma) \partial_\eta \hat{f}(\eta) \hat{f}(\xi - \eta - \sigma) \overline{\hat{f}(-\sigma)}\,d\eta d\sigma\label{eqn:I-k-hi-1}\\
    &+\frac{1}{4\pi t} \int e^{it\phi} \nablaes \cdot \biggl(\frac{(\xi - \eta - \sigma) \Psi_{\ell, \leq \ell}(\eta,\sigma)\nablaes \phi}{|\nablaes \phi|^2}\biggr)\hat{f}(\eta) \hat{f}(\xi - \eta - \sigma) \overline{\hat{f}(-\sigma)} d\eta d\sigma\label{eqn:I-k-hi-2}\\
    &+ \{\text{similar terms}\}\notag
\end{align}\end{subequations}
The argument is now similar to the one for~\cref{eqn:I-lo-1,eqn:I-lo-2}.  Writing 
\begin{equation*}\begin{split}
    \mu^1_{\ell} =& 2^{\ell} \frac{(\xi - \eta - \sigma)\partial_\eta \phi}{|\nablaes \phi|^2} \Psi_{\ell, \leq \ell}(\eta,\sigma) \psi_{\sim \ell}(\xi - \xi_0)\\
    \mu^2_{\ell} =& 2^{2\ell} \nablaes \cdot \left(\frac{(\xi - \eta - \sigma)\partial_\eta \phi}{|\nablaes \phi|^2} \Psi_{\ell, \leq \ell}(\eta,\sigma)\right) \psi_{\sim \ell}(\xi - \xi_0)
\end{split}\end{equation*}
for some $\xi_0$ within distance $O(2^\ell)$ of $\xi$ and observing that $m^1_{\ell}$, $m^2_{\ell}$ satisfy the conditions given in~\Cref{rmk:freq-loc-symbol-bounds} uniformly in $\ell$, we find that
\begin{equation*}\begin{split}
    |\eqref{eqn:I-k-hi-1}| \lesssim& t^{-1}2^{-\ell} \lVert T_{\mu^1_{\ell}}(Lu, P^{\xi_0}_{\lesssim \ell} u, \overline{u}_{\lesssim \ell}) \rVert_{L^2}\\
    \lesssim& M^3\epsilon^3 t^{-7/6}2^{-\ell/2}\\
    |\eqref{eqn:I-k-hi-2}| \lesssim& t^{-1}2^{-2\ell} \lVert T_{\mu^2_{\ell}}(u_{\lesssim \ell}, P^{\xi_0}_{\lesssim \ell} u, \overline{u}_{\lesssim \ell}) \rVert_{L^2}\\
    \lesssim& M^3\epsilon^3 t^{-4/3}2^{-\ell}
\end{split}\end{equation*}
where $\tilde{P}^{\xi_0}_{\lesssim \ell} = \psi_{\lesssim \ell}(D - \xi_0)$.  An analogous argument holds for $\tilde{\rmI}_{\ell}$, giving us the bound
\begin{equation*}
    \left| \sum_{\ell > j + 10} \rmI_{\ell} + \tilde{\rmI}_{\ell} \right| \lesssim M^3\epsilon^3 t^{-1} \left(t^{1/3} 2^k \right)^{-1/2}
\end{equation*}
which allows us to treat these terms as remainders in~\eqref{eqn:hat-f-hamiltonian-dynamics}.

\subsubsection{The estimate for \texorpdfstring{$\rmI_{\textup{stat}}$}{I\_stat}} The integral here contains the four stationary points given in~\eqref{eqn:stationary-pts}.  Note that each of the stationary points is at a distance $\sim 2^j$ from each other.  Using this, we can write
\begin{equation*}
    \rmI_{\text{stat}} = \sum_{r=1}^4 \sum_{2^{\ell_0} \leq 2^{\ell} \ll 2^j} \left(J^{(r)}_{\ell} + \tilde{J}^{(r)}_{\ell}\right) + \{\text{remainder}\}
\end{equation*}
where 
\begin{equation*}\begin{split}
    J^{(r)}_{\ell} =&  -\frac{i}{2\pi} \int e^{-it\phi}\Psi_{\ell, \leq \ell}(\eta - \eta_r,\sigma - \sigma_r) (\xi - \eta - \sigma) \hat{f}(\eta) \hat{f}(\xi - \eta - \sigma) \overline{\hat{f}(-\sigma)}\,d\eta d\sigma\\
    \tilde{J}^{(r)}_{\ell} =&  -\frac{i}{2\pi} \int e^{-it\phi}\Psi_{< \ell, \ell}(\eta - \eta_r, \sigma - \sigma_r) (\xi - \eta - \sigma) \hat{f}(\eta) \hat{f}(\xi - \eta - \sigma) \overline{\hat{f}(-\sigma)}\,d\eta d\sigma
\end{split}\end{equation*}
for $\ell > \ell_0$, and
\begin{equation*}\begin{split}
    J^{(r)}_{\ell_0} =&  -\frac{i}{2\pi} \int e^{-it\phi}\Psi_{\leq \ell_0, \leq \ell_0}(\eta - \eta_r,\sigma - \sigma_r) (\xi - \eta - \sigma) \hat{f}(\eta) \hat{f}(\xi - \eta - \sigma) \overline{\hat{f}(-\sigma)}\,d\eta d\sigma\\
    \tilde{J}^{(r)}_{\ell_0} =&  0
\end{split}\end{equation*}
with $\ell_0$ defined such that $2^{\ell_0} \sim t^{-1/3}(t^{1/3} 2^{k})^{-\gamma}$, where $\gamma > 0$ is a constant which will be specified later.  The contribution from the remainder can be controlled using an argument similar to the one for $\rmI_{\ell}$, so we will focus on controlling the contribution from the $J^{(r)}_{\ell}$ terms.    There are two cases to consider: $\ell = \ell_0$ and $\ell > \ell_0$.

\paragraph{\indent \textbf{Case} $\ell > \ell_0$}
We first consider the bound for $J^{(r)}_{\ell}$.  Integrating by parts gives
\begin{subequations}\begin{align}
    J^{(r)}_{\ell} &= t^{-1} 2^{-\ell} e^{it\xi^3}\hat T_{\mu^1_\ell}(Lu, \tilde{P}^{\xi_0 - \eta_r - \sigma_r}_{\lesssim \ell} u, \overline{u}_{\sim j})\label{eqn:I-stat-away-1}\\
    &+ t^{-1} 2^{-2\ell} e^{it\xi^3}T_{\mu^2_{\ell}}(\tilde{P}^{\eta_r}_{\lesssim \ell} u,\tilde{P}^{\xi_0 - \eta_r + \sigma_r}_{\lesssim \ell} u, \overline{u}_{\sim j})\label{eqn:I-stat-away-2}\\
    &+ \{\text{similar terms}\}\notag
\end{align}\end{subequations}
where $\xi_0$ is any point at a distance $\ll 2^{\ell}$ from $\xi$, and the symbols $m_{\ell}^1$ and $m_{\ell}^2$ are given by
\begin{equation*}\begin{split}
    \mu_{\ell}^1 =&  -2^\ell (\xi - \eta - \sigma) \frac{\partial_\eta \phi}{|\nablaes \phi|^2} \Psi_{\ell, \leq \ell}(\eta - \eta_r,\sigma - \sigma_r) \psi_{\leq \ell} (\xi - \xi_0)\\
    \mu_{\ell}^2 =&  -2^{2\ell} \nablaes \cdot \left(\frac{(\xi - \eta - \sigma) \Psi_{\ell, \leq \ell}(\eta - \eta_r,\sigma - \sigma_r) \nablaes \phi}{|\nablaes \phi|^2}  \right) \psi_{\leq \ell} (\xi - \xi_0)
\end{split}\end{equation*}
It is clear that these symbols are supported on a region of volume $\sim 2^{3\ell}$.  Moreover, on the support of the integral we have that
\begin{equation*}\begin{split}
    |\xi - \eta - \sigma| \lesssim 2^j,\qquad 
    |\nablaes \phi| \sim 2^{j + \ell},\qquad |\partial^\alpha_{\xi,\eta,\sigma} \nablaes \phi| \lesssim 2^{(2 - |\alpha|) \ell} \qquad |\alpha| \geq 2
\end{split}\end{equation*}
where for the last inequality we have used the fact that $2^\ell \ll 2^j$.  Thus, we see that $\mu^1_\ell$ and $\mu^2_\ell$ obey the Coifman-Meyer bounds
\begin{equation*}
    |\partial_{\xi,\eta,\sigma}^\alpha \mu^1_{\ell}| +  |\partial_{\xi,\eta,\sigma}^\alpha \mu^2_{\ell}| \lesssim_\alpha 2^{-|\alpha|\ell}
\end{equation*}
and hence
\begin{equation*}\begin{split}
    \left|\eqref{eqn:I-stat-away-1}\right| \leq& t^{-1}2^{-\ell} \left\lVert{T}_{\mu_{\ell}^1}(Lu, \tilde{P}^{(\xi - \eta_r - \sigma_r)}_{\lesssim \ell}u,\overline{u}_{\sim j}) \right\rVert_{L^1}\\
    \lesssim& M^3\epsilon^3 t^{-4/3} 2^{-j/2} 2^{-\ell/2}\\
    \left|\eqref{eqn:I-stat-away-2} \right| \leq& t^{-1} 2^{-2\ell} \left\lVert {T}_{\mu_{\ell}^2}(\tilde{P}^{\eta_r}_{\lesssim \ell} u, \tilde{P}^{\xi - \eta_r - \sigma_r}_{\lesssim \ell} u, \overline{u}_{\sim j}) \right\rVert_{L^1}\\
    \lesssim& M^3\epsilon^3 t^{-3/2} 2^{-j/2} 2^{-\ell}
\end{split}\end{equation*}
Summing over $\ell > \ell_0$ yields
\begin{equation}\label{eqn:J-hi-contrib}
    \left|\sum_{\ell > \ell_0} J^{(r)}_{\ell}\right| \lesssim t^{-1} (t^{1/3} 2^j)^{-1/2 + \gamma}
\end{equation}
A similar argument gives an identical bound for the $\tilde{J}^{(r)}_{\ell}$.

\paragraph{\indent \textbf{Case} $\ell = \ell_0$}
By performing the linear change of variables $\eta \to \eta + \eta_r$, $\sigma \to \sigma + \sigma_r$, we obtain
\begin{equation*}\begin{split}
    J^{(r)}_{\ell_0} =& -\frac{i}{2\pi} \int e^{-it\phi} \psi_{\leq \ell_0}(\eta) \psi_{\leq \ell_0}(\sigma) (\xi - \eta_r - \sigma_r - \eta - \sigma) F_r(\xi,\eta,\sigma)\,d\eta d\sigma
\end{split}\end{equation*}
where
\begin{equation*}
    F_r(\xi, \eta, \sigma) = -\frac{i}{2\pi} \hat{f}(\eta+ \eta_r) \hat{f}(\xi - \eta - \eta_r - \sigma -\sigma_r) \overline{\hat{f}(-\sigma - \sigma_r)}
\end{equation*}
We can re-write this as
\begin{subequations}\begin{align}
    J^{(r)}_{\ell_0} =& \int e^{-it\phi} \Psi_{\leq \ell_0, \leq \ell_0}(\eta,\sigma)(\xi - \eta_r - \sigma_r - \eta - \sigma) \left(F_r(\xi,\eta,\sigma) - F_r(\xi,0,0)\right)\,d\eta d\sigma\label{eqn:I-stat-1}\\
    & - F_r(\xi,0,0) \int e^{-it\phi} \Psi_{\leq \ell_0, \leq \ell_0}(\eta,\sigma) (\eta + \sigma) \,d\eta d\sigma\label{eqn:I-stat-2}\\
    & + F_r(\xi,0,0)(\xi - \eta_r - \sigma_r) \int e^{-it\phi} \Psi_{\leq \ell_0, \leq \ell_0}(\eta,\sigma) \,d\eta d\sigma\label{eqn:I-stat-3}
\end{align}\end{subequations}
For~\eqref{eqn:I-stat-1}, we recall that the $L^2$ bound on $xf$ implies that $\hat{f}$ is $1/2$-H\"older, so
\begin{equation*}
    |F_r(\xi, \eta, \sigma) - F_r(\xi,0,0)| \lesssim M^3\epsilon^3 t^{1/6} (|\eta| + |\sigma|)^{1/2}
\end{equation*}
which gives us the bound 
\begin{equation*}\begin{split}
    |\eqref{eqn:I-stat-1}| \lesssim& M^3\epsilon^3 t^{1/6}\int \Psi_{\leq \ell_0, \leq \ell_0}(\eta,\sigma) (\eta + \sigma + \eta_r + \sigma_r) (|\eta| + |\sigma|)^{1/2}\,d\eta d\sigma\\
    \lesssim& M^3\epsilon^3 t^{-1} \left(t^{1/3} 2^j\right)^{1 -5/2 \gamma}
\end{split}\end{equation*}
Similarly, the $L^\infty$ bound on $\hat{f}$ from~\eqref{eqn:bootstrap-hypotheses} shows that $|F_r(\xi,0,0)| \lesssim M^3\epsilon^3$, so
\begin{equation*}\begin{split}
    \left|\eqref{eqn:I-stat-2}\right| 
    \lesssim& M^3\epsilon^3 t^{-1} \left(t^{1/3} 2^j\right)^{-3\gamma}
\end{split}\end{equation*}
The term~\eqref{eqn:I-stat-3} contains the leading order contribution to~\eqref{eqn:hat-f-hamiltonian-dynamics}.  We will extract this contribution using the method of stationary phase.  By direct calculation, we find
\begin{equation*}
    \phi(\xi,\eta+\eta_r, \sigma+\sigma_r) = \phi_r + Q_r(\eta,\sigma) + O(|\eta|^3 + |\sigma|^3)
\end{equation*}
where $\phi_r = \phi(\xi,\eta_r,\sigma_r)$ and $Q_r$ is the quadratic form associated to the Hessian matrix $\Hess_{\eta,\sigma} \phi(\xi,\eta_r,\sigma_r)$.  
Thus, $\left|e^{-it\phi} - e^{-it(\phi_r + Q_r(\eta,\sigma))}\right| \lesssim t (|\eta|^3 + |\sigma|^3)$, so
\begin{equation*}\begin{split}
    \biggl|\int \Big(e^{-it\phi} - e^{-it(\phi_r + Q_r(\eta,\sigma))}\Big) \Psi_{\leq \ell_0, \leq \ell_0}(\eta,\sigma) \,d\eta d\sigma\biggr|
    \lesssim& 2^{-j}\big( t^{\frac{1}{3}} 2^j \big)^{1 - 5\gamma}
\end{split}\end{equation*}
By rescaling and using stationary phase, we find that
\begin{equation*}\begin{split}
    \int e^{-it Q_r(\eta,\sigma)} \psi_{\leq \ell_0}(\eta) \psi_{\leq \ell_0}(\sigma) \,d\eta d\sigma =& 2^{2\ell_0}\int e^{-it2^{2\ell_0} Q_r(\eta,\sigma)} \psi_{\leq 0}(\eta) \psi_{\leq 0}(\sigma) \,d\eta d\sigma\\
    =& \frac{2\pi e^{i\frac{\pi}{4} \signature \Hess_{\eta,\sigma} \phi(\xi,\eta_r,\sigma_r)}}{t \sqrt{|\det \Hess_{\eta,\sigma} \phi(\xi,\eta_r,\sigma_r)|}}\\
    &+ O\left(t^{-2}2^{-2\ell_0}2^{-2j} \right)
\end{split}\end{equation*}
where on the last line we have used the fact that
 $|\det \Hess_{\eta,\sigma} \phi(\xi,\eta_r,\sigma_r)| \sim 2^{2j}$ to obtain the error term.  Collecting all these calculations, we find that
\begin{equation*}\begin{split}
    \eqref{eqn:I-stat-3} =&  2\pi F_r(\xi,0,0)\frac{(\xi - \eta_r - \sigma_r) e^{-it\phi_r + i\frac{\pi}{4} \signature Q_r}}{t \sqrt{|\det Q_r|}}+  O(M^3\epsilon^3 t^{-1} (t^{1/3} 2^j)^{-1/14})
\end{split}\end{equation*}
A quick calculation shows that for $r=1,2,3$, we have
\begin{align*}
    \phi_r =& 0,& \det Q_r =& -36 \xi^2,& \signature Q_r =& 0\\
    \phi_4 =& \frac{8}{9}\xi^3,&  \det Q_4 =& 12 \xi^2,&\signature Q_r =& -2 \sgn \xi
\end{align*}
Thus, taking $\gamma = 3/7$, we have shown that
\begin{equation*}\begin{split}
     \sum_{2^\ell \ll 2^j}\sum_{r=1}^4 J^{(r)}_{\ell} =&  -\frac{i}{6t}\sgn \xi | \hat{f}(\xi, t)|^2 \hat{f}(\xi, t) - e^{it\frac{8}{9}\xi^3}\frac{\sgn \xi e^{- i\frac{\pi}{2}\sgn\xi}}{3\sqrt{12}t} | \hat{f}(\xi/3, t)|^2 \hat{f}(\xi/3, t) \\
    &\qquad+  O(M^3\epsilon^3 t^{-1} (t^{1/3} 2^j)^{-1/14})
\end{split}\end{equation*}
which concludes the proof of~\eqref{eqn:hat-f-hamiltonian-dynamics}.

\subsection*{Acknowledgements}
The author would like to thank P. Germain for many helpful discussions. The author would also like to thank P. Deift and R. C\^ote for discussions about the self-similar solution.  Finally, the author would like to thank the anonymous reviewer for several suggestions which greatly improved the exposition in Section 4.

\bibliographystyle{plain}
\bibliography{main}

\begin{thebibliography}{10}

\bibitem{ambrosioGeneralChainRule1990}
L.~Ambrosio and G.~Dal~Maso.
\newblock A general chain rule for distributional derivatives.
\newblock {\em Proceedings of the American Mathematical Society},
  108(3):691--702, 1990.

\bibitem{ancoTravelingWavesConservation2012}
Stephen~C. Anco, Mohammad Mohiuddin, and Thomas Wolf.
\newblock Traveling waves and conservation laws for complex {{mKdV}}-type
  equations.
\newblock {\em Appl. Math. Comput.}, 219(2):679--698, October 2012.

\bibitem{benyaminiGeometricNonlinearFunctional2000}
Yoav Benyamini and Joram Lindenstrauss.
\newblock {\em Geometric Nonlinear Functional Analysis}.
\newblock Number v. 48- in American {{Mathematical Society}} Colloquium
  Publications. {American Mathematical Society}, {Providence, R.I}, 2000.

\bibitem{biswasOpticalSolitonsPresence2018}
Anjan Biswas and Saima Arshed.
\newblock Optical solitons in presence of higher order dispersions and absence
  of self-phase modulation.
\newblock {\em Optik}, 174:452--459, December 2018.

\bibitem{chapoutoRefinedWellPosednessResult2021}
Andreia Chapouto.
\newblock A {{Refined Well}}-{{Posedness Result}} for the {{Modified KdV
  Equation}} in the {{Fourier}}\textendash{{Lebesgue Spaces}}.
\newblock {\em J. Dynam. Differential Equations}, July 2021.

\bibitem{chapoutoRemarkWellposednessModified2021}
Andreia Chapouto.
\newblock A remark on the well-posedness of the modified {{KdV}} equation in
  the {{Fourier}}-{{Lebesgue}} spaces.
\newblock {\em Discrete. Contin. Dyn. Syst.}, 41(8):3915, 2021.

\bibitem{christAsymptoticsFrequencyModulation2003}
Francis~Michael Christ, James~E. Colliander, and Terrence Tao.
\newblock Asymptotics, {Fre\-quen\-cy} modulation, and low regularity
  ill-posedness for {Ca\-no\-ni\-cal} defocusing equations.
\newblock {\em Amer. J. Math.}, 125(6):1235--1293, 2003.

\bibitem{coifmanAuDelaOperateurs1978}
Ronald~R Coifman and Yves Meyer.
\newblock {Au dela des op\'erateurs pseudo-diff\'erentiels}.
\newblock {\em Asterisque}, 57:202, 1978.

\bibitem{collianderSharpGlobalWellposedness2003}
J.~Colliander, M.~Keel, G.~Staffilani, H.~Takaoka, and T.~Tao.
\newblock Sharp global well-posedness for {{KdV}} and modified {{KdV}} on
  {{$\mathbb{R}$}} and {$\mathbb{T}$}.
\newblock {\em J. Amer. Math. Soc.}, 16(3):705--749, 2003.

\bibitem{cordobaGlobalSolutionsGeneralized2019}
Diego C{\'o}rdoba, Javier {G{\'o}mez-Serrano}, and Alexandru~D. Ionescu.
\newblock Global {{Solutions}} for the {{Generalized SQG Patch Equation}}.
\newblock {\em Arch. Ration. Mech. Anal.}, 233(3):1211--1251, September 2019.

\bibitem{correiaAsymptoticsFourierSpace2020}
Sim{\~a}o Correia, Rapha{\"e}l C{\^o}te, and Luis Vega.
\newblock Asymptotics in {{Fourier}} space of self-similar solutions to the
  modified {{Korteweg}}-de {{Vries}} equation.
\newblock {\em J. Math. Pures Appl.}, 137:101--142, May 2020.

\bibitem{correiaSelfSimilarDynamicsModified2020}
Sim{\~a}o Correia, Rapha{\"e}l C{\^o}te, and Luis Vega.
\newblock Self-{{Similar Dynamics}} for the {{Modified Korteweg}}\textendash de
  {{Vries Equation}}.
\newblock {\em Int. Math. Res. Not. IMRN}, 2021(rnz383), January 2020.

\bibitem{deiftSteepestDescentMethod1993}
Percy Deift and Xin Zhou.
\newblock A {{Steepest Descent Method}} for {{Oscillatory Riemann}}--{{Hilbert
  Problems}}. {{Asymptotics}} for the {{MKdV Equation}}.
\newblock {\em Ann. Math.}, 137(2):295--368, March 1993.

\bibitem{deiftAsymptoticsPainleveII1995}
Percy Deift and Xin Zhou.
\newblock Asymptotics for the painlev\'e {{II}} equation.
\newblock {\em Comm. Pure Appl. Math.}, 48(3):277--337, 1995.

\bibitem{fonsecaPersistencePropertiesFractional2015}
German Fonseca, Felipe Linares, and Gustavo Ponce.
\newblock On persistence properties in fractional weighted spaces.
\newblock {\em Proc. Amer. Math. Soc.}, 143(12):5353--5367, 2015.

\bibitem{fukumotoThreedimensionalDistortionsVortex1991}
Yasuhide Fukumoto and Takeshi Miyazaki.
\newblock Three-dimensional distortions of a vortex filament with axial
  velocity.
\newblock {\em J. Fluid Mech.}, 222:369--416, January 1991.

\bibitem{germainSpacetimeResonances2010}
Pierre Germain.
\newblock Space-time resonances.
\newblock {\em Journ. \'Equ. d\'eriv. partielles}, pages 1--10, 2010.

\bibitem{germainGlobalExistenceCoupled2011}
Pierre Germain.
\newblock Global existence for coupled {{Klein}}-{{Gordon}} equations with
  different speeds.
\newblock {\em Ann. Inst. Fourier}, 61(6):2463--2506, 2011.

\bibitem{germainGlobalExistenceEulerMaxwell2014}
Pierre Germain and Nader Masmoudi.
\newblock Global existence for the {{Euler}}-{{Maxwell}} system.
\newblock {\em Ann. Sci. \'Ecole Norm. Sup.}, 47(3):469--503, 2014.

\bibitem{germainGlobalSolutions3D2008}
Pierre Germain, Nader Masmoudi, and Jalal Shatah.
\newblock Global {{Solutions}} for {{3D Quadratic Schrodinger Equations}}.
\newblock {\em Int. Math. Res. Not. IMRN}, December 2008.

\bibitem{germainGlobalSolution2D2012}
Pierre Germain, Nader Masmoudi, and Jalal Shatah.
\newblock Global solution for {{2D}} quadratic {{Schr\"odinger}} equations.
\newblock {\em J. Math. Pures Appl.}, 97(5):505--543, May 2012.

\bibitem{germainAsymptoticStabilitySolitons2016}
Pierre Germain, Fabio Pusateri, and Fr{\'e}d{\'e}ric Rousset.
\newblock Asymptotic stability of solitons for {{mKdV}}.
\newblock {\em Adv. Math.}, 299:272--330, August 2016.

\bibitem{grunrockImprovedLocalWellposedness2004}
Axel Gr{\"u}nrock.
\newblock An improved local well-posedness result for the modified {{KdV}}
  equation.
\newblock {\em Int. Math. Res. Not. IMRN}, 2004(61):3287--3308, January 2004.

\bibitem{grunrockLocalWellposednessModified2009}
Axel Gr{\"u}nrock and Luis Vega.
\newblock Local well-posedness for the modified {{KdV}} equation in almost
  critical {$\widehat{H^s_r}$}-spaces.
\newblock {\em Trans. Amer. Math. Soc.}, 361(11):5681--5694, 2009.

\bibitem{guoGlobalWellposednessKorteweg2009}
Zihua Guo.
\newblock Global well-posedness of {{Korteweg}}\textendash de {{Vries}}
  equation in $h^{-3/4}(\mathbb{R})$.
\newblock {\em J. Math. Pures Appl.}, 91(6):583--597, June 2009.

\bibitem{gustafsonScatteringTheoryGross2009}
Stephen Gustafson, Kenji Nakanishi, and Tai-Peng Tsai.
\newblock Scattering theory for the gross\textendash pitaevskii equation in
  three dimensions.
\newblock {\em Commun. Contemp. Math.}, 11(04):657--707, August 2009.

\bibitem{haniScatteringZakharovSystem2013}
Zaher Hani, Fabio Pusateri, and Jalal Shatah.
\newblock Scattering for the {{Zakharov System}} in 3 {{Dimensions}}.
\newblock {\em Commun. Math. Phys.}, 322(3):731--753, September 2013.

\bibitem{hardyInequalities1959}
Godfrey~H. Hardy, John~E. Littlewood, and George P{\'o}lya.
\newblock {\em Inequalities}.
\newblock {Cambridge University Press}, {Cambridge}, 2nd ed. edition, 1959.

\bibitem{harrop-griffithsLongTimeBehavior2016}
Benjamin {Harrop-Griffiths}.
\newblock Long time behavior of solutions to the {{mKdV}}.
\newblock {\em Comm. Partial Differential Equations}, 41(2):282--317, February
  2016.

\bibitem{harrop-griffithsSharpWellposednessCubic2020}
Benjamin {Harrop-Griffiths}, Rowan Killip, and Monica Visan.
\newblock Sharp well-posedness for the cubic {{NLS}} and {{mKdV}} in
  {$H^s(\mathbb{R})$}.
\newblock {\em arXiv}, March 2020.

\bibitem{hasimotoSolitonVortexFilament1972}
Hidenori Hasimoto.
\newblock A soliton on a vortex filament.
\newblock {\em Journal of Fluid Mechanics}, 51(3):477--485, February 1972.

\bibitem{hastingsBoundaryValueProblem1980}
S.P. Hastings and J.B. McLeod.
\newblock A boundary value problem associated with the second painlev\'e
  transcendent and the {{Korteweg}}-de {{Vries}} equation.
\newblock {\em Arch. Ration. Mech. Anal.}, 73(1):31--51, 1980.

\bibitem{hayashiModifiedKortewegVries2001}
Nakao Hayashi and Pavel Naumkin.
\newblock On the {{Modified Korteweg}}\textendash{{De Vries Equation}}.
\newblock {\em Math. Phys. Anal. Geom.}, 4(3):197--227, September 2001.

\bibitem{hayashiLargeTimeBehavior1999}
Nakao Hayashi and Pavel~I. Naumkin.
\newblock Large time behavior of solutions for the modified {{Korteweg}}-de
  {{Vries}} equation.
\newblock {\em Int. Math. Res. Not. IMRN}, 1999(8):395--418, January 1999.

\bibitem{heFewcycleOpticalRogue2014}
Jingsong He, Lihong Wang, Linjing Li, K.~Porsezian, and R.~Erd{\'e}lyi.
\newblock Few-cycle optical rogue waves: {{Complex}} modified {{Korteweg}}--de
  {{Vries}} equation.
\newblock {\em Phys. Rev. E}, 89(6):062917, June 2014.

\bibitem{huang_asymptotics_2020}
Lin Huang and Jonatan Lenells.
\newblock Asymptotics for the {{Sasa}}\textendash{{Satsuma}} equation in terms
  of a modified painlevé {II} transcendent.
\newblock {\em J. Differ. Equations}, 268(12):7480--7504, 2020.

\bibitem{huangLongtimeAsymptoticHirota2015}
Lin Huang, Jian Xu, and En-gui Fan.
\newblock Long-time asymptotic for the {{Hirota}} equation via nonlinear
  steepest descent method.
\newblock {\em Nonlinear Anal. Real World Appl.}, 26:229--262, December 2015.

\bibitem{ifrimGlobalBoundsCubic2015}
Mihaela Ifrim and Daniel Tataru.
\newblock Global bounds for the cubic nonlinear {{Schr\"odinger}} equation
  ({{NLS}}) in one space dimension.
\newblock {\em Nonlinearity}, 28(8):2661--2675, August 2015.

\bibitem{ifrimTwoDimensionalWater2016}
Mihaela Ifrim and Daniel Tataru.
\newblock Two dimensional water waves in holomorphic coordinates {{II}}: Global
  solutions.
\newblock {\em Bull. Soc. Math. France}, 144(2):369--394, 2016.

\bibitem{ionescuEinsteinKleinGordonCoupledSystem2020}
Alexandru~D. Ionescu and Benoit Pausader.
\newblock The {{Einstein}}-{{Klein}}-{{Gordon}} coupled system: Global
  stability of the {{Minkowski}} solution.
\newblock {\em arXiv}, February 2020.

\bibitem{ionescuGlobalRegularity2d2018}
Alexandru~D. Ionescu and Fabio Pusateri.
\newblock Global {{Regularity}} for 2d {{Water Waves}} with {{Surface
  Tension}}.
\newblock {\em Mem. Amer. Math. Soc.}, 256(1227), November 2018.

\bibitem{katoNewProofLongrange2011}
Jun Kato and Fabio Pusateri.
\newblock A new proof of long-range scattering for critical nonlinear
  {{Schr\"odinger}} equations.
\newblock {\em Differ. Integr. Equ.}, 24(9-10):923--940, 2011.

\bibitem{katoCauchyProblemGeneralized1983}
Tosio Kato.
\newblock On the {{Cauchy}} problem for the (generalized) {{Korteweg}}-de
  {{Vries}} equation.
\newblock In {\em Studies in Applied Mathematics}, volume~8 of {\em Advances in
  Mathematics}, pages 93--128. {Academic Press}, {New York, NY}, 1983.

\bibitem{kenigGeneralizedKortewegdeVries1989}
Carlos~E. Kenig, Gustavo Ponce, and Luis Vega.
\newblock On the (generalized) {{Korteweg}}-de {{Vries}} equation.
\newblock {\em Duke Math. J.}, 59(3):585--610, December 1989.

\bibitem{kenigWellposednessScatteringResults1993}
Carlos~E. Kenig, Gustavo Ponce, and Luis Vega.
\newblock Well-posedness and scattering results for the generalized korteweg-de
  vries equation via the contraction principle.
\newblock {\em Comm. Pure and Appl. Math.}, 46(4):527--620, 1993.

\bibitem{kenigIllposednessCanonicalDispersive2001}
Carlos~E. Kenig, Gustavo Ponce, and Luis Vega.
\newblock On the ill-posedness of some canonical dispersive equations.
\newblock {\em Duke Math. J.}, 106(3):617--633, February 2001.

\bibitem{kishimotoWellposednessCauchyProblem2009}
Nobu Kishimoto.
\newblock Well-posedness of the {{Cauchy}} problem for the {{Korteweg}}-de
  {{Vries}} equation at the critical regularity.
\newblock {\em Differential Integral Equations}, 22(5/6):447--464, May 2009.

\bibitem{klainermanNullConditionGlobal1986}
Sergiu Klainerman.
\newblock The null condition and global existence to nonlinear wave equations.
\newblock In {\em Nonlinear Systems of Partial Differential Equations in
  Applied Mathematics}, volume~1 of {\em Lectures in Applied Mathematics},
  pages 293--326. {American Mathematical Society}, {Providence, R.I}, 1986.

\bibitem{lambSolitonsMovingSpace1977}
G.~L. Lamb.
\newblock Solitons on moving space curves.
\newblock {\em Journal of Mathematical Physics}, 18(8):1654--1661, August 1977.

\bibitem{liuLongtimeAsymptoticsSasa2019}
Nan Liu and Boling Guo.
\newblock Long-time asymptotics for the {{Sasa}}\textendash{{Satsuma}} equation
  via nonlinear steepest descent method.
\newblock {\em J. Math. Phys.}, 60(1):011504, January 2019.

\bibitem{liu_long-time_2019}
Nan Liu and Boling Guo.
\newblock Long-time asymptotics for the sasa–satsuma equation via nonlinear
  steepest descent method.
\newblock {\em J. Math. Phys.}, 60(1):011504, 2019.

\bibitem{perelmanSelfsimilarPlanarCurves2007}
G.~Perelman and L.~Vega.
\newblock Self-similar planar curves related to modified
  {{Korteweg}}\textendash de {{Vries}} equation.
\newblock {\em J. Differential Equations}, 235(1):56--73, April 2007.

\bibitem{rodriguezStandardEmbeddedSolitons2003}
R.~F. Rodr{\'i}guez, J.~A. Reyes, A.~{Espinosa-Cer{\'o}n}, J.~Fujioka, and
  B.~A. Malomed.
\newblock Standard and embedded solitons in nematic optical fibers.
\newblock {\em Phys. Rev. E}, 68(3):036606, September 2003.

\bibitem{scorerNumericalEvaluationIntegrals1950}
R.~S. Scorer.
\newblock Numerical evaluation of integrals of the form {$I = \int_{x_1}^{x_2}
  f(x) e^{i\phi(x)} dx$} and the tabulation of the function
  {$\operatorname{Gi}(z) = ( 1/\pi) \int_0^\infty \sin(uz +
  \frac{1}{3}u^3)du$}.
\newblock {\em Q J Mech Appl Math}, 3(1):107--112, January 1950.

\bibitem{shatahNormalFormsQuadratic1985}
Jalal Shatah.
\newblock Normal forms and quadratic nonlinear {{Klein}}-{{Gordon}} equations.
\newblock {\em Comm. Pure Appl. Math.}, 38(5):685--696, 1985.

\bibitem{stewartAblowitzLadik}
Gavin Stewart.
\newblock Asymptotics for small data solutions of the {Ablowitz}-{Ladik}
  equation.
\newblock In preparation.

\bibitem{taoNonlinearDispersiveEquations2006}
Terence Tao.
\newblock {\em Nonlinear Dispersive Equations: Local and Global Analysis}.
\newblock Number no. 106 in Conference {{Board}} of the {{Mathematical
  Sciences}} Regional Conference Series in Mathematics. {American Mathematical
  Society}, {Providence, R.I}, 2006.

\bibitem{zhangSpectralAnalysisLongtime2022}
Hong-Yi Zhang and Yu-Feng Zhang.
\newblock Spectral analysis and long-time asymptotics of complex {{mKdV}}
  equation.
\newblock {\em J. Math. Phys.}, 63(2):021509, February 2022.

\bibitem{zworskiSemiclassicalAnalysis2012}
Maciej Zworski.
\newblock {\em Semiclassical {{Analysis}}}, volume 138 of {\em Graduate
  {{Studies}} in {{Mathematics}}}.
\newblock {American Mathematical Society}, July 2012.

\end{thebibliography}

\end{document}